\documentclass[noinfoline]{imsart}

\newlength\Colsep
\RequirePackage[OT1]{fontenc}
\RequirePackage{amsthm,amsmath}
\RequirePackage{natbib}

\allowdisplaybreaks

\usepackage{fullpage}

\usepackage{layout}

\usepackage{dsfont}
\usepackage[mathscr]{eucal}
\usepackage[toc,page]{appendix}
\usepackage{mathrsfs}
\usepackage{color}
\usepackage{pifont}
\usepackage{bm}
\usepackage{amsfonts}
\usepackage{amssymb}
\usepackage{epsfig}
\usepackage{graphicx}
\usepackage{multirow}

\startlocaldefs
\numberwithin{equation}{section}
\theoremstyle{plain}
\endlocaldefs

\DeclareFontEncoding{FMS}{}{}
\DeclareFontSubstitution{FMS}{futm}{m}{n}
\DeclareFontEncoding{FMX}{}{}
\DeclareFontSubstitution{FMX}{futm}{m}{n}
\DeclareSymbolFont{fouriersymbols}{FMS}{futm}{m}{n}
\DeclareSymbolFont{fourierlargesymbols}{FMX}{futm}{m}{n}
\DeclareMathDelimiter{\hsnorm}{\mathord}{fouriersymbols}{152}{fourierlargesymbols}{147}

\newcommand{\hs}{\big\hsnorm}

\usepackage{fullpage}

\RequirePackage[OT1]{fontenc}
\RequirePackage{amsthm,amsmath}

\usepackage{verbatim}

\allowdisplaybreaks

\newcommand {\R}{\mathbb{R}}

\newcommand {\F}{\mathscr{F}}

\newcommand {\B}{\mathscr{B}}
\newcommand{\Q}{\mathscr{Q}}
\newcommand{\scrS}{\mathscr{S}}

\newcommand{\G}{\mathscr{G}}

\newcommand{\V}{\mathscr{V}}
\newcommand{\U}{\mathscr{U}}

\newcommand{\cum}{\rm cum}

\usepackage{subfigure}
\usepackage{layout}

\newtheorem{theorem}{Theorem}
\newtheorem{proposition}{Proposition}

\newtheorem{lemma}{Lemma}

\newtheorem{remark}{Remark}

\newtheorem{assumptions}{Assumptions}

\input xy
\xyoption{all}

\usepackage{dsfont}
\usepackage[mathscr]{eucal}
\usepackage[toc,page]{appendix}
\usepackage{mathrsfs}
\usepackage{color}
\usepackage{pifont}
\usepackage{bm}
\usepackage{latexsym}
\usepackage{amsmath}
\usepackage{amsthm}
\usepackage{hyperref}

\startlocaldefs
\numberwithin{equation}{section}
\theoremstyle{plain}

\endlocaldefs

\begin{document}

\begin{frontmatter}
\title{\large{ Methodology and Convergence Rates for Functional Time Series Regression\thanksref{T1}}}

\runtitle{Functional Time Series Regression}
\thankstext{T1}{Research supported by an ERC Starting Grant Award to Victor M. Panaretos.}

\begin{aug}

\noindent\begin{minipage}{0.7\textwidth}
\begin{minipage}[c][2.5cm][c]{\dimexpr0.5\textwidth-0.5\Colsep\relax}

\author{\fnms{Tung} \snm{Pham}\thanksref{t1}\ead[label=e1]{pham.t@unimelb.edu.au}}

\address{School of Mathematics and Statistics\\
Universiry of Melbourne\\
\printead{e1}}

\end{minipage}\hfill
\begin{minipage}[c][2.5cm][c]{\dimexpr0.5\textwidth-0.5\Colsep\relax}
\author{\fnms{Victor M.} \snm{Panaretos}
\ead[label=e2]{victor.panaretos@epfl.ch}
}

\address{Institut de Math\'ematiques\\
Ecole Polytechnique F\'ed\'erale de Lausanne\\
\printead{e2}\\
}

\end{minipage}%
\end{minipage}

\thankstext{t1}{Part of this work was carried out while the first author was with the Institute of Mathematics, Ecole Polytechnique F\'ed\'erale de Lausanne.}

\runauthor{T. Pham \& V.~M. Panaretos}

\end{aug}

\begin{abstract}
The functional linear model extends the notion of linear regression to the case where the response and covariates are iid elements of an infinite dimensional Hilbert space. The unknown to be estimated is a Hilbert-Schmidt operator, whose inverse is by definition unbounded, rendering the problem of inference ill-posed. In this paper, we consider the more general context where the sample of response/covariate pairs forms a weakly dependent stationary process in the respective product Hilbert space: simply stated, the case where we have a regression between functional time series. We consider a general framework of potentially nonlinear processes, expoiting recent advances in the spectral analysis of time series. This allows us to quantify the inherent ill-posedness, and motivate a Tikhonov regularisation technique in the frequency domain. Our main result is the establishment of the rate of convergence for the corresponding estimators of the regression coefficients, the latter forming a summable sequence in the space of Hilbert-Schmidt operators. In a sense, our main result can be seen as a generalisation of the classical functional linear model rates, to the case of time series, and rests only upon Brillinger-type mixing conditions. It is seen that, just as the covariance operator eigenstructure plays a central role in the independent case, so does the spectral density operator's eigenstructure in the dependent case. While the analysis becomes considerably more involved in the dependent case, the rates are strikingly comparable to those of the i.i.d. case, but at the expense of an additional factor caused by the necessity to estimate the spectral density operator at a nonparametric rate, as opposed to the parametric rate for covariance operator estimation. 
\end{abstract}



\end{frontmatter}

\tableofcontents

\newpage

\section{Introduction}

Functional regression generalises the classical linear model of multivariate statistics to the case where the parameter, the response, and the error components reside in general separable Hilbert spaces, while the design matrix is replaced by a linear operator between these spaces (\citet{grenander1981}). The most studied case is that where the covariate lies in space $(L^2[0,1],\langle\cdot,\cdot\rangle,\|\cdot\|)$ of square integrable real functions on the unit interval (\citet{HK12}, \cite{HE15}, \citet{RS05}). Here one has independent random elements $X,\epsilon\in L^2[0,1]$, and a bounded linear operator $\mathscr{B}:L^2[0,1]\rightarrow \mathscr{H}$ mapping into a separable Hilbert space $\mathscr{H}$, yielding the regression model
$$Y=\mathscr{B}X+\varepsilon.$$
The random elements $X$ and $Y$ are assumed observable, but $\varepsilon$ is unobservable and $\mathscr{B}$ is unknown and is to be estimated from i.i.d. replicates $\{(X_n,Y_n)\}_{n\in\mathbb{N}}$ of $(X,Y)$. The most studied case is the so called scalar-on-function regression, where $\mathscr{H}=\mathbb{R}$ and so $\mathscr{B}$ reduces to a bounded linear functional $\mathscr{B}f=\langle f,\beta\rangle$, and the function $\beta$ is the parameter of interest. More general is the case where $\mathscr{H}=L^2[0,1]$, and the operator $\mathscr{B}$ is an integral operator with kernel $\beta\in L^2([0,1]^2)$,
$$\mathscr{B}f=\int_0^1 \beta(\sigma,\tau) f(\tau)d\tau,\qquad\forall\,f\in L^2[0,1].$$  
In either of these cases, writing down the normal equations reveals that one is confronted with an ill-posed inverse problem: the equations involve the application of the inverse of the trace-class covariance operator $\mathscr{R}$ of the random element $X$. Worse still, the operator $\mathscr{R}$ is unknown, and needs to be replaced by its empirical version. Consequently, the statistical methodology for functional regression must involve some means of regularisation, the most popular being PCA regression (or spectral truncation), where one replaces the empirical covariance $\hat{\mathscr{R}}$ by its best rank $K$ approximation in nuclear norm, for some regularisation parameter $K$ (that is of course allowed to grow with $n$; see, e.g.  \citet[Chapter 10]{RS05}; \citet{FerratyReg}; \citet{CFF02}; \citet{HP06}).

In a landmark contribution on the functional linear model, \citet{HH07}  demonstrated that while the PCA estimator can achieve minimax rates (in probability) in some cases, the ridge estimator (corresponding to Tikhonov regularisation, and adding a multiple of the identity to the empirical covariance) can have important advantages. Theoretically, the Tikhonov estimator can achieve the minimax mean square error (MSE) rate, whereas the truncated PCA estimators would need to undergo a nonlinear modification to achieve similar MSE rates (see, e.g. \citet[Theorem 5, Appendix A.2]{HHN06}, and the remarks following \citet[Theorem 1]{HH07})). Practically, \citet{HH07} showed that the Tikhonov estimator enjoyed better stability properties and was robust to eigenvalue ties. The results of \citet{HH07} apply to the scalar-on function case, and extensions thereof have recently been considered in the function-on-function case (\citet{imaizumi2016pca}).   

In this paper, we attack the problem of extending the Tikhonov-based methodology and rates of convergence of \citet{HH07} to the case of the function-on-function regression of time series (which can also be seen as a functional linear system identification problem). Here, the observed covariates $\{X_t\}_{t\in Z}$ and unobservable errors $\{\epsilon_t\}_{t=1}^{T}$ are no longer i.i.d., but constitute stationary processes in $L^2[0,1]$. The resulting response process $\{Y_t\}_{t\in \mathbb{Z}}$ is then also a stationary process, linearly coupled to the $X_t$ and $\epsilon_t$ via a sequence of operators $\{\mathscr{B}_t\}_{t\in\mathbb{Z}}$,
$$Y_t=\sum_s\mathscr{B}_{t-s}X_s+\epsilon_t,\qquad t\in\mathbb{Z}.$$
Of interest is the estimation of the operators (or \emph{filter}) $\{\mathscr{B}_t\}$, on the basis of the observation of a finite stretch of pairs $\{(X_t,Y_t)\}_{t=0}^{T-1}$. This case is considerably more challenging than the i.i.d. function-on-function case. The reason is that further to the intrinsic covariation of each regressor function $X_t$, encapsulated in the covariance $\mathscr{R}$, one needs to account for the temporal covariation between lagged regressor functions $X_t$ and $X_{t+s}$. These too contribute to the ill-posedness of the problem, which is now doubly ill-posed: one needs to solve an operator deconvolution problem, where the ``Fourier division" step is replaced with the solution of an integral equation. To account for these two layers of ill-posedness, one needs to consider the frequency domain framework (\citet{panaretos2013fourier}, \citet{panaretos2013cramer}), and it turns out that the operator that needs to be inverted as part of the normal equations is now the spectral density operator of the process $\{X_t\}$, 
 $$\mathscr{F}_{\omega}^{XX}=\frac{1}{2\pi}\sum_t e^{-\mathbf{i}t\omega} \mathscr{R}^{X}_t,$$
 the Fourier transform of the lag $t$ autocovariance operators $\mathscr{R}_t$ of $\{X_t\}$.

Just as estimation in the i.i.d case is based on the spectral truncation or the ridge regularisation of the covariance operator, estimation in the time series case can be based on the spectral truncation or ridge regularisation of the spectral density operator (achieved by \emph{harmonic} or \emph{dynamic} PCA, see \citet{panaretos2013cramer} and \citet{hormann2015dynamic}). The spectral truncation approach was recently considered and studied by \citet{hormann2015estimation}, and indeed this appears to be the first contribution to the theory of time series regression without any structural assumptions further to weak dependence (to be contrasted to the functional regression of \emph{linear processes}, which are much better understood, see \citet{bosq2012linear}). \citet{hormann2015estimation} show that by truncating the spectral density operator at a certain rate, one can obtain consistent estimators of the operators $\{\mathscr{B}_t\}$, under weak dependence conditions. An elegant aspect of their approach is that the ``correct" truncation rate can in principle be deduced from the data. Still, convergence rates have to date not been established.

Inspired by the work of \citet{HH07}, we set forth to establish such convergence rates. In view of the technical difficulties of PCA regression in the i.i.d. case, it seems unlikely that MSE error rates would be attainable for the truncated harmonic PCA estimator without some nonlinear modification -- after all, the i.i.d. setup is a special case of the time series setup, and so any difficulties encountered in the former will apply to the latter, too. This motivates us to introduce a different regularisation method than that of \citet{hormann2015estimation}, adopting the Tikhonov perspective. In this framework, we establish the rate of convergence under Brillinger type weak dependence conditions \citep{brillinger2001time}, and mild ill-posedness assumptions formulated in direct analogy to the assumptions of \citet{HH07} (and of \citet{imaizumi2016pca}). The convergence rate turns out to be the same as in the i.i.d. case, except for the presence of a bandwidth factor that results from the fact that one needs to estimate the spectral density operator by smoothing the periodogram operator: unless one \emph{knows} the processes to actually be uncorrelated in $t\in\mathbb{Z}$, this is a term that cannot be escaped. 

The paper is organised as follows. Section \ref{notation} establishes notational conventions and analytic notions employed throughout the paper. Section \ref{time_series_background} then briefly reviews the framework of functional time series, including the key objects of frequency domain functional time series used in the sequel. Functional time series regression and its diagonalisation are considered in Sections \ref{model_definition} and \ref{diagonalisation}. This motivates the methodological contribution of the paper, which is the Fourier-Tikhonov estimator, presented in Section \ref{estimation} and discussed in detail in comparison to PCA-based methodology. Our central theoretical result is given in the form of Theorem \ref{main_theorem} in Section \ref{rate}, and is the MSE rate of convergence of the Fourier-Tikhonov estimator. The proof of the main result is quite involved, and is thus developed via a sequence of intermediate results in the separate Section \ref{proof_section}, with accessory steps proven separately in the Appendix \ref{appendix}.

\section{Basic Definitions and Notation}\label{notation}

We will be working in the usual context of functional data analysis, which assumes that each datum arises as the realisation of a random elemmaent of the separable Hilbert space $L^2([0,1])$ of square integrable real functions on $[0,1]$. The latter is equipped with the standard inner product and norm
$$\langle f,g\rangle=\int_0^1 f(\tau)g(\tau)d\tau,\qquad \|f\|^2=\int_0^1 f^2(\tau)d\tau=\langle f,f\rangle.$$
Given a linear operator $\mathscr{B}:L^2([0,1])\rightarrow L^2([0,1])$, we will denote its adjoint by $\mathscr{B}^*$, its generalised inverse by $\mathscr{B}^{\dagger}$, and its inverse by $\mathscr{B}^{-1}$, if well defined. The Schatten-$\infty$ norm (operator norm), Schatten-2 norm (Hilbert-Schmidt norm), and Schatten-1 norm (nuclear norm) will be respectively denoted by
$$\hs \mathscr{B}\hs_{\infty}=\sup_{\|h\|=1}\|\mathscr{B}h\|,\quad \hs \mathscr{B} \hs_2 =\sqrt{\mbox{trace}\left(\mathscr{B}^*\mathscr{B}\right)},\quad  \hs \mathscr{B} \hs_1=\mbox{trace}\left(\sqrt{\mathscr{B}^*\mathscr{B}}\right).$$ 
Occasionally, we will abuse notation, and apply a Schatten norm to the kernel of the corresponding integral operator, in which case it should be understood that the norm applies to the induced operator. For example, if $f\in L^2([0,1]^2)$, we may write $\hs f\hs_1$ to denote the Schatten-1 norm of the operator $g\mapsto \int_0^1 f(s,t)g(t)dt$.

The identity operator will be denoted by $\mathscr{I}$. For a pair of elements $f,g\in L^2[0,1]$, we define the tensor product (operator) as $f\otimes g:L^2[0,1]\to L^2[0,1]$ 
$$(f\otimes g)u=\langle g,u\rangle f,\qquad u\in \mathcal{H}.$$
We will make use of the same notation for tensor products of operators, i.e. if $\mathscr{A}$, $\mathscr{B}$, and $\mathscr{G}$ are operators $L^2[0,1]\rightarrow L^2[0,1]$, we write
$$(\mathscr{A}\otimes \mathscr{B})\,\mathscr{G}=\mbox{trace}\left(\mathscr{B}^*\mathscr{G}\right)\mathscr{A}.$$

Finally, we will use $Conv_C\big(L^2([0,1]^2,\mathbb{C}) \times  L^2([0,1]^2,\mathbb{C})\big)$ to denote the set of finite convex combinations of elements of the form $f \times g$ with $f,g \in L^2([0,1]^2,\mathbb{C})$ whose induced operators have Schattern 1-norm uniformly bounded by a  constant $C$. A generic element of $Conv_C\big(L^2([0,1]^2,\mathbb{C}) \times  L^2([0,1]^2,\mathbb{C})\big)$ will be denoted by $\vartheta_1\odot\vartheta_2$, which is understood as implying that this element can be written in the form
$$\vartheta_1(a_1,a_2)\odot \vartheta_2(a_3,a_4)=\sum_{j=1}^J \pi_j f_j(a_1,a_2)h_j(a_3,a_4),$$
for a probability measure $\{\pi_j\}_{j=1}^{J}$ and functions $f_j,h_j\in L^2([0,1]^2,\mathbb{C})$ such that the Schatten-1 norms $\{\hs f_j\hs_1,\hs h_j\hs _1\}_{j=1}^{J}$ of the operators $L^2([0,1],\mathbb{C})\rightarrow L^2([0,1],\mathbb{C})$ with kernels $f_j$ and $h_j$ are all bounded by $C$. This notation will be used frequently to abbreviate cumbersome terms in Taylor expansions involving linear combinations of products of kernels.

\section{Functional Time Series Background}\label{time_series_background}

A \emph{functional time series} is a sequence of random elements $\{X_t\}$ of $L^2[0,1]$, indexed indexed by $t\in \mathbb{Z}$ (interpreted as time). The argument of each function $X_t$ is denoted by $\tau\in[0,1]$,
$$
X_t(\tau):[0,1] \rightarrow  \mathbb{R},\qquad \mbox{ for }t\,\in \mathbb{Z}.\\
$$
We will consider only \emph{strictly stationary} time series: given any finite index set $I\subset \mathbb{Z}$, and any $s\in \mathbb{Z}$, it holds that
$$\{X_t\}_{t\in I}\stackrel{d}{=}\{X_{t+s}\}_{t\in I}.$$
The mean function and lag $t$ covariance kernel of $\{X_t\}$ are given by,
$$\mu^X(\tau) = \mathbb{E}[X_{t}(\tau)], \quad r_{t}^{X}(\tau, \sigma) = \mathbb{E}\left\{ \left( X_{t+s}(\tau) - \mu^X(\tau) \right) \left( X_{s}(\sigma) - \mu^X(\sigma) \right) \right\},\, t,s \in \mathbb{Z},$$ 
and are well-defined for almost all $\tau\in [0,1]$ and $(\tau,\sigma)\in[0,1]^2$, respectively, when $\mathbb{E}\| X_0\|^2<\infty$. The lag $t$ covariance operator  $\mathscr{R}^{X}_t:L^2[0,1]\rightarrow L^2[0,1]$ is then defined by the action
$$\mathscr{R}^{X}_{t}h = \mathbb{E}\left[\left( X_{t+s} - \mu^X\right) \otimes \left( X_{s} - \mu^X \right)\right] =\mbox{cov} \left[ \langle X_0,h \rangle, X_t \right],\qquad h \in L^2[0,1],$$
and is a nuclear integral operator with integral kernel $r^X_t$. Assuming that the sequence $\mathscr{R}^X_t$ is nuclear-summable,
$$\sum_t\hs \mathscr{R}^{X}_t \hs_{1}<\infty,$$
 we may define the spectral density operator $\mathscr{F}^{X}_{\omega}$ at frequency $\omega\in[-\pi,\pi]$ as 
 $$\mathscr{F}_{\omega}^{XX}=\frac{1}{2\pi}\sum_t e^{-\mathbf{i}t\omega} \mathscr{R}^{X}_t.$$
 where $\mathbf{i}^2=-1$. This is a nuclear self-adjoint operator with integral kernel
$$f^{XX}_{\omega}(\tau,\sigma)=\frac{1}{2\pi}\sum_t e^{-\mathbf{i}t\omega} r^{X}_t(\tau,\sigma).$$

Given a second functional time series $\{Y_t\}$ satisfying the same (corresponding) assumptions, we may define the lag $t$ cross-covariance kernel as
$$r^{YX}_t(\tau,\sigma)= \mathbb{E}\left\{ \left( X_{t+s}(\tau) - \mu^X(\tau) \right) \left( Y_{s}(\sigma) - \mu^Y(\sigma) \right) \right\},\, \tau, \sigma \in [0,1] \mbox{ \& } t,s \in \mathbb{Z},$$
which in turn induces the lag $t$ cross-covariance operator $\mathscr{R}^{YX}_t:L^2[0,1]\rightarrow L^2[0,1]$ by
$$\mathscr{R}^{YX}_{t}h =\mathbb{E}\left[ \left( X_{t+s} - \mu^X \right)\otimes \left( Y_{s} - \mu^Y \right)\right]h=\mbox{cov} \left[ \langle Y_0,h \rangle, X_t \right],\qquad h \in L^2[0,1].$$
The cross-spectral density operator  $\mathscr{F}^{YX}_{\omega}$ at frequency $\omega\in[-\pi,\pi]$ is then defined as 
 $$\mathscr{F}_{\omega}^{YX}=\frac{1}{2\pi}\sum_t e^{-\mathbf{i}t\omega} \mathscr{R}^{YX}_t$$
with associated integral kernel
$$f^{YX}_{\omega}(\tau,\sigma)=\frac{1}{2\pi}\sum_t e^{-\mathbf{i}t\omega} r^{YX}_t(\tau,\sigma).$$ 

Finally, we will consider \emph{cumulant kernels} (and corresponding operators) as a means of quantifying the strength of temporal dependence in $\{X_{t}\}$ via Brillinger mixing conditions. Given any $(\tau_1,\ldots\tau_k)\in[0,1]^k$, we defined the order-$k$ cumulant kernel of $\{X_t\}$ as
$$ \mbox{cum}\left\{X_{t_{1}}(\tau_{1}), \ldots, X_{t_{k}}(\tau_{k})\right\} =  \sum_{\nu = (\nu_{1}, \ldots, \nu_{p})} (-1)^{p-1} (p-1)!  \prod_{l=1}^{p} \mathbb{E}\left\{ \prod_{j \in \nu_{l}} X_{t_{j}}(\tau_{j}) \right\},$$
with summation being over unordered partitions $\nu = (\nu_{1}, \ldots, \nu_{p})$ of $\{1, \ldots, k\}$. The kernel exists almost everywhere on $[0,1]^k$ provided that 
$\mathbb{E} \|X_{0}\|^{k} < \infty$.
A cumulant kernel of of order $2k$ gives rise to a corresponding \emph{$2k$-th order cumulant operator} $\mathscr{R}_{t_1,...,t_{2k-1}}: L^2([0,1]^k,\mathbb{R})\rightarrow L^2([0,1]^k,\mathbb{R})$, defined by right integration. More generally, we remark that any $g\in L^2([0,1]^{2k}, \R)$, induces a corresponding operator $\mathscr{G}$ on $L^2([0,1]^{k}$, which is defined as 
\begin{eqnarray*}
\mathscr{G} h(\tau_1,\ldots,\tau_k) = \int_{[0,1]^k} g(\tau_1,\ldots,\tau_{2k}) \times h(\tau_1,\ldots,\tau_k) d\tau_1 \ldots d\tau_k,
\end{eqnarray*}
provided the integral is well-defined.

\section{Functional Time Series Regression}\label{model_definition}

In the context of a functional time series regression, we will consider a collection of \emph{covariates} $\{X_t\}$ and associated \emph{responses} $\{Y_t\}$,  each comprising a strictly stationary time series of random elements in $L^2[0,1]$. A functional linear model for the pair $(X_t,Y_t)$ stipulates that the two time series are defined on the same probability space and are \emph{approximately linearly coupled}. That is, there exists a sequence of Hilbert-Schmidt operators $\{\mathscr{B}_t\}$ with integral kernels $\{b_t\}$,
$$\mathscr{B}_t: L^2\rightarrow L^2,\quad b_t(\sigma,\tau):[0,1]^2\rightarrow \mathbb{R},\quad (\mathscr{B}_tf)(\tau)=\int_0^1b_t(\sigma,\tau)f(\sigma)d\sigma, \quad f\in L^2,$$
and a collection of centred i.i.d. \emph{perturbations} in $L^2$, $\{\epsilon_t\}_{t\in\mathbb{Z}}$, such that
\begin{equation}\label{model}
Y_t=\sum_s\mathscr{B}_{t-s}X_s+\epsilon_t,\qquad t\in\mathbb{Z}.
\end{equation}
Notice that the temporal convolution is the only possible linear coupling, if both $X_t$ and $Y_t$ are to be stationary. We make the following natural assumptions:

\begin{assumptions}[Moment and Dependence Assumptions]\label{assumptions}
In the context of model \ref{model}, we assume:
\begin{itemize}
\item[(A1)] The \emph{filter} $\{\mathscr{B}_t\}$ is \emph{Hilbert-Schmidt} summable,
$$\sum_t \hs \mathscr{B}_t\hs_2=\sum_t\left(\int_0^1 \int_0^1 \big| b_t(\tau,\sigma)\big|^2 d\sigma d \tau\right)^{1/2}<\infty.$$

\item[(A2)] The i.i.d. perturbation process $\{\epsilon_t\}$ is independent of the covariate process $\{X_t\}$, and 
$$
\mathbb{E}\|X_t\|^2+\mathbb{E}\|\epsilon_t\|^2<\infty,\qquad \mathbb{E}[X_t]=\mathbb{E}[\epsilon_t]=0.
$$

\item[(A3)] The covariance operators $\{\mathscr{R}^X_t\}_{t\in\mathbb{Z}}$ are \emph{nuclear} summable,
$$\sum_t \hs \mathscr{R}^X_t\hs_1<\infty.$$
\end{itemize}
\end{assumptions}

Whenever Assumptions (A1)-(A3) are satisfied, it holds that $\{Y_t\}$ is also second order (\citet{bosq2012linear}), and possesses nuclear-summable covariance operators $\{\mathscr{R}^Y_t\}_{t\in\mathbb{Z}}$,
$$\mathbb{E}\|Y_t\|^2<\infty\qquad\&\qquad\sum_t \hs \mathscr{R}^Y_t\hs_1<\infty.$$ 
The statistical task at hand is to estimate the unknown sequence of operators (or \emph{filter}) $\{\mathscr{B}_t\}_{t\in\mathbb{Z}}$ on the basis of the observation of a finite stretch of $\{(X_t,Y_t);t=0,\hdots,T-1\}$. As usual, the $\epsilon_t$ are unobservable.

\section{Diagonalising the Problem by Fourier Transformation}\label{diagonalisation}

As with iid functional regression, the key to constructing estimators will be to establish a connection between the cross-covariance of $\{X_t\}$ with $\{Y_t\}$, and the sequence of operators $\{\mathscr{B}_t\}$. The next lemma does precisely that:
\begin{proposition} In the notation of Section \ref{notation}, and under Assumptions \ref{assumptions}, it holds that
$$\mathscr{R}^{YX}_t=\sum_u \mathscr{B}_{t-u}\mathscr{R}^{X}_u,\qquad\sum_{t}\hs \mathscr{R}^{YX}_t\hs_1<\infty,\qquad\mathscr{F}^{YX}_{\omega}= \Q_{\omega}\mathscr{F}^{XX}_{\omega},$$
where $\Q_{\omega}$ is the linear operator with kernel
$$f^B_{\omega}(\tau,\sigma) = \sum_t e^{-\mathbf{i} \omega t } b_t(\tau,\sigma),$$
and satisfies
$$\int_{-\pi}^{\pi}\hs \Q_{\omega}\hs_2^2d\omega=\sum_{t}\hs \mathscr{B}_{t}\hs_2^2<\infty.$$
\end{proposition}
\begin{proof}
We begin by noting that given any $f\in L^2[0,1]$, we have
$$
\left[\left(\sum_u \mathscr{B}_{t-u}X_u\right)\otimes X_0\right] f= \langle X_0,f\rangle\sum_u \mathscr{B}_{t-u}X_u
=\sum_u \mathscr{B}_{t-u}\langle X_0,f\rangle X_u.
$$
As a result, it holds that
$$\mathbb{E}\left[\sum_u \mathscr{B}_{t-u}\langle X_0,f\rangle X_u\right]=\sum_u \mathscr{B}_{t-u}\mathbb{E}\left[\langle X_0,f\rangle X_u\right]$$
using Fubini's theorem and the fact that
\begin{eqnarray*}
\sum_u \mathbb{E}\|\mathscr{B}_{t-u}\langle X_0,f\rangle X_u\|&\leq&\sum_u \mathbb{E} \left[\hs \mathscr{B}_{t-u} \hs_2  \hs X_u\otimes X_0\hs_2\|f\|\right]\\
&\leq&\|f\|  \sum_u \hs \mathscr{B}_{t-u}\hs_2 \mathbb{E} \left[ \| X_u\| \|X_0\|\right]\\
&\leq&\|f\|^2  \sqrt{\mathbb{E}[ \|X_u\|^2]\mathbb{E}[\|X_0\|^2]} \sum_u \hs \mathscr{B}_{t-u}\hs_2\\
&<&\infty.
\end{eqnarray*}
Consequently, since $\{\epsilon_t\}$ is uncorrelated with $\{X_t\}$, we have that
$$\mathscr{R}^{YX}_tf=\mathbb{E}\left[\sum_u \mathscr{B}_{t-u}\langle X_0,f\rangle X_u\right]=\sum_u \mathscr{B}_{t-u}\mathbb{E}\left[\langle X_0,f\rangle X_u\right]=\sum_u\mathscr{B}_{t-u}\mathscr{R}_uf$$
which proves the first part of the proposition, since $f\in L^2[0,1]$ was arbitrary. In order to show that $\mathscr{R}^{YX}_t$ is nuclear-summable, we use H\"older's inequality for Schatten norms, and Tonelli's theorem to write
$$\sum_t\hs \mathscr{R}^{YX}_t\hs_1\leq \sum_t\sum_u\hs \mathscr{B}_{t-u}\mathscr{R}^X_{u}\hs_1\leq \sum_u\sum_t\hs \mathscr{B}_{t-u}\hs_2 \hs\mathscr{R}^X_{u}\hs_2\leq \sum_u\hs \mathscr{B}_{t-u}\hs_2\sum_t \hs\mathscr{R}^X_{u}\hs_2<\infty.$$
It follows that the Fourier transform $\mathscr{F}^{YX}_{\omega}$ of $\mathscr{R}^{YX}_t$ is well-defined. Following the standard manipulations leading to the convolution formula, we have
\begin{eqnarray*}
\mathscr{F}^{YX}_{\omega}&=&\frac{1}{2\pi}\sum_t e^{-\mathbf{i} t \omega} \mathscr{R}^{YX}_t\\
&=&\frac{1}{2\pi}\sum_t e^{-\mathbf{i} t \omega}\sum_u\mathscr{B}_{t-u}\mathscr{R}_{u}\\
&=&\frac{1}{2\pi}\sum_t \sum_ue^{-\mathbf{i} (t-u) \omega}\mathscr{B}_{t-u}e^{-\mathbf{i} u\omega}\mathscr{R}_{u}\\
&=&\sum_te^{-\mathbf{i} (t-u) \omega}\mathscr{B}_{t-u}\frac{1}{2\pi}\sum_u e^{-\mathbf{i} u \omega}\mathscr{R}_{u}\\
&=&\Q_{\omega}\mathscr{F}^{XX}_{\omega}.
\end{eqnarray*}
Here, we have made use of Fubini's theorem, noting that
$$\sum_u \sum_t \hs e^{-\mathbf{i} t \omega}\mathscr{B}_{t-u}\mathscr{R}^X_{u}\hs_1=\sum_u \sum_t \hs \mathscr{B}_{t-u}\mathscr{R}^X_{u}\hs_1<\infty.$$

\end{proof}

When $\mathscr{F}^{XX}_{\omega}$ is strictly positive uniformly over $\omega$ (so that its range is $L^2([0,1],\mathbb{C})$ itself), then the proposition implies that
\begin{equation}\label{eq:inversion}
\Q_{\omega} = \mathscr{F}^{YX}_{\omega} \big(\mathscr{F}^{XX}_{\omega}\big)^{-1} .
\end{equation}
It follows that the operator $\mathscr{B}_t$ can be deduced by inverse Fourier transforming,
\begin{align*}
\mathscr{B}_t = \int_{-\pi}^{\pi} \Q_{\omega} \exp\big\{-\mathbf{i} t \omega \big\} d \omega.
\end{align*}
This allows us to formulate an estimation strategy in the Fourier domain, as described in the next section.

\section{Methodology for Estimation} \label{estimation}

The results in the previous Section suggest the following estimation strategy, when we have a finite stretch $\{(X_t,Y_t)\}_{t=0}^{T-1}$ of length $T$ of the coupled series at our disposal. 
\begin{enumerate}
\item Estimate $\mathscr{F}^{XX}_{\omega}$ and $\mathscr{F}^{YX}_{\omega}$ nonparametrically, say by $\widehat{\mathscr{F}}^{XX}_{\omega,T}$ and $\widehat{\mathscr{F}}^{YX}_{\omega,T}$. This can be done using the approach introduced in \cite{panaretos2013fourier}, described in more detail in Section \ref{F_estimation}.

\item Construct a regularised estimator of $(\mathscr{F}^{XX}_{\omega})^{-1}$ based on $\widehat{\mathscr{F}}^{XX}_{\omega,T}$. Notice that regularisation is necessary, as the operator $\widehat{\mathscr{F}}^{XX}_{\omega,T}$ will be of finite rank and its maximal eigenvalue will diverge as $T$ grows. We consider this problem in Section \ref{Finverse_estimation}.

\end{enumerate}

Once these two steps have been completed, one can plug the corresponding estimators into Equation \ref{eq:inversion}, to obtain the regularised estimator of $\Q_{\omega}$, and consequently of $\mathscr{B}_t$. This is defined in Section \ref{B_estimator}.

\subsection{Estimation of $\mathscr{F}^{XX}_{\omega}$ and $\mathscr{F}^{YX}_{\omega}$}\label{F_estimation}
Following \cite{panaretos2013fourier}, let $W(x):\mathbb{R}\rightarrow  (0,\infty)$ be a positive real function such that 
\begin{enumerate}
	\item $W$ is of bounded variation.
	\item $W(x) = 0$ if $|x| \geq 1$.
	\item $\int_{-\infty}^{\infty} W(x)d x = 1,$
	\item $\int_{-\infty}^{\infty} W(x)^{2} d x < \infty.$
\end{enumerate}
Define a kernel of bandwidth $B_T$ as
\begin{align*}
W^{(T)}(x) = \frac{1}{B_T} \sum_{k\in \mathbb{Z}} W\Big(\frac{x + 2k \pi }{B_T} \Big).
\end{align*}
We will use this kernel in order to construct estimators in the frequency domain. Specifically, defining the discrete Fourier transforms of the two time series as
$$\widetilde{X}_{\omega,T} = \frac{1}{\sqrt{T}}\sum_{t=0}^{T-1} X_t \exp\{-\mathbf{i} \omega t \}\qquad\&\qquad \widetilde{Y}_{\omega,T} = \frac{1}{\sqrt{T}} \sum_{t=0}^{T-1} Y_t \exp\{-\mathbf{i} \omega t \},$$
the periodogram operator of $\{X_t\}$ at frequency $\omega$ (and its corresponding kernel) will be given by the empirical covariance (and its corresponding kernel) of the discrete Fourier transform at frequency $\omega$,
$$\mathscr{P}^{XX}_{\omega,T}=\widetilde{X}_{\omega,T}\otimes \overline{\left(\widetilde{X}_{\omega,T}\right)} \qquad\&\qquad p_{\omega}^{(T)}(\tau,\sigma) =  \widetilde{X}_{\omega,T}(\tau)\otimes \overline{\left(\widetilde{X}_{\omega,T}(\sigma)\right)}.$$
Similarly, the empirical cross-covariance of the discrete Fourier transforms of $X$ and $Y$ yields the cross-periodogram operator,
$$\mathscr{P}^{YX}_{\omega,T}=\widetilde{Y}_{\omega,T}\otimes \overline{\left(\widetilde{X}_{\omega,T}\right)}.$$
These can be smoothed using $W^{(T)}$, in order to yield the estimators of the spectral density operator of $X$ (and spectral density kernel), and of the cross-spectral density operator of $(X,Y)$,
\begin{equation}\label{FXX_estimator}
\widehat{\mathscr{F}}^{XX}_{\omega,T}= \frac{1}{T} \sum_{s=0}^{T-1} W^{(T)}(\omega - \nu_s)\mathscr{P}^{XX}_{\omega,T}\quad\&\quad f_{\omega}^{(T)}(\tau,\sigma) = \frac{1}{T} \sum_{s=0}^{T-1} W^{(T)}(\omega - \nu_s) p_{\nu_s}^{(T)}(\tau,\sigma),
\end{equation}
\begin{equation}\label{FXY_estimator}
\widehat{\mathscr{F}}^{YX}_{\omega,T}= \frac{1}{T}\sum_{s=0}^{T-1} W^{(T)}(\omega - \nu_s)\mathscr{P}^{YX}_{\nu_s,T}
\end{equation}
where
$$\nu_s=\frac{2\pi s}{T},\qquad s=0,1,\ldots,T-1$$

\subsection{Regularised Estimation of $(\mathscr{F}^{XX}_{\omega})^{-1}$}\label{Finverse_estimation}

Once we have the estimators $\widehat{\mathscr{F}}^{XX}_{\omega,T}$ and $\widehat{\mathscr{F}}^{XX}_{\omega,T}$, a naive approach to estimating $\Q_{\omega}$ would be to use the estimator
$$\widehat{\mathscr{F}}^{YX}_{\omega,T} \left(\widehat{\mathscr{F}}^{XX}_{\omega,T}\right)^{\dagger}$$
where $\left(\widehat{\mathscr{F}}^{XX}_{\omega,T}\right)^{\dagger}$ is the pseudoinverse of $\widehat{\mathscr{F}}^{XX}_{\omega,T}$. However, as can clearly be seen using the spectral decompositions
$$\widehat{\mathscr{F}}^{XX}_{\omega,T}=\sum_{n=1}^{T}\widehat{\lambda}_{n}^{\omega}\widehat{\varphi}_{n}(\omega)\otimes\widehat{\varphi}_n(\omega),\quad \left(\widehat{\mathscr{F}}^{XX}_{\omega,T}\right)^{\dagger}=\sum_{n=1}^{T}(\widehat{\lambda}_{n}^{\omega})^{-1}\widehat{\varphi}_{n}(\omega)\otimes\widehat{\varphi}_n(\omega),$$
the eigenvalues of $\left(\widehat{\mathscr{F}}^{XX}_{\omega,T}\right)^{\dagger}$ will not remain bounded as $T$ diverges when $\widehat{\mathscr{F}}^{XX}_{\omega,T}$ is consistent for $\mathscr{F}_{\omega}^{XX}$ (the latter being nuclear).

This effect will generally \emph{not} be annihilated by the application of the integral operator $\widehat{\mathscr{F}}^{YX}_{\omega,T} $ from the left, when forming the naive estimator $\widehat{\mathscr{F}}^{YX}_{\omega,T} \left(\widehat{\mathscr{F}}^{XX}_{\omega,T}\right)^{\dagger}$. The problem is that the spectrum of $\widehat{\mathscr{F}}^{YX}_{\omega,T} $ depends on $\Q_{\omega}$, which a priori has no structural relationship with $\mathscr{F}^{XX}_{\omega}$. Said differently, if $\widehat{\mathscr{F}}^{YX}_{\omega,T} $ is expanded in the tensor product basis given by the eigenfunctions of $\mathscr{F}^{XX}_{\omega}$ (extended to a complete system),
$$\widehat{\mathscr{F}}^{YX}_{\omega,T}=\sum_{n,m}\widehat a_{n,m}^{\omega}\widehat{\varphi}_{n}(\omega)\otimes\widehat{\varphi}_m(\omega)$$
 there is no reason to expect that the resulting basis coefficients $\{\widehat a_{n,m}^{\omega}\}$ will decay sufficiently fast for the products $\widehat a_{n,m}^{\omega}(\widehat{\lambda}_n^{\omega})^{-1}$ to remain bounded in $n$ as $T$ grows. Thus, it is necessary to use some form of regularisation. Two classical strategies are:
\begin{itemize}

\item[(i)] \emph{Spectral truncation}. Here, one would replace the generalised inverse $\left(\widehat{\mathscr{F}}^{XX}_{\omega,T}\right)^{\dagger}$ by
$$\sum_{n=1}^{K(T)}(\widehat{\lambda}_{n}^{\omega})^{-1}\widehat{\varphi}_{n}(\omega)\otimes\widehat{\varphi}_n(\omega)$$
where $K(T)<T$ grows sufficiently slowly in order to control the terms $a_{n,m}^{\omega}(\widehat{\lambda}_n^{\omega})^{-1}$  
\item[(ii)] \emph{Tikhonov regularisation}. Here, one would replace the generalised inverse $\left(\widehat{\mathscr{F}}^{XX}_{\omega,T}\right)^{\dagger}$ by a ridge-regularised inverse
$$\big[ \widehat{\mathscr{F}}^{XX}_{\omega,T} + {\zeta_T} \mathscr{I}  \big]^{-1}=\sum_{n=1}^{\infty}(\zeta_T+\widehat{\lambda}_n^{\omega})^{-1}\widehat{\varphi}_{n}(\omega)\otimes\widehat{\varphi}_n(\omega)$$
where $\zeta_T$ decays to zero sufficiently slowly in order to control the behaviour of the terms $a_{n,m}^{\omega}(\zeta_T+\widehat{\lambda}_n^{\omega})^{-1}$.

\end{itemize}

\noindent The first approach (spectral truncation) is essentially the approach described by \citet[Equation 3.4]{hormann2015estimation}. It can be seen as the extension of functional PCA regression (e.g. \citet{HH07}, \citet{imaizumi2016pca}) to the case of functional time series. 
\citet{hormann2015estimation} choose the value $K(T)$ to be dependent on the rate of decay of $\sup_{\omega}\hs\widehat{\mathscr{F}}^{XX}_{\omega}-\mathscr{F}^{XX}_{\omega}\hs_{\infty}$ and $\sup_{\omega}\hs\widehat{\mathscr{F}}^{YX}_{\omega}-\mathscr{F}^{YX}_{\omega}\hs_{\infty}$ (assumed known), in a way that guarantees consistency of the estimator eventually constructed. In principle, one could be more ambitious, and use a a frequency-dependent truncation level $K(T,\omega)$, but it seems unlikely to have detailed enough information on the decay rates $\hs\widehat{\mathscr{F}}^{XX}_{\omega}-\mathscr{F}^{XX}_{\omega}\hs_{\infty}$ and $\hs\widehat{\mathscr{F}}^{YX}_{\omega}-\mathscr{F}^{YX}_{\omega}\hs_{\infty}$ at each frequency $\omega$. 

Though spectral truncation is a very popular technique in the i.i.d. case, it poses some challenges both in terms of theoretical study, as well as practical performance, which might be exacerbated in the dependent case:
\begin{enumerate}
\item To this date, and to the best of our knowledge, there are no results concerning the MSE convergence rates for the spectral truncation estimator, even in the i.i.d. case. \citet{HH07} (and \citet{imaizumi2016pca}) establish rates of convergence for small ball probabilities, but not for the MSE itself. \citet{HH07} explain that to upgrade to MSE results, the spectral truncation estimator needs to be modified to a non-linear truncated version (see the discussion after \citet[Theorem 1]{HH07}) and also \citet[Theorem 5, Appendix A.2]{HHN06}). It thus seems that in the more challenging weakly dependent case, spectral truncation may not be the most fruitful avenue to obtain MSE convergence rates.

\item In practical terms, a challenge that spectral truncation encounters is in the case where eigenvalues are nearly tied, because the chosen  subspace $\{\varphi_n(\omega)\}_{j=1}^{K(T)}$ makes no reference to the quantity of interest $\Q_{\omega}$. Specifically, $\Q_{\omega}$ might be well expressed in some but not other eigenfunctions of  $\widehat{\mathscr{F}}^{XX}_{\omega}$, and this irrespectively of the size of the corresponding eigenvalues (according to which the truncation is performed). Thus, if the eigenvalues $\{\widehat\lambda^{\omega}_{K\pm j}\}_{j=1}^{J(\omega)}$ of $\widehat{\mathscr{F}}^{XX}_{\omega}$ are nearly tied, the sample variability of the estimator will increase, if this is well expressed in some but not all of the eigenfunctions of order $\{K\pm j; j=1,...,J(\omega)\}$. Intuitively: a certain term that is highly correlated with $\Q_{\omega}$ may come in or be left out of the truncation simply because of sample variability, leading to variance inflation. This phenomenon was doscumented by \citet{HH07} in the standard functional linear model, and can be a serious issue in the time series case, since we are considering a whole range frequencies, and thus of approximate eigenvalue ties $\{\widehat\lambda^{\omega}_{K\pm j}:j=1,...,J(\omega); \omega\in[-\pi,\pi]\}$.

\end{enumerate}

\citet{HH07} introduced and studied Tikhonov regularisation as an alternative that circumvents these issues. Indeed, they were able to deduce convergence rates for the MSE of the Tikhonov estimator, as opposed to the small ball probability rates for spectral truncation. For these two reasons, we will follow the Tikhonov approach here, defining $\big[ \widehat{\mathscr{F}}^{XX}_{\omega,T} + \zeta_T \mathscr{I}\big]^{-1} $ to be the (regularised) estimator of $\big[\mathscr{F}_{\omega}^{XX}\big]^{-1}$. We put all the elements together in the next section, to define our estimator.

\subsection{The Smoothed Fourier-Tikhonov Estimator of $\{\mathscr{B}_t\}$}

Following the discussion in the two previous subsection, let $B_T>0$ be a bandwidth and $\zeta_T$ a Tikhonov parameter. The (smoothed) Fourier-Tikhonov estimator of $\{\mathscr{B}_t\}_t$ is defined to be
\begin{equation}\label{B_estimator}
\widehat{\mathscr{B}}_t =  \int_{-\pi}^{\pi} \widehat{\Q}_{\omega,T} \exp\big\{-\mathbf{i} t\omega \big\} d\omega
\end{equation}
where
\begin{equation}\label{Q_estimator}
\widehat{\Q}_{\omega,T} =\widehat{\mathscr{F}}^{YX}_{\omega,T}  \big[ \widehat{\mathscr{F}}^{XX}_{\omega,T} + \zeta_T \mathscr{I}\big]^{-1} 
\end{equation}
is the estimator of $\Q_{\omega}$. Recall here that $\widehat{\mathscr{F}}^{YX}_{\omega,T}$ and $ \widehat{\mathscr{F}}^{XX}_{\omega,T}$ are the smoothed periodogram and smoothed cross-periodogram estimators defined in Section \ref{F_estimation} (see Equations \ref{FXX_estimator} and \ref{FXY_estimator}). The asymptotic performance of our estimator, and its dependence on the choice of $\zeta_T$ is investigated in the next Section.

\section{Rate of Convergence}\label{rate}

In this section, we state the main result of this paper, concerning the rate of convergence of the MSE of the Smoothed Fourier-Tikhonov Estimator \ref{B_estimator}. We begin by noting that one can establish consistency (without a rate of convergence), by letting $B_T\rightarrow 0$ and $TB_T\rightarrow\infty$ as $T\rightarrow \infty$, provided that the decay rate of $\zeta_T$ is taken to be a suitable function of $\sup_{\omega}\hs\widehat{\mathscr{F}}^{XX}_{\omega}-\mathscr{F}^{XX}_{\omega}\hs_{\infty}$ and $\sup_{\omega}\hs\widehat{\mathscr{F}}^{YX}_{\omega}-\mathscr{F}^{YX}_{\omega}\hs_{\infty}$. This would follow similar steps as \citet{hormann2015estimation}, but adapted to the case of Tikhonov regularisation, and would \emph{not} require any structural assumptions on the rate of decay of $\{\lambda_n^{\omega}\}$ or indeed on the spectra of $\{\mathscr{B}_t\}$, just as the results of \citet{hormann2015estimation} did not either.

However, we would like to be able to make more refined statements, and in particular to establish convergence rates, in the form of a rate of decay for the mean square error
\begin{align*}
\mathbb{E}\left\{\sum_t\hs\mathscr{B}_t -\widehat{\mathscr{B}}_t  \hs_{2}^2\right\}  = \mathbb{E}\left\{\int_{0}^{2\pi}\hs \Q_{\omega} - \widehat{\Q}_{\omega,T}\hs_{2}^2\right\} d\omega,
\end{align*}
where equality follows from Parseval's relation. Such rates will necessarily depend on the decay rate of $\{\lambda_n^{\omega}\}$, and furthermore on the spectra of $\{\mathscr{B}_t\}$. Our goal will thus be to establish a convergence rate that links these behaviours, and illustrates their interplay with the tuning parameters $B_T$ and $\zeta_T$. 

We will work in the so-called \emph{mildly ill-posed} setting, where the spectra involved exhibit a polynomial decay. Specifically, recall that $ \mathscr{F}^{YX}_{\omega}$, $\mathscr{F}^{XX}_{\omega}$ and $\mathscr{F}^B$ have integral kernels admitting series representations
\begin{align*}
f^{XX}_{\omega}(\tau,\sigma) &= \sum_{i=1}^{\infty} \lambda_i^{\omega} \varphi_i^{\omega}(\tau) \overline{\varphi_i^{\omega}(\sigma)}\\
f^{YX}_{\omega}(\tau,\sigma) &= \sum_{i,j=1}^{\infty} a_{ij}^{\omega} \varphi_i^{\omega}(\tau) \overline{\varphi_j^{\omega}(\sigma)} \\
f^B_{\omega}(\varsigma,\tau) & = \sum_{i,j=1}^{\infty} b_{ij}^{\omega} \varphi_i^{\omega}(\varsigma) \overline{\varphi_j^{\omega}(\tau)}.
\end{align*}

Our necessary further assumptions are collected below.
\begin{assumptions}[Ill-Posedness, Spectral Smoothness, and Weak Dependence]\label{assumptions}
In the context of model \ref{model}, we assume:
\begin{itemize}
\item[(B1)] For all $j$ and $\omega$ it holds that
$$\lambda_j^{\omega} \asymp C j^{-\alpha}, \qquad \sum_i\big| b_{ij}^{\omega}\big| \leq C j^{-\beta}.$$
with $\alpha > 1$, $\beta > 1/2$ and $\alpha < \beta + 1/2$ . 

\item[(B2)] Whenever $\varphi_i^{\omega}$ is an eigenfunction of $\mathscr{F}^{XX}_{\omega}$, then so is its complex conjugate $\overline{\varphi_i^{\omega}}$. That is,  
$$\Big\{\varphi_i^{\omega} :i  =1,2,\ldots \Big\} = \Big\{\overline{\varphi_i^{\omega}} :  i = 1,2,\ldots \Big\}.$$
Therefore, for each $i$, there exists unquely an index $i^{'}$ such that  $\big \langle \varphi_i^{\omega} , \varphi_{i'}^{\omega} \big \rangle=1$ and $ \big \langle \varphi_i^{\omega} , \varphi_j^{\omega}\big \rangle=0$ when $ j\neq i^{'}$.

\item[(B3)]  The functions $r_t^X$ and $b_t$ are continuous for all $t \in \mathbb{Z}$ with respect to $\tau, \sigma \in [0,1]$ and
\begin{eqnarray*}
&&\sum_{t \in \mathbb{Z}} |t|^{p+2} \big\| r_t^X(\tau,\sigma) \big\|_{\infty} < \infty \\
&&\sum_{t \in \mathbb{Z}} |t|^{p+2} \big\| b_t(\tau,\sigma) \big\|_{\infty} < \infty.
\end{eqnarray*}

\item[(B4)]  
\begin{eqnarray*}
&&\sum_{t \in \mathbb{Z}} |t|^{p+5} \hs \mathscr{R}_t \hs_1 < \infty \\
&&\sum_{t \in \mathbb{Z}} |t|^{p+5} \hs \B_t\hs_1 < \infty.
\end{eqnarray*}

\item[(B5)] The kernel $W$ is uniformly bounded,  compactly supported and {even}  on $[-1,1]$
 $\int_{\mathbb{R}} W(\alpha) d\alpha = 1$;
 There exists a positive integer $p$ such that $B_T^{p+1} < T^{-1}$ and for $j\leq p-1$.
\begin{align*}
\int_{\mathbb{R}} W(\alpha) \alpha^j d\alpha = 0.
\end{align*}

\item[(B6)] There exists constant $C<\infty$  such that
\begin{align*}
 \sum_{t_1,t_2,t_3 \in \mathbb{Z}} \hs{\text{cum}} (X_{t_1}, X_{t_2}, X_{t_3}, X_0) \hs_{1}  < C.
\end{align*}

\end{itemize}
\end{assumptions}

Condition (B1) is the direct extension of the mild ill-posedness conditions of \citet{HH07} to the time series context (see \citet[Section 3]{HH07} for a detailed discussion; \citet[Section 3.1]{imaizumi2016pca} also introduce the same conditions in the function-on-function regression case)\footnote{Note, however, that we do not need to make any assumption on the separation of the eigenvalues, since Tikhonov regularisation is immune to eigenvalue ties.}. Condition (B2) assumes that the set of eigenfunctions is closed under conjugation (i.e. that if a function is an eigenfunction, then its complex conjugate will also be an eigenfunction). The conditions in (B3) can be seen as weak dependence conditions that suffice for the existence of Taylor expansions of sufficiently high order of the spectral density operator and the Fourier transform of the filter with respect to the frequency argument. Finally, conditions (B4) and (B6) are also weak dependence conditions of Brillinger-type, that are sufficient for the establishment of convergence rates of the spectral density estimator to its estimand (as in \citet{panaretos2013fourier}). Finally, (B5) is a standard higher order kernel assumption often encountered in density estimation and deconvolution.

\smallskip
We are now able to state our main result:

\begin{theorem}[Rate of Convergence]\label{main_theorem}
Let $\{\widehat{\mathscr{B}}_t\}$ be the Fourier-Tikhonov estimator \ref{B_estimator} of the coefficients $\{\mathscr{B}_t\}$  in the functional time series regression model \ref{model} satisfying Assumptions \ref{assumptions}. Then, under conditions (B1)-(B7), there exists a sequence of events $G_T$ such that $\mathbb{P}\big[G_T \big] \rightarrow 1$, and
\begin{equation}\label{rate}
\mathbb{E}\left[\sum_t\hs\mathscr{B}_t -\widehat{\mathscr{B}}_t  \hs_{2}^2\,;\, {G_T}\right] = \frac{1}{B_T}O\Big(T^{-(  2\beta - 1)/(\alpha + 2\beta)} \Big),
\end{equation}
provided the Tikhonov parameter satisfies $\zeta_T = T^{-\alpha/(\alpha + 2\beta)}$ and the bandwidth satisfies $B_T = T^{-\gamma}$, with $\gamma$ such that 
$$\frac{\alpha -1}{\alpha+2\beta} < \gamma <  \frac{2\beta -\alpha}{\alpha + 2\beta}.$$ 
\end{theorem}

\begin{remark}
Note that assuming that $\frac{\alpha -1}{\alpha+2\beta} < \gamma <  \frac{2\beta -\alpha}{\alpha + 2\beta}$ is compatible with assumption (B1) since $\alpha < \beta + 1/2$.
\end{remark}

If we compare the rate \ref{rate} with the one obtained by \citet{HH07} in the i.i.d case, we see that they are identical except for the presence of the $B_T^{-1}$ factor in our case. Intuitively, this is the price we have to pay for the fact that the estimation of the spectral density operator is intrinsically harder than the estimation of a covariance operator: for densely observed functional data, a covariance operator can be estimated at a parametric rate  (\citet{hall2006properties}), but the spectral density operator can only be estimated at nonparametric rates (\citet{panaretos2013fourier}).

\medskip
The proof of Theorem \ref{main_theorem} is quite technical, and will be constructed via a series of intermediate results in the next Section.

\section{Proof of Theorem \ref{main_theorem}} \label{proof_section}

In the interest of tidiness, we introduce some additional notational conventions here that will be made frequent use of in the forthcoming lemmas and propositions.
\begin{itemize}
\item For fixed $\omega \in [0,2\pi]$, define $u_s$ to be an element of $\big\{\nu_s + 2k \pi: k \in \mathbb{Z} \big\}$ such that $\big|\omega - u_s \big| \leq \pi$. By this definition $u_s$ is well-defined and

$$f_{\nu_s}^{XX} = f_{u_s}^{XX}; \qquad W^{(T)} \big(\omega - \nu_s \big)  = W^{(T)}\big(\omega - u_s \big),$$

since $f_{\omega}^{XX}$ and $W^{(T)}$ are periodic with period $2\pi$.
\item $h_T(t) = 1_{[0,T-1]}$.
\end{itemize}
Note that $B_T = T^{-\gamma}$ and $\gamma > (\alpha -1)/(\alpha + 2\beta)$, thus
\begin{align*}
T^{-1} B_T  \zeta_T^{-2} =  T^{2\alpha/(\alpha + 2\beta)} T^{-\gamma} T^{-1}  = T^{-(2\beta -\alpha)/(\alpha + 2\beta)} T^{-\gamma} 
=  T^{-(2\beta -1)/(\alpha + 2\beta)} T^{(\alpha-1)/(\alpha + 2\beta)} T^{-\gamma} \\
= O\Big( T^{-(2\beta -1)/(\alpha + 2\beta)} \Big) .
\end{align*}

We first recall that 
\begin{align*}
\mathbb{E} \left\{ \sum_t \hs \B_t - \widehat{\B}_t \hs_2^2 \right\} = \mathbb{E} \left\{\int_0^{2\pi} \hs \Q_{\omega} - \widehat{\Q}_{\omega}\hs_2^2 \right\} d\omega = \int_{0}^{2\pi} \mathbb{E} \hs \Q_{\omega} - \widehat{\Q}_{\omega}\hs_2^2 d\omega.
\end{align*}
Hence, we first need to find the Hilbert-Schmidt norm by applying part C of Lemma \ref{SCoij} and then take the integral over $\omega$. Let 
\begin{align*}
\widetilde{\mathscr{Q}}_{\omega} &: = \F^{YX}_{\omega} \big(\F^{XX}_{\omega} + \zeta_T \mathscr{I}\big)^{-1} .
\end{align*}
Using the triangle inequality,
\begin{align*}
\hs \Q_{\omega} - \widehat{\Q}_{\omega,T} \hs_2^2 \leq 2 \hs\widetilde{\Q}_{\omega} - \Q_{\omega} \hs_2^2 + 2 \hs \widetilde{\Q}_{\omega} - \widehat{\Q}_{\omega,T}\hs_2^2.
\end{align*}
By definition,
\begin{align*}
\Q_{\omega} &= \Big\{\sum_{i,j} a_{ij}^{\omega} \varphi_i^{\omega} \otimes \overline{\varphi_j^{\omega}}\Big\}  \Big\{\sum_j \frac{1}{\lambda_j^{\omega}} \varphi_j^{\omega} \otimes  \overline{\varphi_j^{\omega}} \Big\}  = \sum_j  \sum_{i}\frac{a_{ij}^{\omega}}{\lambda_j^{\omega}} \varphi_i^{\omega}\otimes \overline{\varphi_j^{\omega}}  = \sum_j  \sum_{i}b_{ij}^{\omega} \varphi_i^{\omega}\otimes \overline{\varphi_j^{\omega}}\\
\widetilde{\Q}_{\omega} &= \Big\{\sum_{i,j} a_{ij}^{\omega} \varphi_i^{\omega} \otimes \overline{\varphi_j^{\omega}}\Big\} \Big\{\sum_j \frac{1}{\lambda_j^{\omega} + \zeta_T} \varphi_j^{\omega} \otimes \overline{\varphi_j^{\omega}} \Big\}  = \sum_j  \sum_{i}\frac{a_{ij}^{\omega}}{\lambda_j^{\omega} + \zeta_T} \varphi_i^{\omega} \otimes \overline{\varphi_j^{\omega} }.
\end{align*}
Thus,
\begin{align*}
\widetilde{\Q}_{\omega} - \Q_{\omega} &= \sum_{i,j} \frac{a_{ij}^{\omega} \zeta_T}{\lambda_j^{\omega} (\lambda_j^{\omega} + \zeta_T)} \varphi_i^{\omega} \otimes\overline{\varphi_j^{\omega}}\\
\hs\widetilde{\Q}_{\omega} - \Q_{\omega} \hs_2^2  &= \sum_{i,j} \frac{\big|a_{ij}^{\omega}\big|^2}{(\lambda_j^{\omega})^2}\times  \frac{\zeta_T^2}{(\lambda_j^{\omega} + \zeta_T)^2} = \sum_{j} \left\{\sum_i \big|b_{ij}^{\omega}\big|^2\right\} \times \frac{\zeta_T^2}{(\lambda_j^{\omega} + \zeta_T)^2}\leq O(1) \times \sum_j j^{-2\beta} \frac{\zeta_T^2}{(\lambda_j^{\omega} + \zeta_T)^2}\\
&=  O\Big(T^{-(2\beta -1)/(\alpha + 2\beta)} \Big).
\end{align*}
The last inequality comes from \eqref{seqA} and assumption (B1). We next decompose  $\widetilde{\Q}_{\omega} - \widehat{\Q}_{\omega,T} $ as
\begin{align*}
 \widetilde{\Q}_{\omega} - \widehat{\Q}_{\omega,T} &= \F_{\omega}^{YX} \big[\F_{\omega}^{XX} + \zeta_T \mathscr{I}\big]^{-1} - \widehat{\F}_{\omega,T}^{YX} \Big[\widehat{\F}^{XX}_{\omega,T} + \zeta_T \mathscr{I} \Big]^{-1}\\
& = \Big(\F_{\omega}^{YX} - \widehat{\F}_{\omega,T}^{YX}\Big) \Big[\F^{XX}_{\omega} + \zeta_T \mathscr{I} \Big]^{-1} + \Big( \widehat{\F}_{\omega,T}^{YX}- \F_{\omega}^{YX} \Big) \Big(\big[\F_{\omega}^{XX} + \zeta_T \mathscr{I}\big]^{-1} - \Big[\widehat{\F}^{XX}_{\omega,T} +   \zeta_T \mathscr{I} \Big]^{-1} \Big)  +\\
& \qquad \F_{\omega}^{YX}  \Big(\big[\F_{\omega}^{XX} + \zeta_T \mathscr{I}\big]^{-1} - \Big[\widehat{\F}^{XX}_{\omega,T} + \zeta_T \mathscr{I} \Big]^{-1} \Big)\\
& = \scrS_1 + \scrS_2 + \scrS_3.
\end{align*}
The remainder of the proof will deal with constructing upper bounds for the three terms $\{\scrS_1, \scrS_2,  \scrS_3\}$. To this aim, the strategy will be to: 
\begin{enumerate}
\item Apply part C of Lemma \ref{SCoij} with the orthogonal basis $\big\{ \varphi_i^{\omega} \big\}$ in order to to reduce the problem to calculations involving integrals of kernel functions.
\item  Use  Propositions  \ref{4cumS1}  and \ref{CovDenX} to break these integrals down into manageable terms.
\item Apply Lemma \ref{VPI} to determine the required upper bound for each of these terms. 
\end{enumerate}
We organise this process into separate subsections for each of the terms $\{\scrS_1, \scrS_2,  \scrS_3\}$, starting with $ \scrS_3$, then moving on to $ \scrS_1$ and finally $ \scrS_2$.
\subsection*{Bounding $\scrS_3$}
Let $\Delta = \widehat{\F}^{XX}_{\omega,T}  - \F_{\omega}^{XX}$. For simplicity,  also write
\begin{align*}
V =  \F_{\omega}^{XX}; \qquad  \widehat{V} = \widehat{\F}^{XX}_{\omega,T}; \qquad V^+ = \big[\F_{\omega}^{XX} + \zeta_T \mathscr{I}\big]^{-1} ; \qquad \widehat{V}^+ =  \Big[\widehat{\F}^{XX}_{\omega,T} + \zeta_T \mathscr{I} \Big]^{-1} .
\end{align*}
Using the identities
\begin{align*}   
&\widehat{V}^+ - V^+ = \widehat{V}^+ \Delta V^+  \Rightarrow  \widehat{V}^+ \big[ I + \Delta V^+\big]  = V^+ \Rightarrow \widehat{V}^+ = V^+ \big[I + \Delta V^+\big]^{-1},
\end{align*}
we obtain 
\begin{align*}
\widehat{V}^+ - V^+= V^+\Delta \widehat{V}^+  = V^+\Delta  V^+ \big[I + \Delta V^+\big]^{-1}.
\end{align*}
Thus, $\scrS_3 = \F_{\omega}^{YX} V^+ \Delta V^+ \big[I + \Delta V^+ \big]^{-1}$.
By Proposition \ref{FhatXX}, $\mathbb{E} \hs \Delta \hs_2 = O\big(B_T^{-1/2}T^{-1/2}\big)$ uniformly over $\omega$, and $\big\|V^+\big\| \leq1/\zeta_T$. Hence, on the event $G_T$ from Proposition \ref{RidOp}
\begin{align*}
\hs\Delta \zeta_T^{-1}\hs_2^2 \lesssim  T^{2\alpha/(\alpha + 2\beta)} T^{-1} T^{\gamma} = T^{-(2\beta - \alpha)/(\alpha + 2\beta)} T^{\gamma} = o(1),
\end{align*}
since $\gamma < (2\beta-\alpha)/(\alpha + 2\beta)$. Thus $\hs\big(I + \Delta V^+\big)^{-1} \hs_2$ is uniformly bounded on the even $G_T$. 
Hence, on $G_T$,
\begin{align*}
\hs \scrS_3 \hs_2^2 \leq O(1) \hs \F_{\omega}^{YX} V^+ \Delta V^+\hs_2^2  \leq O(1) \hs \F_{\omega}^{YX} V^+\hs_2^2 \times \hs \Delta V^+\hs_2^2. 
\end{align*}
The first factor of the right hand side is bounded by 
\begin{align*}
\sum_{i,j} \frac{\big|a_{ij}^{\omega}\big|^2}{(\lambda_j^{\omega} + \zeta_T)^2} = \sum_j \Big\{\sum_i \frac{\big|a_{ij}^{\omega}\big|^2}{(\lambda_j^{\omega} + \zeta_T)^2} \Big\} \leq O(1) \times  \sum_j j^{-2\beta} = O(1).
\end{align*}
Upon expanding  $\Delta=  \sum_{ij} \Delta_{ij} \varphi_i^{\omega} \otimes \overline{\varphi_j^{\omega}}$, we obtain 
\begin{align*}
\Delta \big[\F_{\omega}^{XX} + \zeta_T \mathscr{I}\big]^{-1} = \Big(\sum_{ij} \Delta_{ij} \varphi_i^{\omega} \otimes \overline{\varphi_j^{\omega}} \Big) \Big(\sum_j\frac{1}{\lambda_j^{\omega} + \zeta_T} \varphi_j^{\omega} \otimes \overline{\varphi_j^{\omega}} \Big) = \sum_{i,j} \frac{\Delta_{ij}}{\lambda_j^{\omega} + \zeta_T} \varphi_i^{\omega} \otimes \overline{\varphi_j^{\omega}}.
\end{align*}
It follows that
\begin{align}
\mathbb{E}\hs \Delta \big[\F_{\omega}^{XX} + \zeta_T \mathscr{I}\big]^{-1}\hs_2^2 = \sum_{i,j} \frac{\mathbb{E}\big|\Delta_{ij}\big|^2}{(\lambda_j^{\omega} + \zeta_T)^2} \label{ExpDel}.
\end{align}
Letting $\kappa$ be the integral kernel of $\Delta$, another way to express $\big|\Delta_{ij}\big|^2$ is via Lemma \ref{SCoij}, yielding
\begin{align*}
\mathbb{E}\big| \Delta_{ij}\big|^2 = \int_{[0,1]^4} \mathbb{E} \Big[\kappa(\tau_1,\sigma_1) \overline{\kappa(\tau_2,\sigma_2)}\Big] \times \overline{\varphi_i^{\omega}(\tau_1)} \varphi_j^{\omega}(\sigma_1) \varphi_i^{\omega}(\tau_2) \overline{\varphi_j^{\omega}(\sigma_2)} d\tau_1 d\sigma_1 d\tau_2 d\sigma_2.
\end{align*}
By proposition \ref{CovDenX},  $\mathbb{E} \Big[\kappa(\tau_1,\sigma_1) \overline{\kappa(\tau_2,\sigma_2)}\Big]$ can be decomposed as 
\begin{align*}
O\big(T^{-1} B_T^{-1} \big) \times \Big\{  f_{\omega}^{XX}(\tau_1,\tau_2) f_{-\omega}^{XX}(\sigma_1,\sigma_2)  + 1_{I_T}(\omega) f_{\omega}^{XX}(\tau_1,\sigma_2) f_{-\omega}^{XX}(\tau_2,\sigma_1)  \Big\}+ \\
 O\big(T^{-1}  \big) \times 1_{I_T}(\omega) \Big\{f_{\omega}^{XX}(\tau_1,\sigma_2) f_{-\omega}^{XX,(1)}(\sigma_1,\tau_2) + f_{-\omega}^{XX}(\sigma_1,\tau_2) f_{\omega}^{XX,(1)}(\tau_1,\sigma_2)  \Big\} + \\
\qquad  O\big(T^{-1} B_T \big) \times \big\{ \vartheta_1(\tau_1,\tau_2) \odot \vartheta_2(\sigma_1,\sigma_2) +1_{I_T}(\omega) \vartheta_3(\tau_1,\sigma_2) \odot \vartheta_4(\sigma_1,\sigma_2\big\} +\\
\qquad \frac{1}{T^2} \sum_{r,s=0}^{T-1} W^{(T)}(\omega -\nu_s) W^{(T)}(\omega - \nu_r ) p_{r,s}^{(T)}(\tau_1,\sigma_1,\tau_2,\sigma_2).
\end{align*}
We now bound each term of $\mathbb{E} |\Delta_{ij}|^2$, before summing over $i$ and $j$ as in \eqref{ExpDel}, and then taking the integral over $\omega$.

\subsubsection*{Bounding the term $O\big(T^{-1} B_T^{-1}\big)f_{\omega}^{XX}(\tau_1,\tau_2)  f_{-\omega}^{XX}(\sigma_1,\sigma_2)$.} \label{fpmomega1}
Write
\begin{align*}
&\int_{[0,1]^4} f_{\omega}^{XX}(\tau_1,\tau_2) \overline{f_{\omega}^{XX}(\sigma_1,\sigma_2)} \times  \overline{\varphi_i^{\omega}(\tau_1)}\varphi_j^{\omega}(\sigma_1) \varphi_i^{\omega}(\tau_2)\overline{\varphi_j^{\omega}(\sigma_2)}  d\tau_1 d\sigma_1 d\tau_2 d\sigma_2  \\
&\qquad \qquad = \int_{[0,1]^2} f_{\omega}^{XX}(\tau_1,\tau_2) \overline{\varphi_i^{\omega}(\tau_1)} \varphi_i^{\omega}(\tau_2) d\tau_1 d\tau_2   \times \int_{[0,1]^2}  \overline{f_{\omega}^{XX}(\sigma_1,\sigma_2)} \varphi_j^{\omega}(\sigma_1) \overline{\varphi_j^{\omega}(\sigma_2)} d\sigma_1 d\sigma_2  \\
& \qquad \qquad = \lambda_i^{\omega} \lambda_j^{\omega}.
\end{align*}
Taking the sum over $i$ and $j$ as in \eqref{ExpDel}, by Lemma \ref{Slambdaj}, we obtain
\begin{align*}
O\big(T^{-1} B_T^{-1} \big) \sum_{i,j} \frac{\lambda_i^{\omega} \lambda_j^{\omega} }{(\lambda_j + \zeta_T)^2} = O\big(T^{-1} B_T^{-1} \big) \sum_i \lambda_i^{\omega} \sum_j \frac{\lambda_j^{\omega}}{(\lambda_j^{\omega} + \zeta_T)^2} = O\big(B_T^{-1} \big) \times O\Big(T^{-(2\beta -1)/(\alpha + 2\beta)} \Big).
\end{align*}
\subsubsection*{Bounding the term  $O\big(T^{-1} B_T^{-1} \big) 1_{I_T}(\omega) f_{\omega}^{XX}(\tau_1,\sigma_2)  f_{-\omega}^{XX}(\tau_2,\sigma_1)$.} \label{fpmomega2}
Write
\begin{align*}
&\int_{[0,1]^4} f_{\omega}^{XX}(\tau_1,\sigma_2) \overline{f_{\omega}^{XX}(\tau_2,\sigma_1)} \times  \overline{\varphi_i^{\omega}(\tau_1)}\varphi_j^{\omega}(\sigma_1) \varphi_i^{\omega}(\tau_2)\overline{\varphi_j^{\omega}(\sigma_2)}  d\tau_1 d\sigma_1 d\tau_2 d\sigma_2  \\
&\qquad \qquad = \int_{[0,1]^2} f_{\omega}^{XX}(\tau_1,\sigma_2) \overline{\varphi_i^{\omega}(\tau_1)} \overline{\varphi_j^{\omega}(\sigma_2)} d\tau_1 d\sigma_2  \times \int_{[0,1]^2} \overline{f_{\omega}^{XX}(\sigma_1,\tau_2)}\varphi_j^{\omega}(\sigma_1)  \varphi_i^{\omega}(\tau_2) d\sigma_1 d\tau_2 \\
& \qquad \qquad = \lambda_i^{\omega} \int_{[0,1]} \overline{\varphi_i^{\omega}(\sigma_2)} \overline{\varphi_j^{\omega}(\sigma_2)} d\sigma_2 \times \lambda_j^{\omega} \int_{[0,1]} \varphi_j^{\omega}(\tau_2) \varphi_i^{\omega}(\tau_2) d\tau_2.
\end{align*}
This is dominated by $\lambda_i^{\omega} \lambda_j^{\omega}$ (the same argument as in the previous part has been applied here).

\subsubsection*{Bounding the term $O\big( T^{-1}\big) \times 1_{I_T}(\omega) f_{\omega}^{XX}(\tau_1,\sigma_2) f_{-\omega}^{XX,(1)}(\sigma_1,\tau_2)  $.}
Note that $\hs \F_{\omega}^{XX}\hs_1$ and $\hs \F_{\omega}^{XX,(1)} \hs_1$ are uniformly bounded. Applying Lemma  \ref{VPI}, we obtain the bound
\begin{align*}
O(T^{-1}) 1_{I_T}(\omega) \zeta_T^{-2}.
\end{align*}  
Note that the length of $I_T$ is of order $B_T$. Taking the integral over $\omega$, then we get 
$O\Big(T^{-1} \zeta_T^{-2}B_T\Big)  = O\Big(T^{-(2\beta -1)/(\alpha + 2\beta)} \Big)$.

\subsubsection*{Bounding the term $O\big( T^{-1}\big) 1_{I_T}(\omega) \times f_{\omega}^{XX,(1)}(\tau_1,\sigma_2) f_{-\omega}^{XX}(\sigma_1,\tau_2) $.}
For this term, we apply Lemma \ref{VPI} as in the previous paragraph.

\subsubsection*{Bounding the term $O\big(T^{-1} B_T \big) \vartheta_1(\tau_1,\tau_2) \odot \vartheta_2(\sigma_1,\sigma_2)$.} \label{vartheta1}
Applying Lemma \ref{VPI}, we obtain
$ O\big(T^{-1} B_T \big) \times \zeta_T^{-2} 
=  O\big(T^{-(2\beta-1)/(\alpha + 2\beta)} \big)$.

\subsubsection*{Bounding the term $O(T^{-1}B_T) \vartheta_3(\tau_1,\sigma_2) \odot \vartheta_4(\sigma_1,\tau_2)$.} \label{vartheta2}
Applying Lemma \ref{VPI} we obtain $O\big(T^{-1} B_T \zeta_T^{-2} \big) =  O\big(T^{-(2\beta-1)/(\alpha + 2\beta)} \big)$. 

\subsubsection*{Bounding the term $\frac{1}{T^2} \sum_{r,s=0}^{T-1} W^{(T)}(\omega -\nu_s) W^{(T)}(\omega - \nu_r ) p_{r,s}^{(T)}(\tau_1,\sigma_1,\tau_2,\sigma_2)$.}
 Applying Lemma \ref{VPI} part (C) and Proposition \ref{4cumS1} part (C), we have 
$$T^{-1} \zeta_T^{-2} = \frac{1}{B_T} O\Big( T^{-(2\beta -1)/(\alpha + 2\beta)} \Big) . $$
In summary, we have upper bounded $\scrS_3$ as required. 

\subsection*{Bounding $\scrS_1$}
Recall that
\begin{align*}
\widehat{\F}^{YX}_{\omega,T} & = \frac{1}{T}\sum_{s=0}^{T-1} W^{(T)}(\omega - \nu_s) \mathscr{P}^{YX}_{\nu_s,T} ; \qquad
\scrS_1= \Big( \F_{\omega}^{YX} - \widehat{\F}_{\omega,T}^{YX} \Big) \big[\F_{\omega}^{XX} + \zeta_T \mathscr{I}\big]^{-1} .
\end{align*}
Now write 
\begin{align*}
R_t &:= \sum_s\B_{t-s}  X_s\\
H_{\omega,T}  & := \frac{1}{\sqrt{T}} \left\{\sum_t h_T(t) R_t  \exp\big\{-\mathbf{i} \omega t \big\}\right\} \\
L_{\omega,T} &: =  H_{\omega,T} - \F^B_{\omega}  \widetilde{X}_{\omega,T} \\
K_{\omega,T} &: =  \frac{1}{\sqrt{T}} \left\{\sum_t h_T(t)\epsilon_t \exp\big\{-\mathbf{i} \omega t \big\}\right\}.
\end{align*}
In this notation,
\begin{align*}
 \F^B_{\omega}  \widetilde{X}_{\omega,T} & = \left\{\sum_t \B_{t} \exp \big\{-\mathbf{i} \omega t \big\} \right\}  \frac{1}{\sqrt{T}} \left\{ \sum_s X_s h_T(s) \exp \big\{-\mathbf{i} \omega s \big\}\right\} \\
& = \frac{1}{\sqrt{T}}\sum_{t,s} \B_t  X_s h_T(s) \exp\big\{-\mathbf{i} \omega (t+s) \big\} \\
H_{\omega,T} & = \frac{1}{\sqrt{T}} \sum_{u,v} h_T(u+v) \B_u X_v \exp\big\{-\mathbf{i} \omega (u+v) \big\}\\
L_{\omega,T}& = \frac{1}{\sqrt{T}}\sum_{u,v} \big\{ h_T(v) - h_T(u+v) \big\} \B_u X_v \exp\big\{-\mathbf{i} \omega (u+v) \big\} 
\end{align*}
The operator $\mathscr{P}_{\omega,T}^{YX}  = \widetilde{Y}_{\omega,T} \otimes \overline{\Big(\widetilde{X}_{\omega,T} \Big)}$ can be decomposed as
\begin{align*}
\mathscr{P}_{\omega,T}^{YX}  &= \frac{1}{\sqrt{T}} \left\{\sum_t h_T(t) \big[R_t + \epsilon_t\big] \exp\big\{-\mathbf{i} \omega t \big\}\right\} \otimes \widetilde{X}_{-\omega,T} \\
&= \frac{1}{\sqrt{T}} \left\{\sum_t h_T(t) R_t  \exp\big\{-\mathbf{i} \omega t \big\}\right\} \otimes \widetilde{X}_{-\omega,T}  + \frac{1}{\sqrt{T}} \left\{\sum_t h_T(t)\epsilon_t \exp\big\{-\mathbf{i} \omega t \big\}\right\} \otimes \widetilde{X}_{-\omega,T}\\
&=  H_{\omega,T} \otimes \widetilde{X}_{-\omega,T} + K_{\omega,T} \otimes \widetilde{X}_{-\omega,T},\\
& = \F_{\omega}^B \widetilde{X}_{\omega,T} \otimes \widetilde{X}_{-\omega,T}  +  L_{\omega,T} \otimes \widetilde{X}_{-\omega,T} + K_{\omega,T} \otimes \widetilde{X}_{-\omega,T}.
\end{align*}
Hence,
\begin{align*}
\widehat{\F}_{\omega,T}^{YX} & = \frac{1}{T} \sum_{s=0}^{T-1} W^{(T)}(\omega -\nu_s) \Big[ \F_{\nu_s}^B   \widetilde{X}_{\nu_s,T} \otimes \widetilde{X}_{-\nu_s,T}  +  L_{\nu_s,T} \otimes \widetilde{X}_{-\nu_s,T}  + K_{\nu_s,T} \otimes \widetilde{X}_{-\nu_s,T}\Big]  \\
& = \mathscr{D}_1 + \mathscr{D}_2 + \mathscr{D}_3.
\end{align*}
We can now decompose $\scrS_1$ based on  the $\mathscr{D}_i$,
\begin{align*}
\scrS_1 &= \Big( \F_{\omega}^{YX} - \widehat{\F}_{\omega,T}^{YX} \Big) \big[\F_{\omega}^{XX} + \zeta_T \mathscr{I}\big]^{-1}   \\
&= \Big( \F_{\omega}^{YX} - \mathscr{D}_1\Big) \big[\F_{\omega}^{XX} + \zeta_T \mathscr{I}\big]^{-1} +  \mathscr{D}_2  \big[\F_{\omega}^{XX} + \zeta_T \mathscr{I}\big]^{-1}   + \mathscr{D}_3  \big[\F_{\omega}^{XX} + \zeta_T \mathscr{I}\big]^{-1}    \\
&= \scrS_{11} + \scrS_{12} + \scrS_{13}.
\end{align*}
By the Cauchy-Schwarz inequality
\begin{align*}
\hs \scrS_1 \hs_2^2  \leq 3 \hs \scrS_{11} \hs_2^2 + 3 \hs \scrS_{12} \hs_2^2 +  3\hs \scrS_{13} \hs_2^2.
\end{align*}
We focus on each of the three terms in the following paragraphs.
\subsubsection*{Bounding $\scrS_{13}$.}
 Let $D_3$ be the integral kernel of $\mathscr{D}_3$. Then, similar to \eqref{ExpDel}, 
\begin{align}
\mathbb{E} \hs \scrS_{13} \hs_2^2 = \sum_{i,j} \frac{1}{(\lambda_j^{\omega} + \zeta_T)^2} \int_{[0,1]^2} \mathbb{E} \Big[D_3(\tau_1,\sigma_1) \overline{D_3(\tau_2,\sigma_2)} \Big]\overline{\varphi_i^{\omega}(\tau_1)}\varphi_j^{\omega}(\sigma_1) \varphi_i^{\omega}(\tau_2) \overline{\varphi_j^{\omega}(\sigma_2)}  d\tau_1 d\sigma_1 d\tau_2 d\sigma_2. \label{S13}
\end{align}
We need to work on the expectation inside the integral. First,
\begin{align*}
\overline{D_3(\tau_2,\sigma_2)} =  \sum_{r=0}^{T-1} W^{(T)}(\omega - \nu_r) \overline{ K_{\nu_r,T}(\tau_2) \widetilde{X}_{-\nu_r,T}(\sigma_2) } =  \sum_{r=0}^{T-1} W^{(T)}(\omega - \nu_r)  K_{-\nu_r,T}(\tau_2) \widetilde{X}_{\nu_r,T}(\sigma_2) 
\end{align*}
By independence of $X$ and $\epsilon$,
\begin{align*}
\mathbb{E}\big( K_{\nu_s,T}(\tau_1)  \widetilde{X}_{-\nu_s,T}(\sigma_1)  K_{-\nu_r,T}(\tau_2) \widetilde{X}_{\nu_r,T}(\sigma_2)\big) = \mathbb{E} \big[K_{\nu_s,T}(\tau_1) K_{-\nu_r,T}(\tau_2)\big] \times \mathbb{E} \big[\widetilde{X}_{-\nu_s,T}(\sigma_1)\widetilde{X}_{\nu_r,T}(\sigma_2)\big] .
\end{align*}
Let  $q_{\omega}$ be the spectral density function of $\big\{ \epsilon_t \big\}$. By independence of the $\{\epsilon_t\}$, we have $q_{\omega} = q_0$. Apply Proposition \ref{2.6} to the sequence $\{\epsilon_t\}$ to obtain
\begin{align}
&\mathbb{E} \Big[D_3(\tau_1,\sigma_1) \times \overline{D_3(\tau_2,\sigma_2)} \Big]  = \nonumber \\
& \frac{1}{T^2} \sum_{s,r=0}^{T-1} W^{(T)}(\omega -\nu_s) W^{(T)}(\omega  -\nu_r)   \mathbb{E}\big[K_{\nu_s,T}(\tau_1) K_{-\nu_r,T}(\tau_2) \big]  \times  \mathbb{E}\big[\widetilde{X}_{\nu_r,T} (\sigma_2)\widetilde{X}_{-\nu_s,T}(\sigma_1) \big]  \nonumber \\
&=  \frac{1}{T^2} \sum_{s,r=0}^{T-1} W^{(T)}(\omega -\nu_s) W^{(T)}(\omega -\nu_r)  \Big\{\delta_{rs} q_{\nu_s}(\tau_1,\tau_2) + \frac{1}{T}\vartheta_{1,\nu_r,\nu_s}(\tau_1,\tau_2) \Big\} \times \nonumber \\
&\qquad \qquad \Big\{\delta_{rs}f_{\nu_s}^{XX}(\sigma_1,\sigma_2) + \frac{1}{T} \vartheta_{2,\nu_r,\nu_s}(\sigma_1,\sigma_2)  \Big\}  \nonumber \\
&= \frac{1}{T^2} \sum_{s=0}^{T-1} W^{(T)}(\omega - \nu_s) W^{(T)}(\omega - \nu_s) q_{\nu_s}(\tau_1,\tau_2) f_{\nu_s}^{XX} (\sigma_1,\sigma_2)   + \nonumber \\
& \frac{1}{T^3} \sum_{s=0}^{T-1} W^{(T)}(\omega - \nu_s) W^{(T)}(\omega - \nu_s) \times q_{\nu_s}(\tau_1,\tau_2) \vartheta_{2,\nu_r,\nu_s}(\sigma_1,\sigma_2) + \nonumber \\
&  \frac{1}{T^3} \sum_{s=0}^{T-1} W^{(T)}(\omega - \nu_s) W^{(T)}(\omega - \nu_s) \times \vartheta_{1,\nu_r,\nu_s}(\tau_1,\tau_2)  f_{\nu_s}^{XX} (\sigma_1,\sigma_2)  + \nonumber \\
&\frac{1}{T^4} \sum_{s=0}^{T-1} W^{(T)}(\omega - \nu_s) W^{(T)}(\omega - \nu_r) \times \vartheta_{1,\nu_r,\nu_s}(\tau_1,\tau_2) \vartheta_{2,\nu_r,\nu_s}(\sigma_1,\sigma_2) \nonumber \\
&=\frac{1}{T^2} \sum_{s=0}^{T-1} W^{(T)}(\omega - \nu_s) W^{(T)}(\omega - \nu_s) q_{\nu_s}(\tau_1,\tau_2) f_{\nu_s}^{XX} (\sigma_1,\sigma_2)  + \frac{1}{T^{2} B_T^2} \vartheta_{1,q}(\tau_1,\tau_2) \odot \vartheta_{2,q}(\sigma_1,\sigma_2)  \label{D3} .
\end{align}
Replacing $\nu_s$ by $u_s$, and using Taylor's expansion for $f_{u_s}$ as in Lemma  \ref{SchNo12}, we obtain
\begin{align*}
f_{u_s}^{XX}(\sigma_1,\sigma_2)  &= f_{\omega}^{XX}(\sigma_1,\sigma_2)  + \frac{1}{1!} (u_s - \omega) f^{XX,(1)}_{\omega}(\sigma_1,\sigma_2)   + \frac{1}{2!} (u_s -\omega)^2 g_{2,u_s,\omega}(\sigma_1,\sigma_2).
\end{align*}
Thus, by Lemma \ref{SumW}, the first term of right hand side of \eqref{D3} becomes 
\begin{align*}
\frac{1}{T^2} \sum_{s=0}^{T-1} W^{(T)}(\omega - u_s) W^{(T)}(\omega - u_s) q_{0}(\tau_1,\tau_2) f_{\omega}^{XX}(\sigma_1,\sigma_2)  + \\
\frac{1}{T^2} \sum_{s=0}^{T-1} W^{(T)}(\omega - u_s) W^{(T)}(\omega - u_s) \times (u_s -\omega) \times q_{0}(\tau_1,\tau_2) f^{XX,(1)}_{\omega}(\sigma_1,\sigma_2)    + \\
 \frac{1}{2T^2} \sum_{s=0}^{T-1} W^{(T)}(\omega - u_s) W^{(T)}(\omega - u_s)  \times (u_s -\omega)^2 q_0(\tau_1,\tau_2) g_{2,u_s,\omega}(\sigma_1,\sigma_2) \\
= O\big(T^{-1} B_T^{-1}  \big) \times q_0(\tau_1,\tau_2) f_{\omega}^{XX}(\sigma_1,\sigma_2)  + O\big(T^{-1} B_T\big)  \times \vartheta_{1,g}(\tau_1,\tau_2) \odot \vartheta_{2,g}(\sigma_1,\sigma_2).
\end{align*}
For each $i$ and $j$ we compute
\begin{align*}
&\int_{[0,1]^4} q_0(\tau_1,\tau_2) f_{\omega}^{XX}(\sigma_1,\sigma_2) \overline{\varphi_i^{\omega}(\tau_1)}\varphi_j^{\omega}(\sigma_1) \varphi_i^{\omega}(\tau_2) \overline{\varphi_j^{\omega}(\sigma_2)}  d\tau_1 d\sigma_1 d\tau_2 d\sigma_2 \\
&\qquad \qquad = \Big\langle q_0(\tau_1,\tau_2)  , \overline{\varphi_i^{\omega}(\tau_1)} \varphi_i^{\omega}(\tau_2)   \Big\rangle \times  \Big\langle  f_{\omega}^{XX}(\sigma_1,\sigma_2) ,   \varphi_j^{\omega}(\sigma_1) \overline{\varphi_j^{\omega}(\sigma_2)} \Big\rangle \\
&\qquad \qquad = q_{0,ii} \times \lambda_j.
\end{align*}
So, taking the sum over $i$ and $j$ as in \eqref{S13}, we deduce that 
\begin{align*}
O\big(T^{-1} B_T^{-1} \big) \sum_{i,j} \frac{q_{0,ii}^{\omega} \lambda_j}{(\lambda_j + \zeta_T)^2}  = O\big(T^{-1} B_T^{-1} \big) \Big\{ \sum_i q_{0,ii}\Big\} \times \Big\{\sum_j \frac{\lambda_j}{(\lambda_j + \zeta_T)^2}\Big\}  = B_T^{-1} O\Big(T^{-(2\beta -1)/(\alpha + 2\beta)} \Big).
\end{align*}
For $\vartheta_{1,q}(\tau_1,\tau_2) \odot \vartheta_{2,q}(\sigma_1,\sigma_2)$ and $\vartheta_{1,g}(\tau_1,\tau_2) \odot \vartheta_{2,g}(\sigma_1,\sigma_2)$, we now apply Lemma \ref{VPI}, and we obtain  $O\big(T^{-1} B_T \big) \zeta_T^{-2} = O\Big(T^{-(2\beta -1)/(\alpha + 2\beta)} \Big)$
\subsubsection*{Bounding $\scrS_{12}$} 
Let $D_2$ be the integral kernel of $\mathscr{D}_2$ and  expand $\mathscr{D}_2 = \sum_{i,j}D_{2,ij} \varphi_i^{\omega} \otimes \overline{\varphi_j^{\omega}}$. Similarly with \eqref{ExpDel}, 
\begin{align*}
&\mathbb{E} \hs \scrS_{12} \hs_2^2  
= \sum_{i,j} \frac{\mathbb{E} D_{2,ij}^2}{(\lambda_j^{\omega} + \zeta_T)^2}  \leq \frac{1}{\zeta_T^2} \sum_{i,j} \mathbb{E} D_{2,ij}^2\\
&D_2(\tau_1,\sigma_1) \overline{D_2(\tau_2,\sigma_2)} = \sum_{s,r=0}^{T-1} W^{(T)}(\omega - \nu_s) W^{(T)}(\omega -\nu_r) L_{\nu_s,T}(\tau_1) \widetilde{X}_{-\nu_s,T}(\sigma_1) L_{-\nu_r,T}(\tau_2) \widetilde{X}_{\nu_r,T}(\sigma_2).
\end{align*}
For simplicity, denote
 \begin{align*}
U_s =  L_{\nu_s,T}(\tau_1); \qquad V_{-s} = \widetilde{X}_{-\nu_s,T}(\sigma_1); \qquad U_{-r} = L_{-\nu_r,T}(\tau_2); \qquad V_{r} = \widetilde{X}_{\nu_r,T}(\sigma_2).
 \end{align*}
Using the fourth order cumulant equation, we have
\begin{align*}
\mathbb{E} \big( U_s V_{-s} U_{-r} V_{r} \big)  = \mathbb{E} \big( U_s U_{-r}\big)   \mathbb{E} \big( V_{-s} V_r\big) 
+ \mathbb{E} \big(U_s V_{-s}\big)   \mathbb{E} \big( U_{-r} V_r \big) + \mathbb{E} \big( U_s V_r\big)    \mathbb{E} \big( U_{-r}V_{-s}\big) \\
+  \quad {\cum}\big( U_s,V_r,U_{-s},V_{-r}\big).
\end{align*}
We first estimate $\mathbb{E} \big(U_s V_{-s}\big)   \mathbb{E} \big( U_{-r} V_r \big) + \mathbb{E} \big( U_s V_r\big)    \mathbb{E} \big( U_{-r}V_{-s}\big)$.
For each term,
\begin{align}
\mathbb{E} \Big(L_{\omega_1,T}\otimes \widetilde{X}_{\omega_2,T} \Big) &=  \frac{1}{T} \sum_{u,v,t} \big\{h_T(v) - h_T(u+v) \big\}\times h_T(t) \times  \B_u  \mathbb{E}\Big(X_v \otimes X_t\Big) \exp\big\{-\mathbf{i} (\omega_1 (u+v)  +\omega_2 t) \big\} \nonumber \\
&=\frac{1}{T} \sum_{u,v,t} \big\{h_T(v) - h_T(u+v) \big\}\times h_T(t) \times  \B_u \mathscr{R}_{t-v} \exp\big\{-\mathbf{i} (\omega_1 (u+v) + \omega_2t ) \big\} \label{uvt}
\end{align}
For each $u$ there are at most $2|u|$ different values of $v$ such that  $h_T(v) - h_T(u+v) = \pm 1$. With $t$ ranging from 0 to $T-1$, the multiplicity of the term $\B_u \mathscr{R}_w$ in \eqref{uvt} is no more than 2 $|u|$. Therefore, 
\begin{align*}
\hs \mathbb{E} \big(L_{\omega_1,T} \otimes  \widetilde{X}_{\omega_2,T} \big)\hs_1 &\leq \frac{1}{T} \sum_{u,w}2\big|u \big| \times \hs \B_u \mathscr{R}_w \hs_1 \leq \frac{2}{T} \sum_{u,w} \big|u \big| \times \hs \B_u\hs_2 \times \hs \mathscr{R}_w \hs_2   \\
&\leq \frac{2}{T}  \Big\{ \sum_u |u| \times \hs \B_u \hs_2 \Big\} \times \Big\{\sum_w \hs \mathscr{R}_w \hs_2 \Big\} = O(T^{-1}).
\end{align*}
And, consequently, 
\begin{align*}
\frac{1}{T^2}\sum_{s,r=0}^{T-1} W^{(T)}(\omega - \nu_s) W^{(T)}(\omega - \nu_r) \mathbb{E}(U_s V_r) \mathbb{E}(U_{-r} V_{-s}) = T^{-2} B_T^{-2} \vartheta_{1,uv}(\tau_1,\sigma_2)\odot  \vartheta_{2,uv}(\sigma_1,\tau_2).
\end{align*}
Applying Lemma \ref{VPI}, we obtain an upper bound of the same order.
Similarly, we have
\begin{align*}
\frac{1}{T^2}\sum_{s,r=0}^{T-1} W^{(T)}(\omega - \nu_s) W^{(T)}(\omega - \nu_r) \mathbb{E}(U_s V_{-s}) \mathbb{E}(U_{-r} V_{r}) = T^{-2} B_T^{-2}\vartheta_{3,uv}(\tau_1,\tau_2)\odot  \vartheta_{4,uv}(\sigma_1,\sigma_2).
\end{align*}
Again, Lemma \ref{VPI} yields an upper bound of the same order.

For $\mathbb{E}(U_s U_{-r}) \mathbb{E}(V_r V_{-s})$, we expand $\mathbb{E}\big(L_{\omega_1,T} \otimes  L_{\omega_2,T} \big)$
\begin{align*}
\mathbb{E}\big(L_{\omega_1,T} \otimes L_{\omega_2,T}\big) = \frac{1}{T}  \sum_{u_1,v_1,u_2,v_2} \big\{h_T(v_1) - h_T(u_1 + v_1) \big\}
 \times \big\{h_T(v_2) - h_T(u_2 + v_2) \big\} \\
 \times \exp\big\{-\mathbf{i} \big(\omega_1(u_1 +v_1) + \omega_2(u_2 + v_2) \big) \big\}  
\times\B_{u_1} \mathscr{R}_{v_1 -v_2} \B_{u_2}^{*}.
\end{align*}
For $u_1$ and $u_2$  fixed, there are at most $2|u_1|$ and $2|u_2|$ values of $v_1$ and $v_2$, respectively, such that $h_T(v_1) - h_T(u_1 + v_1) , h_T(v_2) - h_T(u_2 + v_2)  = \pm 1$. Therefore,
\begin{align*}
 \hs \mathbb{E}\big(L_{\omega_1,T} \otimes  L_{\omega_2,T} \big) \hs_1 &\leq \frac{4}{T}\sum_{u_1,u_2} \big|u_1 \big| \times \big| u_2 \big|  \times \sum_v \hs \B_{u_1} \mathscr{R}_{v} \B_{u_2}^{*}\hs_1 \\
&\leq \frac{4}{T} \Big\{\sum_{u} |u| \times \hs \B_u\hs_2 \Big\}^2 \times \Big\{\sum_v \hs \mathscr{R}_v \hs_2 \Big\}   .
\end{align*}
But Proposition \ref{2.6} implies that
\begin{align*}
 \mathbb{E} \big( \widetilde{X}_{-\nu_s}^{(T)}(\sigma_1) \widetilde{X}_{\nu_r}^{(T)}(\sigma_2)\big) = \begin{cases} O\big(T^{-1} \big) \vartheta_{\nu_r,\nu_s}(\sigma_1,\sigma_2) &\qquad {\rm if }   \qquad r\neq s  \\
f_{\nu_s}^{XX}(\sigma_1,\sigma_2) + O(T^{-1}) \vartheta_{\nu_s}(\sigma_1,\sigma_2)  & \qquad {\rm if }   \qquad r = s .\end{cases}
\end{align*}
Therefore,
\begin{align*}
\frac{1}{T^2}\sum_{s,r=0}^{T-1} W^{(T)}(\omega - \nu_s) W^{(T)}(\omega - \nu_r) \mathbb{E}(U_s U_{-r}) \mathbb{E}(V_r V_{-s}) \\
= \frac{1}{T^2} \sum_{s=0}^{T-1} W^{(T)}(\omega - \nu_s)^2 \big\{f_{\nu_s} (\sigma_1,\sigma_2) + O(T^{-1}) \vartheta_{\nu_s}(\sigma_1,\sigma_2)  \big\} \times  \mathbb{E}(U_s U_{-s})  + \\
\frac{1}{T^2}\sum_{s\neq r=0}^{T-1} W^{(T)}(\omega - \nu_s) W^{(T)}(\omega - \nu_r) O(T^{-1}) \vartheta_{\nu_r,\nu_s}(\sigma_1,\sigma_2) \mathbb{E}(U_s U_{-r}) \\
 = T^{-2} B_T^{-2} \vartheta_{5,uv}(\tau_1,\tau_2) \odot \vartheta_{6,uv}(\sigma_1,\sigma_2).
\end{align*}
Applying Lemma \ref{VPI},  we obtain the desired rate. The last term in $\scrS_{12}$ is the cumulant term
\begin{align*}
&cum\Big(L_{\nu_s}^{(T)}, \widetilde{X}_{-\nu_s}^{(T)}, L_{-\nu_r}^{(T)},\widetilde{X}_{\nu_r}^{(T)} \Big)  \\
&= \frac{1}{T^2}cum \Bigg( \sum_{u_1,v_1} \big\{h_T(v_1) - h_T(u_1 + v_1) \big\} \B_{u_1} X_{v_1} \exp\big\{-\mathbf{i} \nu_s (u_1 + v_1) \big\}  ,  \sum_{v_2} h_T(v_2) X_{v_2} \exp\big\{\mathbf{i}\nu_s v_2 \big\} ,\\
& \qquad \qquad \qquad \sum_{u_3,v_3} \big\{h_T(v_3) - h_T(u_3 + v_3) \big\} \B_{u_3} X_{v_3} \exp \big\{\mathbf{i} \nu_r (u_3 + v_3)) \big\} , \sum_{v_4} h_T(v_4)  X_{v_4} \exp\big\{-\mathbf{i} \nu_r v_4 \big\}\Bigg)
\end{align*}
For fixed $u_1$ and $u_3$, denote
\begin{align*}
L(u_1,u_3) : =   \sum_{v_1,v_2,v_3,v_4} \Big\{h_T(v_1)  - h_T(u_1 + v_1)\Big\} \times \Big\{h_T(v_3) - h_T(u_3 + v_3) \Big\} \times h_T(v_2) \times h_T(v_4) \times \\
 \exp \Big\{-\mathbf{i} \big[\nu_s(u_1 + v_1)  - \nu_s v_2 - \nu_r(u_3 + v_3) + \nu_r v_4   \big] \Big\}   \times 
 {\cum} \Big( \B_{u_1}X_{v_1}, X_{v_2}, \B_{u_3}X_{v_3},X_{v_4} \Big).
\end{align*}
Then,
\begin{align*}
cum\Big(L_{\nu_s}^{(T)}, \widetilde{X}_{-\nu_s}^{(T)}, L_{-\nu_r}^{(T)},\widetilde{X}_{\nu_r}^{(T)} \Big) =\frac{1}{T^2} \sum_{u_1,u_3\in \mathbb{Z}}  L(u_1,u_3).
\end{align*}
By the multilinearity of cumulants,
\begin{align*}
cum \Big( \big(\B_{u_1}X_{v_1}\big)(\tau_1), X_{v_2}(\tau_2), \big(\B_{u_3}X_{v_3}\big)(\sigma_1) ,X_{v_4}(\sigma_2)  \Big) = \\
cum \Big( \int_0^1 b_{u_1}(\tau_1,\varsigma_1)X_{v_1}(\varsigma_1) d\varsigma_1 , X_{v_2}(\tau_2), \int_0^1 b_{u_3}(\sigma_1,\varsigma_3)X_{v_3}(\varsigma_3) d\varsigma_3 ,X_{v_4}(\sigma_2)  \Big) =\\
\int_0^1 b_{u_1}(\tau_1,\varsigma_1) \int_0^1 b_{u_3}(\sigma_1,\varsigma_3) {\cum} \Big( X_{v_1}(\varsigma_1)  , X_{v_2}(\tau_2), X_{v_3}(\varsigma_3) d ,X_{v_4}(\sigma_2)  \Big)d\varsigma_1 d\varsigma_3.
\end{align*}
Denote the above term by  $B_{u_1}{\cum} \Big(X_{v_1},X_{v_2},X_{v_3},X_{v_4} \Big) B_{u_3}^*$.
Replace $v_i = t_i + v_4$ for $i=1,2,3,4$, then $t_4 = 0$. Now the exponential factor in $L(u_1,u_3)$ is
\begin{align*}
\exp \Big\{-\mathbf{i} \big[\nu_s(u_1 + v_1)  - \nu_s v_2 - \nu_r(u_3 + v_3) + \nu_s v_4   \big] \Big\} \\
= \exp\Big\{-\mathbf{i} \big[\nu_s u_1 - \nu_r u_3\big] \Big\} \times \exp \Big\{-\mathbf{i}\big[\nu_s t_1 - \nu_s t_2 - \nu_r t_3 \big] \Big\}
 \times \exp\Big\{ -\mathbf{i} v_4 \times 0\Big\}.
\end{align*}
For the $h_T$ factor, denote 
\begin{align*}
H_T(t_1,t_2,t_3,v_4) := \Big\{ h_T(t_1 + v_4) - h_T(u_1 + t_1+v_4) \Big\} \times \Big\{h_T(t_3 + v_4) - h_T(u_3 + t_3 + v_4) \Big\}  \times h_T(t_2 + v_4) \times h_T(v_4) .
\end{align*}
Then
\begin{align*}
L(u_1,u_3) =& \sum_{t_1,t_2,t_3,v_4} H_T(t_1,t_2,t_3,v_4) \times \exp\Big\{-\mathbf{i} \big(\nu_s u_1 - \nu_r u_3\big) \Big\}  \\
&  \times \exp \Big\{-\mathbf{i}\big[\nu_s t_1 - \nu_s t_2 - \nu_r t_3 \big] \Big\} \times B_{u_1}  {\cum} \Big(X_{t_1},X_{t_2},X_{t_3},X_0 \Big)B_{u_3}^{*}.
\end{align*}
Note that the number of $v_4$ such that $h_T(v_4) = 1$ is $T$.  Letting $\mathcal{L}(u_1,u_3)$ be the integral operator of $L(u_1,u_3)$, we have
 \begin{align*}
\hs \mathcal{L}(u_1,u_3) \hs_1 \leq T \times \sum_{t_1,t_2,t_3} \hs \B_{u_1} \hs_2 
 \hs cum (X_{t_1},X_{t_2},X_{t_3},X_0)\hs_2  \hs \B_{u_3}\hs_2.
 \end{align*}
By assumptions (B4) and (B6),
\begin{align*}
cum\Big(L_{\nu_s}^{(T)}, \widetilde{X}_{-\nu_s}^{(T)}, L_{-\nu_r}^{(T)},\widetilde{X}_{\nu_r}^{(T)} \Big) =  \frac{1}{T^2} \sum_{u_1,u_3} L(u_1,u_3) = O(T^{-1}).
\end{align*}
Similarly to previous steps, applying Lemma \ref{VPI} and Proposition \ref{4cumS1} yields
\begin{align*}
 O\Big(T^{-1}\zeta_T^{-2} \Big)  = \frac{1}{B_T}O\Big(T^{-(2\beta -1)/(\alpha + 2\beta)} \Big).
\end{align*}
\subsubsection*{Bounding $\scrS_{11}$}
The steps in this case are similar to those involved in bounding $\scrS_3$. Recall that
\begin{align*}
\mathscr{D}_1 - \F_{\omega}^{YX} = \frac{1}{T} \sum_{s=0}^{T-1} W^{(T)}(\omega -\nu_s) \F_{\nu_s}^B  \widetilde{X}^{(T)}_{\nu_s} \otimes \widetilde{X}^{(T)}_{-\nu_s}  -\F_{\omega}^{YX}  = \frac{1}{T} \sum_{s=0}^{T-1} W^{(T)}(\omega -\nu_s) \F_{\nu_s}^B  \mathscr{P}^{XX}_{\nu_s,T}   -\F_{\omega}^{B} \F_{\omega}^{XX}.
\end{align*}
 We first work with $\mathbb{E} \big[ (\mathscr{D}_1 -\F_{\omega}^{YX}) \otimes (\mathscr{D}_1 - \F_{\omega}^{YX})^{*}\big]$. We have 
\begin{align*}
\mathbb{E} \big[ (\mathscr{D}_1 -\F_{\omega}^{YX})\otimes (\mathscr{D}_1 - \F_{\omega}^{YX})^{*}\big] =  \mathbb{E} \big[ \mathscr{D}_1 \otimes \mathscr{D}_1^* \big]  - \mathbb{E} \mathscr{D}_1 \otimes  \mathbb{E} \mathscr{D}_1^* + (\mathbb{E} \mathscr{D}_1 - \F_{\omega}^{YX} ) \otimes  (\mathbb{E} \mathscr{D}_1 - \F_{\omega}^{YX} )^*.
\end{align*}
To determine $\mathbb{E} \mathscr{D}_1$, we use Proposition \ref{2.6}
\begin{align*}
\mathbb{E}\Big[ \frac{1}{T} \sum_{s=0}^{T-1} W^{(T)}(\omega -\nu_s) \F_{\nu_s}^B \mathscr{P}_{\nu_s,T}^{XX} \Big]=
  \frac{1}{T} \sum_{s=0}^{T-1} W^{(T)}(\omega -\nu_s) \F_{\nu_s}^B \mathbb{E}\Big[\mathscr{P}_{\nu_s,T}^{XX} \Big] \\
=  \frac{1}{T} \sum_{s=0}^{T-1} W^{(T)}(\omega -\nu_s) \F_{\nu_s}^B  \Big\{\F_{\nu_s}^{XX} + \frac{1}{T} \V_{1,\nu_s} \Big\}  = \frac{1}{T} \sum_{s=0}^{T-1} W^{(T)}(\omega -\nu_s) \F_{\nu_s}^B \F_{\nu_s}^{XX} + \frac{1}{TB_T} \V_{1,\nu},
\end{align*}
with $\hs \V_{1,\nu} \hs_1  = O(1)$ uniformly over $\nu \in [0,2\pi]$. Taylor expanding, we have
\begin{align*}
f_{u_s}^{XX} &= f_{\omega}^{XX} + \sum_{j=1}^{p-1} \frac{(u_s -\omega)^j}{j!}  f_{\omega}^{XX,(j)}  + (u_s -\omega)^p g_{p,u_s,\omega} \\
f_{u_s}^B  &= f_{\omega}^B +  \sum_{j=1}^{p-1} \frac{(u_s-\omega)^{j}}{j!} \times f_{\omega}^{B,(j)}  + (u_s -\omega)^p g_{p,u_s,\omega}^B.
\end{align*}
Thus,
\begin{align*}
\frac{1}{T} \sum_{s=0}^{T-1} W^{(T)}(\omega -\nu_s) \F_{\nu_s}^B  \F_{\nu_s}^{XX} = \frac{1}{T} \sum_{s=0}^{T-1}  W^{(T)}(\omega -\nu_s) \Big\{ \sum_{j=0}^{p-1} \sum_{k=0}^{p-1} \F_{\omega}^{B,(k)} \F_{\omega}^{XX,(j)} \times (u_s -\omega)^{k+j}   + \V_{u_s,\omega} \times (u_s-\omega)^p \Big\}   .
\end{align*}
Note that $\hs \F_{\omega}^{B,(k)} \F_{\omega}^{XX,(j)} \hs_1 < \infty$. Using the same idea as in Proposition \ref{3.1}, we have
\begin{align*}
\frac{1}{T} \sum_{s=0}^{T-1} W^{(T)}(\omega -\nu_s) \F_{\nu_s}^B \F_{\nu_s}^{XX} = \F_{\omega}^B \F_{\omega}^{XX}  + \frac{1}{B_T T} \V_{\omega},
\end{align*}
with $\hs \V_{\omega} \hs_1  = O(1)$ uniformly over $\omega \in [0,2\pi]$. It follows that $\mathbb{E}\mathscr{D}_1 - \F_{\omega}^{B} \F_{\omega}^{XX}  = \frac{1}{TB_T} \V_{\omega}$.
Then, $\mathbb{E} \mathscr{D}_1 - \F_{\omega}^{YX} ) \otimes  (\mathbb{E} \mathscr{D}_1 - \F_{\omega}^{YX} )^*  = \frac{1}{T^2B_T^2} \V_{\omega} \otimes \V^*_{\omega}$. Finally, we apply Lemma $\ref{VPI}$,  which takes care of the term $\mathbb{E} \mathscr{D}_1-\F_{\omega}^{YX} $
Now we need to bound $\mathbb{E} \Big[D_1(\varsigma_1,\sigma_1)  \overline{D_1(\varsigma_2,\sigma_2)}\Big]$. This equals 
\begin{align}
&\frac{1}{T^2}\mathbb{E} \left[\sum_{s,r=0}^{T-1} W^{(T)}(\omega - \nu_s) W^{(T)}(\omega -\nu_r) \int_{[0,1]}f_{\nu_s}^B(\varsigma_1,\tau_1) p_{\nu_s}^{(T)}(\tau_1,\sigma_1) d\tau_1  \times \int_{[0,1]} f_{-\nu_r}^B(\varsigma_2,\tau_2) p_{-\nu_r}^{(T)}(\tau_2,\sigma_2) d\tau_2 \right]  \nonumber \\
&=\frac{1}{T^2} \sum_{s,r=0}^{T-1}   W^{(T)}(\omega - \nu_s) W^{(T)}(\omega -\nu_r) \int_{[0,1]^2 } f_{\nu_s}^B(\varsigma_1,\tau_1)   f_{-\nu_r}^B(\varsigma_2,\tau_2) \mathbb{E}\left[p_{\nu_s}^{(T)}(\tau_1,\sigma_1) p_{-\nu_r}^{(T)}(\tau_2,\sigma_2) \right] d\tau_1 d\tau_2. \label{CovD1}
\end{align}
By Proposition \ref{CovPerX}, 
\begin{align}
&\mathbb{E}\left[p_{\nu_s}^{(T)}(\tau_1,\sigma_1) p_{-\nu_r}^{(T)}(\tau_2,\sigma_2) \right]  = \nonumber \\
& \qquad \qquad  \mathbb{E} \Big[p_{\nu_s}^{(T)}(\tau_1,\sigma_1) \Big] \times \mathbb{E} \Big[ p_{-\nu_r}^{(T)}(\tau_2,\sigma_2)\Big] +  p_{r,s}^{(T)}(\tau_1,\sigma_1,\tau_2,\sigma_2) + \nonumber \\
& \qquad \qquad \eta(\nu_r - \nu_s) f_{\nu_s}^{XX} (\tau_1,\tau_2) f_{-\nu_s}^{XX}(\sigma_1,\sigma_2)+ \frac{1}{T} \eta(\nu_s- \nu_r) \vartheta_{1,\nu_s,\nu_r,f}(\tau_1,\tau_2) \odot\vartheta_{2,\nu_s,\nu_r,f}(\sigma_1,\sigma_2)  + \nonumber \\
& \qquad \qquad  \eta(\nu_s +\nu_r) f_{\nu_s}^{XX}(\tau_1,\sigma_2) f_{-\nu_s}^{XX}(\sigma_1,\tau_2) + \frac{1}{T}\eta(\nu_s+ \nu_r)  \vartheta_{3,\nu_s,\nu_r,f}(\sigma_1,\tau_2) \odot \vartheta_{4,\nu_s,\nu_r,f}(\tau_1,\sigma_2) + \nonumber \\
& \qquad \qquad \frac{1}{T^2}\vartheta_{1,\nu_s,\nu_r}(\tau_1,\sigma_1) \times \vartheta_{2,\nu_s,\nu_r}(\tau_2,\sigma_2) + \frac{1}{T^2}\vartheta_{3,\nu_s,\nu_r}(\tau_1,\sigma_2) \times \vartheta_{4,\nu_s,\nu_r}(\sigma_1,\tau_2) \nonumber \\
& \qquad \qquad = \mathbb{E} \Big[p_{\nu_s}^{(T)}(\tau_1,\sigma_1) \Big] \times \mathbb{E} \Big[ p_{-\nu_r}^{(T)}(\tau_2,\sigma_2)\Big]  + \mathbb{G}_{s,r}\label{CovPerX1}.
\end{align}
Moreover,
\begin{align}
&\mathbb{E}\big[D_1(\varsigma_1,\sigma_1) \big] \times  \mathbb{E}\big[\overline{D_1(\varsigma_2,\sigma_2) }\big] = \nonumber \\
&\frac{1}{T^2} \sum_{s,r=0}^{T-1}   W^{(T)}(\omega - \nu_s) W^{(T)}(\omega -\nu_r) \int_{[0,1]^2 }  f_{\nu_s}^B(\varsigma_1,\tau_1)   f_{-\nu_r}^B(\varsigma_2,\tau_2) \mathbb{E}\left[p_{\nu_s}^{(T)}(\tau_1,\sigma_1)\right] \times \mathbb{E}\left[ p_{-\nu_r}^{(T)}(\tau_2,\sigma_2) \right] d\tau_1 d\tau_2 
 \label{ExpD12}.
\end{align}
Combining \eqref{CovD1}, \eqref{CovPerX1} and \eqref{ExpD12}, we obtain
\begin{align*}
\mathbb{E} \Big[D_1(\varsigma_1,\sigma_1)  \overline{D_1(\varsigma_2,\sigma_2)}\Big]  - \mathbb{E} D_1(\varsigma_1,\sigma_1) \mathbb{E} \overline{D_1(\varsigma_2,\sigma_2)} = \\
\frac{1}{T^2} \sum_{s,r=0}^{T-1}   W^{(T)}(\omega - \nu_s) W^{(T)}(\omega -\nu_r) \int_{[0,1]^2 }  f_{\nu_s}^B(\varsigma_1,\tau_1)   f_{-\nu_r}^B(\varsigma_2,\tau_2) \times \mathbb{G}_{s,r} d\tau_1 d\tau_2.
\end{align*}
We must now consider each term resulting from the summands consituting $\mathbb{G}_{s,r}$.

First we begin with the summand $  f_{u_s}^{XX}(\tau_1,\tau_2) f_{-u_s}^{XX}(\sigma_1,\sigma_2) $ which contributes 
\begin{align*}
\frac{1}{T^2}\sum_{s=0}^{T-1} W^{(T)}(\omega -u_s) W^{(T)}(\omega -u_s)\int_{[0,1]^2 }  f_{u_s}^B(\varsigma_1,\tau_1)   f_{-u_s}^B(\varsigma_2,\tau_2)   f_{u_s}^{XX}(\tau_1,\tau_2) f_{-u_s}^{XX}(\sigma_1,\sigma_2) d\tau_1 d\tau_2.
\end{align*}
Taylor expanding yields
\begin{align*}
f_{u_s}^{B} &= f_{\omega}^{B} + (u_s -\omega) f_{\omega}^{B,(1)} + (u_s -\omega)^2 g_{2,u_s,\omega}^{B} \\
f_{u_s}^{XX} &=f_{\omega}^{XX} + (u_s -\omega) f_{\omega}^{XX,(1)} + (u_s -\omega)^2 g_{2,u_s,\omega}.
\end{align*}
Then, 
\begin{align*}
f_{u_s}^B f_{-u_s}^B f_{u_s}^{XX} f_{-u_s}^{XX} &= f_{\omega}^B f_{-\omega}^B f_{\omega} f_{-\omega}  + (u_s -\omega) \Big\{f_{\omega}^{B,(1)} f_{-\omega}^{B} f_{\omega}^{XX} f_{-\omega}^{XX} +  f_{\omega}^B f_{-\omega}^{B,(1)} f_{\omega}^{XX} f_{-\omega}^{XX}  +  \\
& \qquad f_{\omega}^B f_{-\omega}^B f_{\omega}^{XX,(1)} f_{-\omega}^{XX} + f_{\omega}^B f_{-\omega}^B f_{\omega}^{XX} f_{-\omega}^{XX,(1)} \Big\} 
 + (u_s -\omega)^2 \vartheta_{1,u_s} \odot \vartheta_{2,u_s}.
\end{align*}
These further terms are treated individually in the following bullet points:

\noindent $\bullet $ $f_{\omega}^B f_{-\omega}^B f_{\omega}^{XX} f_{-\omega}^{XX}$:  
Recall that
$$f_{\omega}^B(\varsigma,\tau) = \sum_{k,l}b_{kl}^{\omega} \varphi_k^{\omega}(\varsigma) \overline{\varphi_l^{\omega}(\tau)}.$$
Then,
\begin{align*}
&\frac{1}{T^2}\sum_{s=0}^{T-1} W^{(T)}(\omega -u_s) ^2 \int_{[0,1]^2} f_{\omega}^B(\varsigma_1,\tau_1)   f_{-\omega}^B(\varsigma_2,\tau_2)   f_{\omega}^{XX}(\tau_1,\tau_2) f_{-\omega}^{XX}(\sigma_1,\sigma_2) d\tau_1 d\tau_2\\
&=\frac{O(1)}{TB_T}\int_{[0,1]^2} \Big\{\sum_{k,l} b_{kl}^{\omega} \varphi_{k}^{\omega}(\varsigma_1) \overline{\varphi_{l}^{\omega}(\tau_1)} \Big\} \Big\{ \sum_{l}\lambda_l^{\omega}\varphi_l^{\omega}(\tau_1) \overline{\varphi_l^{\omega}(\tau_2)}\Big\} \Big\{\sum_{m,l} \overline{b_{ml}^{\omega}} \varphi_l^{\omega}(\tau_2) \overline{\varphi_m^{\omega}(\varsigma_2}) \Big\} \times  f_{-\omega}^{XX}(\sigma_1,\sigma_2) d\tau_1 d\tau_2\\
&=\frac{O(1)}{TB_T}\Big\{ \sum_{k,l,m} b_{kl}^{\omega} \overline{b_{ml}^{\omega}} \lambda_l^{\omega} \varphi_{k}^{\omega}(\varsigma_1)\overline{ \varphi_m^{\omega}(\varsigma_2) }\Big\}\times f_{-\omega}^{XX}(\sigma_1,\sigma_2).
\end{align*}
Multiplying by  $ \overline{\varphi_{i}^{\omega}(\varsigma_1)} \varphi_{i}^{\omega}(\varsigma_2) \varphi_{j}^{\omega}(\sigma_1) \overline{\varphi_{j}^{\omega}(\sigma_2)}$  and then integrating, we have 
\begin{align*}
\int_{[0,1]^2} \Big\{ \sum_{k,m,l} b_{kl}^{\omega} \overline{b_{ml}^{\omega}} \lambda_l^{\omega} \varphi_{k}^{\omega}(\varsigma_1)\overline{ \varphi_m^{\omega}(\varsigma_2) }\Big\} \overline{\varphi_{i}^{\omega}(\varsigma_1)} \varphi_{i}^{\omega}(\varsigma_2)  d\varsigma_1 d\varsigma_2 \int_{[0,1]^2}     f_{-\omega}^{XX}(\sigma_1,\sigma_2) \varphi_{j}^{\omega}(\sigma_1) \overline{\varphi_{j}^{\omega}(\sigma_2)}  d\sigma_1 d\sigma_2 \\
= \sum_l |b_{i l}^{\omega}|^2  \lambda_l^{\omega} \times \lambda_{j}^{\omega} .
\end{align*}
As in \eqref{ExpDel}, dividing by $(\lambda_{j} + \zeta_T)^2$ and taking the sum over $i$ and $j$ yields
\begin{align*}
\frac{1}{TB_T} \sum_{i,j} \sum_{l}  \frac{|b_{i l}^{\omega}|^2  \lambda_l^{\omega} \times \lambda_{j}^{\omega}}{(\lambda_{j}^{\omega}  + \zeta_T)^2} =  \sum_{i,l} |b_{il}^{\omega}|^2 \lambda_l^{\omega}  \frac{1}{B_T T} \sum_{j} \frac{\lambda_{j}^{\omega}}{(\lambda_{j}^{\omega}  + \zeta_T)^2} \leq    \sum_{l} l^{-2\beta - \alpha} \frac{O(1)}{B_T} T^{-(2\beta -1)/(2\beta + \alpha)}  \\
= \frac{O(1)}{B_T} T^{-(2\beta -1)/(2\beta + \alpha)}.
\end{align*}
\noindent  $\bullet$ $f_{\omega}^{B,(1)} f_{-\omega}^{B} f_{\omega}^{XX} f_{-\omega}^{XX} +  f_{\omega}^B f_{-\omega}^{B,(1)} f_{\omega}^{XX} f_{-\omega}^{XX}  +  f_{\omega}^B f_{-\omega}^B f_{\omega}^{XX,(1)} f_{-\omega}^{XX} + f_{\omega}^B f_{-\omega}^B f_{\omega}^{XX} f_{-\omega}^{XX,(1)}$: Recall the result in Lemma \ref{SumW}, which states that
\begin{align*}
\frac{1}{T} \sum_{s=0}^{T-1} W^{(T)}(\omega -u_s)^2 \times (\omega -u_s)  = O\big(T^{-1} B_T^{-2} \big).
\end{align*}
Note that 
$$\int_{[0,1]^2}  f_{\omega}^{B,(1)}(\varsigma_1,\tau_1)  f_{-\omega}^{B}(\varsigma_2,\tau_2) f_{\omega}^{XX}(\tau_1,\tau_2) f_{-\omega}^{XX}(\sigma_1,\sigma_2)  d\tau_1 d\tau_2$$
can be written in the form of $\vartheta_1(\varsigma_1,\varsigma_2) \times \vartheta_2(\sigma_1,\sigma_2)$ so that their corresponding operators $\V_1$ ,$\V_2$ have finite nuclear norm.  Now we may apply Lemma \ref{VPI}.

\noindent $\bullet$ $(u_s -\omega)^2 \vartheta_{1,u_s} \odot \vartheta_{2,u_s}$: Recall the result in Lemma \ref{SumW}, stating that
\begin{align*}
\frac{1}{T} \sum_{s=0}^{T-1} W^{(T)}(\omega -u_s)^2 \times (\omega -u_s)^2  = O\big(B_T \big).
\end{align*}
It follows that $\frac{1}{T^2} \sum_{s=0}^{T-1} W^{(T)}(\omega -u_s)^2 (\omega -u_s)^2 \times  \vartheta_{1,u_s}(\varsigma_1,\varsigma_2) \odot \vartheta_{2,u_s}(\sigma_1,\sigma_2)$  has the form  $O(T^{-1}B_T) \vartheta_{1,u}(\varsigma_1,\varsigma_2) \odot \vartheta_{2,u}(\sigma_1,\sigma_2) $. We may now apply Lemma 
\ref{VPI} as in the previous part.

This concludes our treatment of the summand  $  f_{u_s}(\tau_1,\tau_2) f_{-u_s}(\sigma_1,\sigma_2) $ in $\mathbb{G}_{s,r}$ .

We move on to the summand $ f_{u_s}^{XX}(\tau_1,\sigma_2) f_{-u_s}^{XX}(\sigma_1,\tau_2) $ in $\mathbb{G}_{s,r}$. This contributes the term
\begin{align*}
\frac{1}{T^2}\sum_{s=0}^{T-1} W^{(T)}(\omega -u_s) W^{(T)}(\omega  + u_s)\int_{[0,1]^2 }  f_{u_s}^B(\varsigma_1,\tau_1)   f_{-u_s}^B(\varsigma_2,\tau_2)   f_{u_s}(\tau_1,\sigma_2) f_{-u_s}(\sigma_1,\tau_2) d\tau_1 d\tau_2
\end{align*}
We apply the same process as with the previous term $f_{u_s}^{XX}(\tau_1,\tau_2) f_{-u_s}^{XX}(\sigma_1,\sigma_2) $ in $\mathbb{G}_{s,r}$. This is done in the following bullet points:

\medskip

\noindent $\bullet$ $f_{\omega}^B f_{-\omega}^B f_{\omega}^{XX} f_{-\omega}^{XX}$: We start with the integral
\begin{align*}
\int_{[0,1]} f_{\omega}^B(\varsigma_1,\tau_1) f_{\omega}^{XX}(\tau_1,\sigma_2) d\tau_1 =  \int_{[0,1]}\Big\{\sum_{k,l} b_{kl}^{\omega} \varphi_k^{\omega}(\varsigma_1) \overline{\varphi_l^{\omega}(\tau_1) }  \Big\} \Big\{\sum_l \lambda_l^{\omega}\varphi_{l}^{\omega}(\tau_1) \overline{\varphi_{l}^{\omega}(\sigma_2)} \Big\} d\tau_1   \\
 = \sum_{k,l} b_{kl}^{\omega} \lambda_l^{\omega}\varphi_{k}^{\omega}(\varsigma_1) \overline{\varphi_l^{\omega}(\sigma_2)}\\
 \int_{[0,1]} f_{-\omega}^B(\varsigma_2,\tau_2) f_{-\omega}^{XX}(\sigma_1,\tau_2) d\tau_2 =  \int_{[0,1]}\Big\{\sum_{u,v} \overline{b_{uv}^{\omega}} \overline{\varphi_u^{\omega}(\varsigma_2)}  \varphi_v^{\omega}(\tau_2)  \Big\} \Big\{\sum_v \lambda_v^{\omega} \overline{\varphi_{v}^{\omega}(\sigma_1)} \varphi_{v}^{\omega}(\tau_2)\Big\} d\tau_2   \\
 = \sum_{u,v} \overline{b_{uv}^{\omega}} \lambda_{v'}^{\omega} \overline{\varphi_{u}^{\omega}(\varsigma_2)} \overline{\varphi_{v'}^{\omega}(\sigma_1)}\\
\end{align*}
Then,
\begin{align*}
\int_{[0,1]^2 }  f_{\omega}^B(\varsigma_1,\tau_1)   f_{-\omega}^B(\varsigma_2,\tau_2)   f_{\omega}^{XX}(\tau_1,\sigma_2) f_{-\omega}^{XX}(\sigma_1,\tau_2) d\tau_1 d\tau_2 =\\
 \Big\{\sum_{k,l} b_{kl}^{\omega} \lambda_l^{\omega} \varphi_k^{\omega}(\varsigma_1) \overline{\varphi_l^{\omega}(\sigma_2)}\Big\} \times \Big\{ \sum_{u,v} \overline{b_{uv}^{\omega}} \lambda_{v'}^{\omega} \overline{\varphi_u^{\omega}(\varsigma_2)} \overline{\varphi_{v'}^{\omega}(\sigma_1)}\Big\}.
\end{align*}
We multiply by $\overline{\varphi_i^{\omega}(\varsigma_1)} \varphi_j^{\omega}(\sigma_1) \varphi_i^{\omega}(\varsigma_2) \overline{\varphi_j^{\omega}(\sigma_2)}$ and integrate to obtain 
\begin{align*}
b_{ij'}^{\omega} \lambda_{j'}^{\omega} \overline{b_{ij'}^{\omega}} \lambda_j^{\omega}.
\end{align*}
Then,
\begin{align*}
\frac{1}{TB_T}\sum_{i,j} \frac{\big|b_{ij^{'}}^{\omega}\big|^2 \lambda_{j'}^{\omega}\lambda_j^{\omega}}{(\lambda_j^{\omega}  + \zeta_T)^2}  = \frac{1}{TB_T} \sum_j \frac{\sum_i |b_{ij'}^{\omega}|^2 \lambda_{j'}^{\omega} \lambda_j^{\omega}}{(\lambda_j^{\omega} + \zeta_T)^2} =  \frac{1}{TB_T}  \sum_j \frac{(j^{'})^{-2\beta-\alpha}  \lambda_j^{\omega}}{(\lambda_j^{\omega} + \zeta_T)^2} \\
= \frac{O(1)}{B_T} O\big(T^{-(2\beta -1)/(\alpha + 2\beta)} \big).
\end{align*}
\noindent $\bullet$ $f_{\omega}^{B,(1)} f_{-\omega}^{B} f_{\omega}^{XX} f_{-\omega}^{XX} +  f_{\omega}^B f_{-\omega}^{B,(1)} f_{\omega}^{XX} f_{-\omega}^{XX}  +  f_{\omega}^B f_{-\omega}^B f_{\omega}^{XX,(1)} f_{-\omega}^{XX} + f_{\omega}^B f_{-\omega}^B f_{\omega}^{XX} f_{-\omega}^{XX,(1)}$: Note that
\begin{align*}
\int_{[0,1]^2 } f_{\omega}^{B,(1)}(\varsigma_1,\tau_1) f_{-\omega}^{B}(\varsigma_2,\tau_2) f_{\omega}^{XX}(\tau_1,\sigma_2) f_{-\omega}^{XX}(\sigma_1,\tau_2) d\tau_1 d\tau_2
\end{align*}
can be rewritten via a form $\vartheta_1(\varsigma_1,\sigma_2) \times \vartheta_2(\sigma_1,\varsigma_2)$ with their correspoding operators have finite Schatten 1-norm $\hs \V_1\hs_1, \hs\V_2\hs_1 < C$. Applying Lemma \ref{SumW},
\begin{align*}
\frac{1}{T^2}\sum_{s=0}^{T-1} W^{(T)}(\omega -u_s) W^{(T)}(\omega + u_s) (\omega -u_s) = 1_{I_T}(\omega) O(T^{-1}),
\end{align*}
and then we obtain the bound
\begin{align*}
1_{I_T}(\omega) O\big(T^{-1}\big) \vartheta_1(\varsigma_1,\sigma_2) \vartheta_2(\sigma_1,\varsigma_2).
\end{align*}
We multiply this by $\overline{\varphi_i^{\omega}(\varsigma_1)} \varphi_j^{\omega}(\sigma_1) \varphi_i^{\omega}(\varsigma_2) \overline{\varphi_j^{\omega}(\sigma_2)}$ and integrate. Applying Lemma \ref{VPI}, we obtain a bound of order $O\big( T^{-1}\zeta_T^{-2} \big)$. Now we integrate over $\omega \in I_T$, obtaining an integral of order
$$O\big(T^{-1} \zeta_T^{-2}B_T \big)  = O\big(T^{-(2\beta -1)/(\alpha + 2\beta)} \big).$$

\noindent $\bullet$ $(u_s -\omega)^2 \vartheta_{1,u_s} \odot \vartheta_{2,u_s}$: The same argument is applied here, using Lemma \ref{VPI} and Lemma \ref{SumW}.

\medskip
\noindent Now we move on to the summand $\frac{1}{T} \eta(\nu_s- \nu_r) \vartheta_{1,\nu_s,\nu_r,f}(\tau_1,\tau_2) \odot \vartheta_{2,\nu_s,\nu_r,f}(\sigma_1,\sigma_2)$  of $\mathbb{G}_{s,r}$. Note that
\begin{align*}
\frac{1}{T^3} \sum_{s=0}^{T-1} W^{(T)}(\omega -\nu_s)^2 f_{\nu_s} f_{-\nu_s}^B \vartheta_{1,\nu_s,\nu_r,f} \odot \vartheta_{2,\nu_s,\nu_r,f}  = \frac{1}{T^2} \times \frac{1}{T} \sum_{s=0}^{T-1} W^{(T)}(\omega -\nu_s)^2 \vartheta_{1,\nu_s,b} \odot \vartheta_{2,\nu_s,b}.
\end{align*}
and the latter is $O\big(T^{-2} B_T^{-1}\big)$, uniformly over $s$. Similarly to our treatment of $\scrS_3$, we may apply Lemma \ref{VPI}.

\medskip
The same argument can be applied to the summand $\frac{1}{T}\eta(\nu_s+ \nu_r)  \vartheta_{3,\nu_s,\nu_r,f}(\sigma_1,\tau_2) \odot \vartheta_{4,\nu_s,\nu_r,f}(\tau_1,\sigma_2)$  of $\mathbb{G}_{s,r}$.

We thus move on to the summands $\vartheta_{1\nu_s,\nu_r} \times \vartheta_{2,\nu_s,\nu_r}$ and $\vartheta_{3,\nu_s,\nu_r} \times \vartheta_{4,\nu_s,\nu_r}$  of $\mathbb{G}_{s,r}$. The quantity
\begin{align*}
\frac{1}{T^4} \sum_{s,r=0}^{T-1} W^{(T)}(\omega - \nu_s)^2 \big[\vartheta_{1,\nu_s,\nu_r} \times \vartheta_{2,\nu_s,\nu_r}  + \vartheta_{3,\nu_s,\nu_r} \times  \vartheta_{4,\nu_s,\nu_r}\big].
\end{align*}
is uniformly of order $O\big(T^{-2} B_T^{-1} \big)$. Similarly with the etimation of  $\scrS_3$ we apply Lemma \ref{VPI}.

\medskip
Finally, we turn to the summand $p_{r,s}^{(T)}(\tau_1,\sigma_1,\tau_2,\sigma_2)$  of $\mathbb{G}_{s,r}$. We need to bound
\begin{align}
\frac{1}{T^2}\sum_{s=0}^{T-1} W^{(T)}(\omega -u_s) W^{(T)}(\omega -u_r)\int_{[0,1]^2 }  f_{u_s}^B(\varsigma_1,\tau_1)   f_{-u_r}^B(\varsigma_2,\tau_2)   p_{r,s}^{(T)}(\tau_1,\sigma_1,\tau_2,\sigma_2) d\tau_1 d\tau_2. \label{cum4B}
\end{align}
Let $\F_{u_s}^B \mathscr{P}_{r,s}^{(T)} \F_{-u_r}^{B}$ be the operator corresponding to the kernels in the integrand. Then,
\begin{align*}
\hs \F_{u_s}^B \mathscr{P}_{r,s}^{(T)} \F_{-u_r}^{B} \hs_1 \leq \hs \F_{u_s}^B \hs_2 \times \hs\mathscr{P}_{r,s}^{(T)}\hs_2 \times \hs \F_{u_s}^B\hs_2 = O(T^{-1}).
\end{align*}
Applying Lemma \ref{VPI} and Proposition \ref{4cumS1}, we obtain a bound of order $O(T^{-1} \zeta_T^{-2} ) = \frac{1}{B_T} O\big(T^{(2\beta-1)/(\alpha + 2\beta)} \big)$.
\subsection*{Bounding $\scrS_2$}
\begin{align*}
\scrS_2 :&= \Big( \widehat{\F}_{\omega,T}^{YX}- \F_{\omega}^{YX} \Big) \Big(\big[\F_{\omega}^{XX} + \zeta_T \mathscr{I}\big]^{-1} - \Big[\widehat{\F}^{XX}_{\omega} +   \zeta_T \mathscr{I} \Big]^{-1} \Big) \\
&= \Big( \widehat{\F}_{\omega,T}^{YX}- \F_{\omega}^{YX} \Big) \Big(\F_{\omega}^{XX} + \zeta_T \mathscr{I} \Big)^{-1} \Delta \Big(\F_{\omega}^{XX} + \zeta_T \mathscr{I} \Big)^{-1} \Big( \mathscr{I} + \Delta  \big[\F_{\omega}^{XX} + \zeta_T \mathscr{I} \big] ^{-1} \Big)^{-1} \\
& = \scrS_1  \Delta \Big(\F_{\omega}^{XX} + \zeta_T \mathscr{I} \Big)^{-1} \Big( \mathscr{I} + \Delta  \big[\F_{\omega}^{XX} + \zeta_T \mathscr{I} \big] ^{-1} \Big)^{-1}.
\end{align*}
The product of the { third (second) and fourth (third)} terms has finite  nuclear norm by our treatment of $\scrS_3$. The last term has finite nuclear norm on the set $G_T$. Then, we have the order of $T^{-(2\beta-1)/(\alpha + 2\beta)}$ on $G_T$ in the Hilbert-Schimdt norm. 

In conclusion, all terms have been shown to be bounded above by at most $\frac{1}{B_T}O\Big(T^{-(  2\beta - 1)/(\alpha + 2\beta)} \Big)$ on the set $G_T$, and the proof is complete.

\section{Appendix}\label{appendix}

The Appendix contains the proofs to several auxiliary results that are required in the proof of Theorem \ref{main_theorem}. 
\begin{lemma}\label{SCoij}
Let $\big\{\varphi_i \big\}$ be a complete orthonormal system of functions in $L^2\big([0,1]^2, \mathbb{C} \big)$ and $\phi(\tau,\sigma)$ be a random bivariate function in $L^2\big([0,1]^2,\mathbb{C}\big)$ with induced integral operator $\Phi$, then
\begin{align*}
(A)\qquad &\phi(\tau,\sigma) = \sum_{i,j} \phi_{ij} \varphi_i(\tau) \overline{\varphi_j(\sigma)}, a.s.,\\
(B)\qquad &\mathbb{E} \big|\phi_{ij}\big|^2 = \int_{[0,1]^2} \mathbb{E} \big[ \phi(\tau_1,\sigma_1) \overline{\phi(\tau_2,\sigma_2)}\big] \times  \overline{\varphi_i(\tau_1)} \varphi_j(\sigma_1) \varphi_i(\tau_2) \overline{\varphi_j(\sigma_2)} d\tau_1 d\sigma_1 d\tau_2 d\sigma_2 \\
(C)\qquad &\mathbb{E} \hs\Phi\hs_2^2 = \sum_{i,j}  \int_{[0,1]^2} \mathbb{E} \big[ \phi(\tau_1,\sigma_1) \overline{\phi(\tau_2,\sigma_2)}\big] \times  \overline{\varphi_i(\tau_1)} \varphi_j(\sigma_1)\varphi_i(\tau_2) \overline{ \varphi_j (\sigma_2)} d \tau_1 d\tau_2 d\sigma_1 d\sigma_2.
\end{align*} 
\end{lemma}
\begin{proof}[Proof]
Since  $\big\{\varphi_i \big\}$ is a complete orthogonal system, $\phi \in L_2$ and
\begin{align*}
\phi_{ij} &= \int_{[0,1]^2} \phi(\tau,\sigma) \overline{\varphi_i(\tau)} \varphi_j(\sigma) d\tau d\sigma,
\end{align*}
so that part (A) is proved. We have
\begin{align*}
\big|\phi_{ij}\big|^2 &= \int_{[0,1]^2} \phi(\tau_1,\sigma_1) \overline{\varphi_i(\tau_1)} \varphi_j(\sigma_1) d\tau_1 d\sigma_1 \int_{[0,1]^2} \overline{\phi(\tau_2,\sigma_2)} \varphi_i(\tau_2) \overline{\varphi_j(\sigma_2)} d\tau_2 d\sigma_2\\
& = \int_{[0,1]^4}  \phi(\tau_1,\sigma_1) \overline{\phi(\tau_2,\sigma_2)} \times \overline{\varphi_i(\tau_1)}  \varphi_j(\sigma_1)\varphi_i(\tau_2) \overline{\varphi_j(\sigma_2)} d\tau_1 d\sigma_1 d\tau_2 d\sigma_2 \\
\mathbb{E} \big|\phi_{ij}\big|^2 & = \int_{[0,1]^4} \mathbb{E} \Big[  \phi(\tau_1,\sigma_1) \overline{\phi(\tau_2,\sigma_2)}\Big] \times \overline{\varphi_i(\tau_1)} \varphi_j(\sigma_1) \varphi_i(\tau_2)\overline{\varphi_j(\sigma_2) } d\tau_1 d\sigma_1 d\tau_2 d\sigma_2,
\end{align*}
This proves part (B). Now,  by definition,
\begin{align*}
\mathbb{E} \hs\Phi \hs_2^2 = \mathbb{E} \sum_{i,j} \big| \phi_{ij}\big|^2 = \sum_{i,j} \int_{[0,1]^4} \mathbb{E} \big\{ \phi(\tau_1,\sigma_1) \overline{\phi(\tau_2,\sigma_2)}\big\} \times \overline{\varphi_i(\tau_1)} \varphi_i(\tau_2) \varphi_j(\sigma_1) \overline{\varphi_j (\sigma_2)}d \tau_1 d\tau_2 d\sigma_1 d\sigma_2,
\end{align*}
and
\begin{align*}
\int_{[0,1]^2} \phi(\tau,\sigma) \overline{\phi(\tau,\sigma)} d\tau d\sigma = \int_{[0,1]^2} \Big\{\sum_{i,j} \phi_{ij} \overline{\varphi_i(\tau)} \varphi_j(\sigma) \Big\} \times \Big\{ \sum_{k,l} \overline{\phi_{kl}} \varphi_k(\tau) \overline{\varphi_l(\sigma)} \Big\} d\tau d\sigma\\
= \int_{[0,1]^2} \sum_{i,j,k,l} \phi_{ij} \overline{\phi_{kl} } \times \overline{\varphi_i(\tau)} \varphi_k(\tau) \varphi_j(\sigma) \overline{\varphi_l(\sigma)} d\tau d\sigma = \sum_{i,j} \big|\phi_{ij}\big|^2,
\end{align*}
proving part (C).
\end{proof}

\begin{lemma}\label{VPI}
 Let $\{\varphi_i \}$ be a complete  orthonormal basis in $L^2([0,1],\mathbb{C})$ that is closed under conjugation (i.e. satisfying the condition $\{\varphi_i: i =1,2,\ldots \} = \{\overline{\varphi_i}: i =1,2,\ldots \}$). Let $\xi_1, \xi_2 \in L^2([0,1]^2,\mathbb{C})$  and $\xi_3 \in L^2([0,1]^4,\mathbb{C})$. Let $\U_i$ be the induced operator of $\xi_i$ for $i =1,2,3$. Then
\begin{align*}
&(A)\quad   \sum_{i,j} \left |\int_{[0,1]^4} \xi_1(\tau_1,\tau_2) \xi_2(\sigma_1,\sigma_2)  \times  \overline{\varphi_i(\tau_1)}\varphi_j(\sigma_1) \varphi_i(\tau_2)\overline{\varphi_j(\sigma_2)}  d\tau_1 d\sigma_1 d\tau_2 d\sigma_2 \right|  \leq \hs \U_1\hs_{1} \hs \U_2\hs_{1} \\
&(B)\quad  \sum_{i,j} \left |\int_{[0,1]^4} \xi_1(\tau_1,\sigma_2) \xi_2(\sigma_1,\tau_2)  \times  \overline{\varphi_i(\tau_1)}\varphi_j(\sigma_1) \varphi_i(\tau_2)\overline{\varphi_j(\sigma_2)}  d\tau_1 d\sigma_1 d\tau_2 d\sigma_2 \right|  \leq \hs\U_1\hs_{2}^2 + \hs\U_2\hs_{2}^2\\
&(C)\quad  \sum_{i,j} \left| \int_{[0,1]^4} \xi_3(\tau_1,\sigma_1,\tau_2,\sigma_2)  \times  \overline{\varphi_i(\tau_1)}\varphi_j(\sigma_1) \varphi_i(\tau_2)\overline{\varphi_j(\sigma_2)}  d\tau_1 d\sigma_1 d\tau_2 d\sigma_2\right| \leq \hs \U_3 \hs_1.
\end{align*}
\end{lemma}
\begin{proof}[Proof] 
Let $\xi_k(\tau,\sigma)= \sum_{i,j} \xi_{k,ij} \varphi_i(\tau) \overline{\varphi_j(\sigma)} $ for $k=1,2$.

\noindent \noindent (A) We start with 
\begin{align*}
&\int_{[0,1]^4} \xi_1(\tau_1,\tau_2) \times \xi_2(\sigma_1,\sigma_2) \times \overline{\varphi_i(\tau_1)}\varphi_j(\sigma_1) \varphi_i(\tau_2) \overline{\varphi_j(\sigma_2)}  d\tau_1 d\sigma_1 d\tau_2 d\sigma_2  \\
& \qquad \qquad  = \int_{[0,1]^2}  \xi_1(\tau_1,\tau_2)  \overline{\varphi_i(\tau_1)}  \varphi_i(\tau_2) d\tau_1 d\tau_2 \times \int_{[0,1]^2}  \xi_2(\sigma_1,\sigma_2)   \varphi_j(\sigma_1) \overline{\varphi_j(\sigma_2)} d\sigma_1 d\sigma_2  \\
&\qquad \qquad = \xi_{1,ii} \times \xi_{2,jj}.
\end{align*}
Then
\begin{align*}
\sum_{i,j} \big| \xi_{1,ii} \times \xi_{2,jj}\big| \leq \sum_i \big|\xi_{1,ii} \big| \times \sum_j \big| \xi_{2,jj} \big| \leq \hs\U_1\hs_{1} \hs\U_2\hs_{1}.
\end{align*}

\noindent (B) Recall that for each $i$, there exists only one $i'$ such that  $\overline{\varphi_i} =\varphi_{i^{'}}$. So,
\begin{align*}
&\int_{[0,1]^4} \xi_1(\tau_1,\sigma_2) \times \xi_2(\sigma_1,\tau_2) \times \overline{\varphi_i(\tau_1)}\varphi_j(\sigma_1) \varphi_i(\tau_2) \overline{\varphi_j(\sigma_2)}  d\tau_1 d\sigma_1 d\tau_2 d\sigma_2 \\
&\qquad = \int_{[0,1]^2} \xi_1(\tau_1,\sigma_2)   \overline{\varphi_i(\tau_1)}  \overline{\varphi_j(\sigma_2)} d\tau_1 d\sigma_2 \times  \int_{[0,1]^2}  \xi_2(\sigma_1,\tau_2)    \varphi_j(\sigma_1) \varphi_i(\tau_2) d\sigma_1 d\tau_2 \\
&\qquad  = \int_{[0,1]^2} \sum_{k,l}\xi_{1,kl} \varphi_k(\tau_1) \overline{\varphi_l(\sigma_2)}   \overline{\varphi_i(\tau_1)}  \overline{\varphi_j(\sigma_2)}d\tau_1 d\sigma_2 \times  \int_{[0,1]^2}  \sum_{u,v} \xi_{2,uv} \varphi_u(\sigma_1) \overline{\varphi_v(\tau_2)}   \varphi_j(\sigma_1) \varphi_i(\tau_2) d\sigma_1 d\tau_2 \\
& \qquad = \xi_{1,ij^{'}}  \times \xi_{2,j^{'}i}.
\end{align*}
Taking the sum over $i,j$ now yields
\begin{align*}
\sum_{i,j}\big| \xi_{1,ij^{'}}  \times \xi_{2,j^{'}i} \big| \leq \sum_{i,j} \big\{ \big| \xi_{1,ij}\big|^2 +  \big| \xi_{2,ij}\big| ^2 \big\} =  \hs \U_1\hs_{2}^2 + \hs \U_2\hs_{2}^2.
\end{align*}

\noindent (C) Since the $\varphi_i$ is a complete orthonormal basis in $L^2([0,1],\mathbb{C})$, the collection $\{ \overline{\varphi_i}  \varphi_j :i,j\}$ is a complete orthonormal basis in $L^2([0,1]^2,\mathbb{C})$. Writing $\overline{\varphi_i} \varphi_j = \varphi_{ij}$ we have
\begin{align*}
 \int_{[0,1]^4} \xi_3(\tau_1,\sigma_1,\tau_2,\sigma_2)  \times  \overline{\varphi_i(\tau_1)}\varphi_j(\sigma_1) \varphi_i(\tau_2)\overline{\varphi_j(\sigma_2)}  d\tau_1 d\sigma_1 d\tau_2 d\sigma_2 = \langle \U_3 \varphi_{ij}, \varphi_{ij} \rangle. 
\end{align*}
Taking the absolute value and summing over $i,j$, we get the upper bound $\hs\U_3 \hs_1$.
\end{proof}

\begin{lemma}\label{SchNo12}
The spectral density of $X$ and the cross-spectral density of $\{X,Y\}$ have the form
\begin{eqnarray*}
&f_{\nu}^{XX} = f_{\nu, \mathfrak{R}}^{XX} + \mathbf{i}  f_{\nu, \mathfrak{I}}^{XX} \\
&f_{\nu}^B = f_{\nu, \mathfrak{R}}^B + \mathbf{i}  f_{\nu, \mathfrak{I}}^B
\end{eqnarray*}
and  assume that they satisfy condition (B3) and (B4). Furthermore, for $\nu,\omega,\alpha \in [0,2\pi]$,  they admit the Taylor expansions
\begin{align*}
f_{\nu}^{XX} &= f_{\omega}^{XX} + \sum_{j=1}^{p-1} \frac{(\nu -\omega)^j}{j!} f_{\omega}^{XX,(j)}  + \big(\nu - \omega\big)^p g_{p,\nu,\omega}\\
f_{\nu}^B &= f_{\omega}^B + \sum_{j=1}^{p-1} \frac{(\nu-\omega)^j}{j!} f_{\omega}^{B,(j)} + (\nu-\omega)^{p} g_{p,\nu,\omega}^B,
\end{align*}
where $f_{\omega}^{XX,(j)} = \frac{\partial^j f_{\alpha}^{XX}}{\partial \alpha^j}\Big|_{\alpha = \omega}$, $f_{\omega}^{B,(j)} = \frac{\partial^j f_{\alpha}^B}{\partial \alpha^j}\Big|_{\alpha = \omega}$, and
\begin{align*}
\big(\nu - \omega\big)^p g_{p,\nu,\omega} &=  \int_{\omega}^{\nu} \frac{(\nu - \zeta)^p}{p!}f_{\zeta,\mathfrak{R}}^{XX,(p+1)} d\zeta + \mathbf{i}   \int_{\omega}^{\nu} \frac{(\nu - \zeta)^p}{p!}f_{\zeta,\mathfrak{I}}^{XX,(p+1)} d\zeta \\
(\nu - \omega)^p g_{p,\nu,\omega}^B &=  \int_{\omega}^{\nu} \frac{(\nu - \zeta)^p}{p!}f_{\zeta,\mathfrak{R}}^{B,(p+1)} d\zeta + \mathbf{i}   \int_{\omega}^{\nu} \frac{(\nu - \zeta)^p}{p!}f_{\zeta,\mathfrak{I}}^{B,(p+1)} d\zeta \\
f_{\omega}^{XX,(p+1)} &= f_{\zeta,\mathfrak{R}}^{XX,(p+1)} + \mathbf{i} f_{\zeta,\mathfrak{I}}^{XX,(p+1)} \\
f_{\omega}^{B,(p+1)} &= f_{\zeta,\mathfrak{R}}^{B,(p+1)} + \mathbf{i} f_{\zeta,\mathfrak{I}}^{B,(p+1)}.
\end{align*}
 Finally, there exists a constant $C$ that does not depend on $\nu,\omega, \alpha$ and such that
$\hs \F_{\omega}^{XX,(j)} \hs_1$, $\hs \F_{\omega}^{B,(j)} \hs_1$,  $\hs \G_{p,\nu,\omega}\hs_{1}$, $\hs \G_{p,\nu,\omega}^B\hs_{1}$, $\hs \F_{\omega}^{\epsilon}\hs_1$  are uniformly bounded by $C$.
Here $\F_{\omega}^{XX,(j)}$, $\F_{\omega}^{B,(j)}$,  $\G_{p,\nu,\omega}$, $\G_{p,\nu,\omega}^B$ and $\F_{\omega}^{\epsilon}$ are the operators induced by the kernels $f_{\omega}^{XX,(j)}, f_{\omega}^{B,(j)},  g_{p,\nu,\omega}, g_{p,\nu,\omega}^B$ and $f_{\omega}^{\epsilon}$.
\end{lemma}
\begin{proof}[Proof]
Recall that
\begin{align*}
f_{\omega}^{XX}(\tau,\sigma) = \sum_{t\in \mathbb{Z} } e^{-\mathbf{i} t\omega} r_t^X(\tau,\sigma).
\end{align*}
Since $f_{\omega}(\tau,\sigma)$ is a complex-valued function, $f_{\omega}^{XX}(\tau,\sigma) = f_{\omega, \mathfrak{R}}^{XX}(\tau,\sigma) + \mathbf{i} f_{\omega, \mathfrak{I}}^{XX}(\tau,\sigma)$.
Using a Taylor expansion of the functions $f_{\omega, \mathfrak{R}}^{XX}(\tau,\sigma)$  and  $f_{\omega, \mathfrak{I}}^{XX}(\tau,\sigma)$, we have
\begin{align*}
f_{\nu, \mathfrak{R}}^{XX}(\tau,\sigma) &=  f_{\omega, \mathfrak{R}}^{XX}(\tau,\sigma)  + \sum_{j=1}^{p-1} \frac{(\nu-\omega)^j}{j!}f_{\omega,\mathfrak{R}}^{XX,(j)} + \int_{\omega}^{\nu} \frac{(\nu - \zeta)^p}{p!}f_{\zeta,\mathfrak{R}}^{XX,(p+1)} d\zeta \\
f_{\nu, \mathfrak{I}}^{XX}(\tau,\sigma) &= f_{\omega, \mathfrak{I}}^{XX}(\tau,\sigma) + \sum_{j=1}^{p-1} \frac{(\nu-\omega)^j}{j!} f_{\omega,\mathfrak{I}}^{XX,(j)} +  \int_{\omega}^{\nu} \frac{(\nu - \zeta)^p}{p!}f_{\zeta,\mathfrak{I}}^{XX,(p+1)} d\zeta.
\end{align*}
Then
\begin{align*}
f_{\nu}^{XX} &= f_{\omega}^{XX} + \sum_{j=1}^{p-1} \frac{(\nu-\omega)^j}{j!}f_{\omega}^{XX,(j)}  +  \int_{\omega}^{\nu} \frac{(\nu - \zeta)^p}{p!}f_{\zeta,\mathfrak{R}}^{XX,(p+1)} d\zeta + \mathbf{i} \int_{\omega}^{\nu} \frac{(\nu - \zeta)^p}{p!}f_{\zeta,\mathfrak{I}}^{XX,(p+1)} d\zeta.
\end{align*}
 Under condition (B3),
\begin{align*}
f^{XX,(j)}_{\omega}(\tau,\sigma) &= \sum_{t \in \mathbb{Z} \backslash\{0\}} e^{-\mathbf{i} t\omega}  t^j r_t^X(\tau,\sigma)\\
f_{\zeta,\mathfrak{R}}^{XX,(p+1)}(\tau,\sigma) &= \begin{cases} \sum_{t \in \mathbb{Z} \backslash \{0\}} (-1)^{(p+1)/2} \cos(- t\zeta) (-t)^{p+1} r_t^X(\tau,\sigma) \qquad & \text{for $p$ odd }\\
 \sum_{t \in \mathbb{Z}\backslash\{0\}} (-1)^{(p+2)/2} \sin(- t\zeta) (-t)^{p+1} r_t^X(\tau,\sigma) \qquad &\text{for $p$ even}. 
 \end{cases}
\end{align*}
 Under condition (B3), for $p$ odd,
\begin{align*}
 \int_{\omega}^{\nu}   \frac{(\nu - \zeta)^p}{p!}f_{\zeta,\mathfrak{R}}^{XX,(p+1)}(\tau,\sigma) d\zeta &= \int_{\omega}^{\nu}  \sum_{t \in \mathbb{Z} \backslash \{0\}}  \frac{(\nu - \zeta)^p}{p!} (-1)^{(p+1)/2}\cos(- t\zeta) (-t)^{p+1} r_t^X(\tau,\sigma) d\zeta \\
&=  \int_{\omega}^{\nu}  \sum_{t \in \mathbb{Z} \backslash \{0 \}} (-1)^{(p+1)/2} \frac{(\nu - \zeta)^p}{p!} \frac{\cos(- t\zeta)}{t^4} \Big\{(-t)^{p+5} r_t^X(\tau,\sigma)\Big\} d\zeta \\
&= \sum_{t \in \mathbb{Z}\backslash \{0\}}  \int_{\omega}^{\nu} (-1)^{(p+1)/2}  \frac{(\nu - \zeta)^p}{p!} \frac{\cos(- t\zeta)}{t^4} d\zeta  \times  \Big\{(-t)^{p+5} r_t^X(\tau,\sigma)\Big\}.
\end{align*}
Under condition (B4), it follows that 
$$\hs  \F^{XX,(j)}_{\omega} \hs_1  \leq \sum_{t\in \mathbb{Z} \backslash \{0\}} |t|^j \hs \mathscr{R}_t^X \hs_1$$

 Given a system of  complete orthonormal functions  $\{e_n\}$
\begin{align*}
\sum_{n\in \mathbb{N} }  \left| \int_{[0,1]^2} \int_{\omega}^{\nu} \frac{(\nu - \zeta)^p}{p!}f_{\zeta,\mathfrak{R}}^{XX,(p+1)}(\tau,\sigma)  d\zeta \times e_n(\tau) \overline{e_n(\sigma)} d\tau d\sigma \right| = \\
 = \sum_{n\in \mathbb{N}} \left| \sum_{t\in \mathbb{Z}\backslash \{0\}}\int_{\omega}^{\nu} \frac{(\nu - \zeta)^p}{p!} \frac{\cos(-t\zeta)}{t^4}  d\zeta \times \Big\{ t^{p+5}  \int_{[0,1]^2} r_t^X(\tau,\sigma) e_n(\tau) \overline{e_n(\sigma)} d\tau d\sigma \Big\}   \right|.
\end{align*}
Denoting   $\int_{[0,1]^2} r_t^X(\tau,\sigma) e_n(\tau) \overline{e_n(\sigma)} d\tau d\sigma  = r_{t,n}^X$, the above term is bounded above by
\begin{align*}
 \sum_{n\in \mathbb{N}} \sum_{t \in \mathbb{Z}} |t|^{p+5} \times  |r_{t,n} | \int_{\omega}^{\nu} \frac{|\nu - \zeta|^p}{p!} d\zeta = O(1) \times |\nu - \omega|^p \sum_{t\in \mathbb{Z}} | t|^{p+5} \hs \mathscr{R}_t \hs_1 = O(1) |\nu -\omega|^p.
\end{align*}
Since the induced operator of $\int_{\omega}^{\nu} \frac{(\nu - \zeta)^p}{p!} f_{\zeta,\mathfrak{R}}^{XX,(p+1)}(\tau,\sigma) d\zeta$ is the limit of a sequence of compact operators, it is a compact operator.

Following the proof of Theorem 1.27 in \citet{Zhu07}, it follows that the induced operator of $\int_{\omega}^{\nu} \frac{(\nu - \zeta)^p}{p!}f_{\zeta,\mathfrak{R}}^{XX,(p+1)} d\zeta$ has nucelar norm of order $|\nu - \omega|^p$. Similar steps yield the same result for $p$ even.
We obtain the same result for the induced operator of $\int_{\omega}^{\nu} \frac{(\nu - \zeta)^p}{p!}f_{\zeta,\mathfrak{I}}^{XX,(p+1)} d\zeta$.  Then 
$\hs \G_{p,\nu,\omega} \hs_{1}$ is uniformly bounded.
The same method of proof is applied to $ \F_{\omega}^{B,(j)}$, $ \G_{\nu,\omega}^B$ and $\F_{\omega}^{\epsilon}$.
\end{proof}

\begin{lemma} \label{Slambdaj}
Let $\alpha$ and $\beta$ be two positive numbers as in  assumption (B1). Let  $\lambda_j =  j^{-\alpha}, b_j=  j^{-\beta}; \zeta_T = T^{-\alpha/(\alpha + 2\beta)}$ then
\begin{align}
&(A) \qquad \sum_{j=1}^{\infty} \zeta_T^2 \frac{b_j^2}{(\lambda_j + \zeta_T)^2} = O\Big(T^{-(2\beta -1)/(\alpha + 2\beta)} \Big) \label{seqA}\\
&(B) \qquad \sum_{j=1}^{\infty} \frac{1}{T} \frac{\lambda_j}{(\lambda_j + \zeta_T)^2} = O\Big(T^{-(2\beta -1)/(\alpha + 2\beta)} \Big) \label{seqB}\\
&(C) \qquad \sum_{j=1} ^{\infty} \frac{\lambda_j^2}{(\lambda_j + \zeta_T)^2} = O\Big( T^{1/(\alpha + 2\beta)}\Big) \label{seqC}.
\end{align}
\end{lemma}
\begin{proof} [Proof]
We use the following facts in the proof
\begin{align*}
\sum_{i=1}^{T} i^{t} &\asymp \begin{cases}
T^{t+1} \qquad &t > -1\\
\log n \qquad &t = -1\\
C \qquad &t < -1
\end{cases}\\
\sum_{i=T+1}^{\infty}i^{-t} &\asymp \quad T^{-t+1} \qquad t > 1.
\end{align*}
(A)  Let $J = T^{1/(\alpha + 2\beta)}$. Since $2\beta > 1$ and $2\beta - 2\alpha < 1$, using the above results we may write
\begin{align*}
 \sum_{j=1}^{\infty} \zeta_T^2 \frac{\lambda_j^2}{(\lambda_j + \zeta_T)^2} &\asymp \sum_{j=1}^{\infty} \zeta_T^2 \frac{j^{-2\beta}}{(j^{-\alpha} + \zeta_T)^2} \leq \sum_{j \leq J} j^{-2\beta + 2\alpha} \times \zeta_T^2 + \sum_{j > J} j^{-2\beta} \asymp J^{-2\beta + 2\alpha + 1} \zeta_T^2  + J^{-2\beta + 1} \\
&\asymp T^{(-2\beta + 2\alpha + 1)/(\alpha + 2\beta)} T^{-2\alpha/(\alpha + 2\beta)} + T^{-(2\beta -1)/(\alpha + 2\beta)}\\
&= O\Big(T^{-(2\beta -1)/(\alpha + 2\beta)} \Big) .
\end{align*}
(B) For $\alpha > 1$, 
\begin{align*}
\sum_{j=1}^{\infty} \frac{1}{T} \frac{\lambda_j}{(\lambda_j + \zeta_T)^2} &= \sum_{j \leq J} \frac{1}{T} \frac{\lambda_j}{(\lambda_j + \zeta_T)^2} + \sum_{j > J} \frac{1}{T} \frac{\lambda_j}{(\lambda_j + \zeta_T)^2} \leq \sum_{j \leq J} \frac{1}{T} \frac{1}{\lambda_j} + \sum_{j > J} \frac{1}{T} \frac{\lambda_j}{\zeta_T^2} \\
& \asymp \frac{1}{T}\sum_{j \leq J} j^{\alpha}  + \frac{1}{T} T^{-2\alpha/(\alpha + 2\beta)} \sum_{j > J} j^{-\alpha} \asymp \frac{1}{T} J^{\alpha + 1} + T^{-(2\beta - \alpha)(\alpha + 2\beta)} J^{-\alpha + 1}\\
&= T^{-1} T^{-(\alpha+1)/(\alpha + 2\beta)} + T^{-(2\beta - \alpha)/(\alpha + 2\beta)} T^{-(\alpha -1)/(\alpha + 2\beta)}\\
&= O\Big(T^{-(2\beta -1)/(\alpha + 2\beta)} \Big).
\end{align*}
(C) For $J = T^{1/(\alpha + 2\beta)}$
\begin{align*}
\sum_{j=1} ^{\infty} \frac{\lambda_j^2}{(\lambda_j + \zeta_T)^2} &= \sum_{j \leq J } \frac{\lambda_j^2}{(\lambda_j + \zeta_T)^2} + \sum_{j > J} \frac{\lambda_j^2}{(\lambda_j + \zeta_T)^2} \leq J + \sum_{j > J} \zeta_T^{-2}  j^{-2\alpha}  = T^{1/(\alpha + 2\beta)} + T^{2\alpha/(\alpha + 2\beta)} J^{-2\alpha + 1} \\
&\asymp T^{1/(\alpha + 2\beta)} + T^{2\alpha/(\alpha + 2\beta)} T^{(-2\alpha + 1)/(\alpha + 2\beta)}\\
&\asymp T^{1/(\alpha + 2\beta)}.
\end{align*}
\end{proof}
\begin{lemma}\label{SumW}
For a fixed $\omega \in [0,2\pi]$ and a non-negative integer $k$,
\begin{align*}
&(A) \qquad \frac{1}{T} \sum_{s=0}^{T-1} W^{(T)}(\omega -u_s) \times (\omega -u_s)^j =  \delta_{0j} + O\big(T^{-1} B_T^{-1} \big) \qquad {\rm for} \quad 0 \leq  j < p \\
&(B) \qquad \frac{1}{T} \sum_{s=0}^{T-1} \Big|W^{(T)}(\omega -u_s) \Big| \times \Big|\omega - u_s \Big|^{p} = O\big(B_T^p \big)  +  O\big(T^{-1} B_T^{-1} \big). \\
&(C) \qquad \frac{1}{T} \sum_{s=0}^{T-1} \Big\{W^{(T)}(\omega -u_s) \Big\}^2 \times (\omega -u_s)^k = \begin{cases} C_k B_T^{k-1}+ O\big( T^{-1}B_T^{-2}\big)  \qquad &k \quad {\rm even}\\
O\big(T^{-1} B_T^{-2} \big) \qquad &k \quad {\rm odd}\\
\end{cases} \\
& (D) \qquad \frac{1}{T} \sum_{s=0}^{T-1}\Big| W^{(T)}(\omega -u_s) W^{(T)}(\omega + u_s) \times (\omega - u_s)^j\Big| = 1_{I_T}(\omega) O\big(B_T^{j-1} \big),
\end{align*}
where $\delta_{0j}$ is Kronecker symbol ,  $0\leq C_k \leq \int_{\mathbb{R}} W(\alpha)^2 \alpha^k d\alpha$ and $I_T = \big[0,B_T\big] \cup \big[\pi - B_T, \pi + B_T\big] \cup \big[2\pi - B_T, 2\pi\big]$.
\end{lemma}
\begin{proof} [Proof]
The proof uses the results of \cite{panaretos2013fourier} on the total variation $V_a^b(h)$ of a function $h:[a,b] \rightarrow \mathbb{C}$.  We first  have the following results. For any positive integers  $\ell$ and $k$ and $x \in [-\pi  \leq \pi]$,
\begin{align*}
V_{0}^{2\pi}\Big( \big\{W^{(T)}\big\}^{\ell} x^k  \Big) \leq 4 \big\|W^{(T)}\big\|_{\infty} \times (2\pi)^k \times V_0^{2\pi}\Big(\big\{W^{(T)}\big\}^{\ell} \Big) = O\big(B_T^{-\ell} \big) .
\end{align*}
In the rest of the proof of this lemma, we frequently use Lemma 7.12 in \cite{panaretos2013fourier} to get an upper bound for the difference between an integral and its linear approximation based on grid of points in that interval.

\noindent $(A)$ We have
\begin{align*}
\frac{1}{T} \sum_{s=0}^{T-1} W^{(T)}(\omega -u_s) \times (\omega -u_s)^j 
&= \int_{-\pi}^{\pi} W^{(T)}(\alpha) \alpha^j d\alpha + O\Big(T^{-1}B_T^{-1} \Big)\\
& = \delta_{0j} + O\Big(T^{-1}B_T^{-1} \Big).
\end{align*}
In the second line, we replaced $\alpha = \omega - v$, then $\alpha$ from $-\pi$ to $\pi$. The integral equals zero,  by the properties of $W$. This proves part $(A)$. 

\noindent $(B)$ For this part,
\begin{align*}
\frac{1}{T} \sum_{s=0}^{T-1} \Big|W^{(T)}(\omega -u_s) \Big| \times \Big|\omega - u_s \Big|^{p} 
&= \int_{-\pi}^{\pi} \Big|W^{(T)}(\alpha) \Big| \times \big|\alpha \big|^{p} d\alpha + O\big(T^{-1} B_T^{-1} \big)
\leq \int_{\mathbb{R}} \Big| W(\alpha)\Big| \times \big|\alpha \big|^p d\alpha \times B_T^{p}  +  O\big(T^{-1} B_T^{-1} \big)\\
&= O\big(B_T^p \big)  +  O\big(T^{-1} B_T^{-1} \big).
\end{align*}
\noindent $(C)$ We have
\begin{align*}
\frac{1}{T} \sum_{s=0}^{T-1} \Big\{W^{(T)}(\omega -u_s)\Big\}^2 \times (\omega -u_s)^k  = \int_{-\pi}^{\pi} \Big\{W^{(T)}(\alpha)\Big\}^2  \alpha^k d\alpha  + \frac{1}{T} O\left( V_{0}^{2\pi}\Big( \big\{W^{(T)}\big\}^2 x^k  \Big)  \right).
\end{align*}
The second term is bounded by  $O\big(T^{-1}B_T^{-2}\big)$, for odd $k$, while the first term vanishes.  We now consider the integral term for even $k$.
Recall that  $$W^{(T)}(\alpha) = \frac{1}{B_T}\sum_{i\in \mathbb{Z}}   W\big(\frac{\alpha + 2\pi i}{B_T} \big),$$ 
and so
\begin{align*}
\big\{W^{(T)}(\alpha) \big\}^2= \Big\{\frac{1}{B_T}\sum_i W\Big(\frac{\alpha + 2\pi i }{B_T} \Big) \Big\} \times \Big\{ \frac{1}{B_T} \sum_j W\Big(\frac{\alpha + 2\pi j }{B_T} \Big)\Big\} .
\end{align*}
For  $i\neq j $, $\pi > 2B_T$, and $W$ is supported on $[-1,1]$. It follows that at least one of $W\Big(\frac{\alpha + 2\pi i }{B_T} \Big)$ and $W\Big(\frac{\alpha + 2\pi j }{B_T} \Big)$ must be zero. Thus,
\begin{align*}
\big\{W^{(T)}(\alpha) \big\}^2 = \frac{1}{B_T^2} \sum_iW^2\Big(\frac{\alpha + 2\pi i}{B_T} \Big) .
\end{align*}
Moreover $\alpha \in [- \pi,\pi]$, so for $|i| \geq 1$ we have $W\Big(\frac{\alpha + 2\pi i}{B_T}\Big) =0$. Hence, for even $k$,
\begin{align*}
\int_{-\pi}^{\pi}W^{(T)}\big(\alpha \big)^2 \alpha^k d\alpha  = \frac{1}{B_T^2}\int_{-\pi}^{\pi} W\Big(\frac{\alpha}{B_T} \Big)^2 \alpha^k d\alpha \leq B_T^{k-1} \int_{\mathbb{R}}W(\alpha)^2 \alpha^k d\alpha .
\end{align*}

\noindent $(D)$ For $\omega \in I_T$ 
\begin{align*}
\frac{1}{T} \sum_{s=0}^{T-1} \Big|W^{(T)}(\omega -u_s) W^{(T)}(\omega + u_s) \times (\omega - u_s)^j  \Big|
&= \int_{-\pi}^{\pi}\Big| W^{(T)}(\alpha) W^{(T)}(2\omega-\alpha) \times \alpha^j \Big| d\alpha + O\big(T^{-1} B_T^{-2} \big).
\end{align*}
Since $W$ is supported on $[-1,1]$,  $W\Big( \frac{ 2\omega -\alpha  + 2k\pi}{B_T} \Big) \neq 0 $ iff  $|2\omega - \alpha + 2k \pi| \leq B_T$. For $T$ sufficiently large, $\alpha \in [-\pi,\pi], \omega \in [0,2\pi]$, the inequality will hold only for $-4 \leq k \leq 4 $. Thus, the integral is bounded by
\begin{align*}
\sum_{k=-4}^4 \int_{-\pi}^{\pi} \frac{1}{B_T^2} \Big|W\Big(\frac{\alpha}{B_T} \Big) \times W\Big(\frac{2\omega -\alpha  + 2k \pi}{B_T} \Big) \times \alpha^j \Big|d\alpha  &= \sum_{k=-4}^4\int_{-\pi B_T^{-1}}^{\pi B_T^{-1}} \frac{1}{B_T^2} \Big|W(x)  W\Big( \frac{2\omega - 2k \pi}{B_T} - x \Big) \times  B_T^{j+1} x^j \Big|dx \\
&= O\Big(B_T^{j-1} \Big).
\end{align*}
 When $\omega \notin I_T$, since $0 < \omega < 2\pi$, we have  $\omega \notin \cup_{k\in \mathbb{Z}}\big[k\pi - B_T, k\pi+ B_T\big]$, and so 
$\big|k\pi - \omega \big| > B_T$. It follows that  
$$\Big|\frac{\alpha}{B_T}  + \frac{2\omega -\alpha + 2k \pi }{B_T} \Big|  =  \Big|\frac{2k \pi + 2\omega}{B_T} \Big| \geq 2.$$
Then, at least one of  $W\Big(\frac{\alpha}{B_T}\Big)$ and $W\Big(\frac{2\omega - \alpha + 2k\pi }{B_T}\Big)$ must equal 0.  When $|\alpha| > B_T$, $W(\alpha/B_T)  = 0$. We deduce that for $T$ large enough and
$\omega \notin [0,B_T] \cup [\pi - B_T, \pi + B_T] \cup [2\pi - B_T, 2\pi]$
\begin{align*}
W^{(T)}(\alpha) W^{(T)}(\alpha - 2\omega)  = \frac{1}{B_T^2} W\Big(\frac{\alpha}{B_T}\Big) \sum_{k\in \mathbb{Z}} W\Big( \frac{\alpha + 2k \pi - 2\omega}{B_T}\Big)  = 0.
\end{align*}
Thus we get zero for $\omega \notin I_T$.
\end{proof}
\begin{lemma} \label{SpecDenX} Let $h_T(t) = 1_{[0,T-1]}(t)$ and  $\Delta^{(T)}(\omega) = \sum_{t=0}^{T-1} e^{-\mathbf{i} \omega t}$. Let
\begin{align*}
f_{\omega_1,\ldots,\omega_{k-1}}(\tau_1,\ldots,\tau_k) &= \frac{1}{(2\pi)^{k-1}} \sum_{t_1,\ldots,t_{k-1} = -\infty}^{\infty} \exp \left\{ -\mathbf{i} \sum_{j=1}^{k-1} \omega_j t_j\right\} {\cum} \big(X_{t_1}(\tau_1),\ldots, X_{t_{k-1}}(\tau_{k-1}), X_0(\tau_k) \big) \\
\varrho_{T,k}: = \varrho_T(\tau_1,\ldots,\tau_k) &= \frac{1}{(2\pi)^{k-1}}\sum_{t_i\geq T} \exp \left\{-\mathbf{i} \sum_{j=1}^{k-1} \omega_j t_j \right\} {\cum} \big(X_{t_1}(\tau_1),\ldots, X_{t_{k-1}}(\tau_{k-1}), X_0(\tau_k) \big)\\
\rho_{T,k}:= \rho_T(\tau_1,\ldots,\tau_{k}) &= \sum_{|t_j| \leq T-1} \exp \left\{-\sum_{j=1}^{k-1} t_j \omega_j \right\} {\cum} \big(X_{t_1}(\tau_1),\ldots, X_{t_{k-1}}(\tau_{k-1}), X_0(\tau_k) \big) \times \\
&\qquad \qquad \left[\sum_{t=0}^{T-1} \exp \Big\{-t\sum_{j=1}^k \omega_j \Big\}  \Big\{1 - h_T(t+t_1) h_T(t+t_2) \ldots h_T(t+t_{k-1}) h_T(t) \Big\}\right]
\end{align*}
 Then
\begin{align*}
{\cum} \Big(\widetilde{X}^{(T)}_{\omega_1},\ldots, \widetilde{X}^{(T)}_{\omega_k} \Big) = \frac{(2\pi)^{k/2-1}}{T^{k/2}} f_{\omega_1,\ldots,\omega_{k-1}} \times \Delta^{(T)} \Big(\sum_{j=1}^k \omega_j \Big) -  \frac{(2\pi)^{k/2-1}}{T^{k/2}} \Delta^{(T)} \Big(\sum_{j=1}^k \omega_j \Big)\times \varrho_{T,k} + \frac{1}{(2\pi T)^{k/2}}\rho_{T,k}.
\end{align*} 
Let $\F_{\omega_1,\ldots,\omega_{k-1}}, \U_{T,k}$ and  $\V_{T,k}$ be the operators induced by $f_{\omega_1,\ldots,\omega_{k-1}}, \varrho_{T,k}$ and $\rho_{T,k}$, respectively. Then
\begin{align*}
\hs \F_{\omega_1,\ldots,\omega_{k-1} }\hs_1  &\leq \frac{1}{(2\pi)^{k-1}} \sum_{t_1,\ldots,t_{k-1} = -\infty}^{\infty}  \hs \mathscr{R}_{t_1,\ldots,t_{k-1}} \hs_1 \\
\hs \U_{T,k} \hs_1 &\leq  \sum_{|t_j| \geq T} \hs \mathscr{R}_{t_1,\ldots,t_{k-1}} \hs_1 \\
\hs \V_{T,k} \hs_1 &\leq \sum_{|t_j| \leq T-1}  \hs \mathscr{R}_{t_1,\ldots,t_{k-1}} \hs_1 \times  \sum_{j=1}^{k-1} \big|t_j \big|.
\end{align*}
\end{lemma}
\begin{proof}[Proof]
Proofs of these results can be found in \cite{panaretos2013fourier}. \end{proof}

\begin{proposition} \label{2.6} Assume assumptions (A1)-(A3) and (B1)-(B6) in Section 4 and 7 are satisfied, 
 then 
\begin{align*}
\mathbb{E} p_{\nu_s}^{(T)}(\tau,\sigma) = f_{\nu_s}^{XX}(\tau,\sigma) + \frac{1}{T} \vartheta_{\nu_s}(\tau,\sigma). 
\end{align*}
For $\nu_r + \nu_s \neq 2\pi $,
\begin{align*}
{\cum} \Big(\widetilde{X}_{\nu_s}(\tau), \widetilde{X}_{\nu_r}(\sigma) \Big) = \frac{1}{ T} \vartheta_{\nu_s,\nu_r}(\tau,\sigma).
\end{align*}
For the corresponding integral operator, we have $\mathscr{P}_{\nu_s,T}^{XX} = \F_{\nu_s}^{XX} +T^{-1}  \V_{\nu_s}$ and $\hs \V_{\nu_s} \hs_1  < C$ and
$\mathbb{E} \widetilde{X}_{\nu_s} \otimes \widetilde{X}_{\nu_r} = T^{-1} \V_{\nu_s,\nu_r}$ with $\hs \V_{\nu_s,\nu_r}\hs_1 <C$ uniformly over $s,r$, where $C$ is an universal constant.
\end{proposition}
\begin{proof}[Proof]
The results follow by Proposition 2.6 in \cite{panaretos2013fourier} and by  Lemma \ref{SpecDenX}.
\end{proof}
\begin{proposition} \label{3.1}Assume assumptions (A1)-(A3) and (B1)-(B6) in Section 4 and 7 are satisfied. Then,  
\begin{align}
\mathbb{E} f^{(T)}_{\omega}(\tau,\sigma) = f_{\omega}^{XX}(\tau,\sigma) + \frac{1}{B_T T}\vartheta_{\omega}(\tau,\sigma)  
 \label{proposition3.1} .
\end{align}
 For the induced operators, $\mathbb{E}\widehat{\F}_{\omega}^{XX} = \F_{\omega}^{XX} + B_T^{-1}T^{-1} \V_{\omega}$ with $\hs \V_{\omega} \hs_{1} <C $ for an universal constant $C$.
\end{proposition}
\begin{proof} The first statement is equivalent to proposition 3.1 in \cite{panaretos2013fourier}. Now, by definition,
\begin{align*}
f_{\omega}^{(T)} = \frac{2\pi}{T} \sum_{s=1}^T W^{(T)}(\omega - \nu_s) p^{(T)}_{\nu_s}.
\end{align*}
Using Proposition \ref{2.6},
\begin{align*}
\mathbb{E} f_{\omega}^{(T)}(\tau,\sigma) & = \frac{2\pi}{T}\sum_{s=0}^{T-1} W^{(T)}(\omega - \nu_s) f_{\nu_s}^{XX}(\tau,\sigma) + \frac{2\pi}{T} \sum_{s=0}^{T-1} W^{(T)}(\omega-\nu_s) \frac{1}{T} \vartheta_{\nu_s}(\tau,\sigma) \\
& = \frac{2\pi}{T}\sum_{s=0}^{T-1} W^{(T)}(\omega - \nu_s) f_{\nu_s}^{XX}(\tau,\sigma) + \frac{1}{TB_T} \vartheta_{(1)}(\tau,\sigma) ,
\end{align*}
with the operator $ \V_{(1)}$ induced by $\vartheta_{(1)}$ satisfying $ \hs \V_{(1)}\hs_1 = O(1)$ uniformly over $\omega$. Replacing $\nu_s$ by $u_s$, and using a Taylor expansion, 
\begin{align*}
f_{u_s}^{XX}(\tau,\sigma) = f_{\omega}^{XX}(\tau,\sigma) + \sum_{j=1}^{p-1} \frac{\big(u_s -\omega \big)^j}{j!} \frac{\partial^j  f_{\alpha}^{XX}(\tau,\sigma)}{\partial \alpha^j}\Big|_{\alpha = \omega} + \big(u_s - \omega \big)^{p} g_{p,u_s,\omega}(\tau,\sigma).
\end{align*}
Summing over $s$ now gives
\begin{align*}
f_{\omega}^{(T)}(\tau,\sigma)  = \frac{2\pi}{T} \sum_{j=0}^{p-1}\Big\{\sum_{s=0}^{T-1} W^{(T)}(\omega -u_s) \times (u_s - \omega)^j \Big\} \times \frac{\partial^j f_{\alpha}^{XX}(\tau,\sigma)}{\partial \alpha^j}\Big|_{\alpha = \omega} + \\
\frac{2\pi}{T} \sum_{s=0}^{T-1} W^{(T)}(\omega-u_s)  \times \big(\omega-u_s \big)^p g_{p,u_s,\omega}(\tau,\sigma).
\end{align*}
For $0\leq j \leq p-1$, using results in Lemma \ref{SumW},
\begin{align*}
 \frac{2\pi}{T}\Big\{\sum_{s=0}^{T-1} W^{(T)}(\omega -u_s) \times (\omega -u_s)^j \Big\} \times \frac{\partial^j f_{\alpha}^{XX}(\tau,\sigma)}{\partial \alpha^j}\Big|_{\alpha = \omega}  &= \Big\{\delta_{0j} + O\big(T^{-1} B_T^{-1} \big) \Big\} \times \frac{\partial^j f_{\alpha}^{XX}(\tau,\sigma)}{\partial \alpha^j}\Big|_{\alpha = \omega}.
\end{align*}
Taking the sum over $j$ then yields
\begin{align*}
 \frac{2\pi}{T} \sum_{j=0}^{p-1}\Big\{\sum_{s=0}^{T-1} W^{(T)}(\omega -u_s) \times (u_s - \omega)^j \Big\} \times f_{\omega}^{XX,(j)}(\tau,\sigma) &= f_{\omega}^{XX}  + O\big(T^{-1} B_T^{-1} \big) \sum_{j=1}^{p-1} f_{\omega}^{XX,(j)}(\tau,\sigma).
\end{align*}
Finally, letting $\mathscr{G}_{p,u_s,\omega}$  be the operator induced by $g_{p,u_s,\omega}$, we have
\begin{align*}
 \hs \frac{2\pi}{T} \sum_{s=0}^{T-1} W^{(T)}(\omega-u_s)  \times \big(\omega-u_s \big)^p \mathscr{G}_{p,u_s,\omega} \hs_{1}
 &\leq \frac{2\pi}{T}\sum_{s=0}^{T-1} \Big| W^{(T)}(\omega-u_s)  \times \big(\omega-u_s \big)^p\Big|  \times \sup_{s} \hs\mathscr{G}_{p,u_s,\omega} \hs_{1} \\
&\leq O\big( B_T^p + T^{-1} B_T^{-1}\big)  = O(T^{-1} B_T^{-1}),
\end{align*}
 by Lemma \ref{SumW} and Lemma \ref{SchNo12}. Combining the above results completes the proof.
\end{proof}
\begin{proposition}\label{4cumS1}
(A) Let 
\begin{align*}
p_{r,s}^{(T)} = {\rm cum } \big(\widetilde{X}_{\nu_s}^{(T)}(\tau_1), \widetilde{X}_{-\nu_s}^{(T)}(\sigma_1), \widetilde{X}_{-\nu_r}^{(T)}(\tau_2), \widetilde{X}_{\nu_r}^{(T)}(\sigma_2) \big)
\end{align*}
and $\mathscr{P}_{r,s}^{(T)}$ be its induced operator. Then,
\begin{align*}
\hs \mathscr{P}_{r,s}^{(T)} \hs_1  = O(T^{-1}).
\end{align*}
(B) Let $p_{r,s} \in L^2([0,1]^4,\mathbb{C})$ such that its associated operator $\mathscr{P}_{r,s}$ satisfies $\hs \mathscr{P}_{r,s} \hs_1 < CT^{-1}$ for all $r,s =0,\ldots, T-1$ and a universal constant $C$. Let
\begin{align*}
 \sum_{r,s=0}^{T-1} W^{(T)}(\omega - \nu_r) W^{(T)}(\omega - \nu_s) p_{r,s}  = p_{\omega}.
\end{align*}
and $\mathscr{P}_{\omega}$ be the induced operator of $p_{\omega}$. Then
\begin{align*}
\int_{0}^{2\pi} \hs \mathscr{P}_{\omega}\hs_1 d\omega  = O(T).
\end{align*}
\end{proposition}

\begin{proof}
\noindent (A) Using Lemma \ref{SpecDenX}
\begin{align*}
\hs \mathscr{P}_{r,s}^{(T)} \hs_1 &\leq \frac{1}{T}\sum_{t_1,t_2,t_3 = -\infty}^{\infty} \hs \mathscr{R}_{t_1,t_2,t_3} \hs_1 + \frac{2\pi}{T}\sum_{|t_j| \geq T} \hs \mathscr{R}_{t_1,t_2,t_3}\hs_1  + \frac{1}{(2\pi T)^2} \sum_{|t_j| < T-1} \hs \mathscr{R}_{t_1,t_2,t_3}\hs_1 \times \sum_{j=1}^3 |t_j| \\
&\leq \frac{O(1)}{T} \sum_{t_1,t_2,t_3  = -\infty}^{\infty} \hs \mathscr{R}_{t_1,t_2,t_3} \hs_1  = O(T^{-1}).
\end{align*}
\noindent (B)
Recall that
\begin{align*}
W^{(T)}(\omega -\nu_s) = \frac{1}{B_T} \sum_{i\in \mathbb{Z}} W\Big(\frac{\omega - \nu_s + 2i\pi}{B_T} \Big) ;\quad W^{(T)}(\omega -\nu_r) = \frac{1}{B_T} \sum_{j\in \mathbb{Z}} W\Big(\frac{\omega - \nu_r + 2j\pi}{B_T} \Big).
\end{align*}
Note now that $W(x) = 0$ for all $|x| > 1$. Hence, if $|x-y| > 2$, then at least one of $W(x)$ and $W(y)$ vanishes.  We  have $-2\pi \leq \nu_r - \nu_s \leq 2\pi$, and so if 
$$\nu_r -\nu_s  \notin S_T:=  [-2\pi, -2\pi + 2B_T] \cup [-2B_T,2B_T] \cup [2\pi - 2B_T, 2\pi],$$
 then $\big| \nu_r - \nu_s  + 2 k \pi \big| > 2B_T$ for any integer $k$. Thus, if $\nu_r - \nu_s \notin S_T$, for all $i,j \in \mathbb{Z}$
\begin{align*}
\left|  \frac{\omega - \nu_s + 2 i \pi}{B_T} - \frac{\omega - \nu_r + 2j \pi}{B_T} \right |  = \left|  \frac{\nu_r- \nu_s + 2(i - j) \pi}{B_T} \right | > 2.
 \end{align*}
This means that  $ W\Big(\frac{\omega - \nu_s + 2j_s \pi}{B_T} \Big)W\Big(\frac{\omega - \nu_r + 2j_r \pi}{B_T} \Big) = 0$ for all $i,j \in \mathbb{Z}$, if $\nu_r - \nu_s \notin S_T$. Write  
\begin{align*}
S_{T,s} = [-2\pi + \nu_s, -2\pi + 2B_T + \nu_s] \cup [-2B_T+\nu_s,2B_T + \nu_s] \cup [2\pi - 2B_T + \nu_s, 2\pi +\nu_s].
\end{align*}
Then $W^{(T)}(\omega - \nu_s) W^{(T)}(\omega -\nu_r) = 0$ for $\nu_r \notin S_{T,s}$.  The number of $r$ such that $\nu_r \in S_{T,s}$ is of order $TB_T$. Therefore,
\begin{align*}
\sum_{s,r=0}^{T-1} W^{(T)}(\omega - \nu_s) W^{(T)}(\omega -\nu_r)  p_{r,s} = \sum_{s=0}^{T-1} \sum_{r \in S_{T,s}} W^{(T)}(\omega -\nu_s) W^{(T)}(\omega -\nu_r) p_{r,s}\\
 = \sum_{s=0}^{T-1} W^{(T)}(\omega -\nu_s) \times \sum_{r\in S_{T,s}} W^{(T)}(\omega -\nu_r)  p_{r,s}.
\end{align*}
For fixed $\nu_s$, let
\begin{align*}
I_{T,s} = \Big\{\omega :  0 \leq \omega \leq 2\pi, \quad W\big(\frac{\omega - \nu_s + 2i\pi }{B_T} \big) \neq 0, \quad \text{ for some } \quad i \in \mathbb{Z} \Big\}.
\end{align*}
This means that  if $\omega \in I_{T,s}$, then
\begin{align*}
-B_T \leq \omega -\nu_s + 2i \pi \leq B_T  \iff \omega \in  [-B_T + \nu_s - 2i \pi, B_T + \nu_s - 2i \pi]
\end{align*}
for some $i \in \mathbb{Z}$. The length of $[-B_T + \nu_s - 2i \pi, B_T + \nu_s - 2i \pi]$ is $2B_T$. For  $|i| \geq 4$, $[-B_T + \nu_s - 2i \pi, B_T + \nu_s - 2i \pi] \cap [0,2\pi]  = \emptyset $. Hence, the length of $I_{T,s}$ is of order $O(B_T)$. By the definition of $I_{T,s}$, 
$$\big| W^{(T)}(\omega - \nu_s) \big| \leq 1_{I_{T,s}}(\omega) \frac{\|W\|_{\infty}}{B_T}.$$
 
The number of $r$ such that $\nu_r \in S_{T,s}$ is of order $TB_T$.  Thus, combining our results
\begin{align*}
\hs \sum_{s,r=0}^{T-1}W^{(T)}(\omega - \nu_s) W^{(T)}(\omega -\nu_r) \mathscr{P}_{r,s}  \hs_1 &\leq \sum_{s=0}^{T-1} 1_{I_{T,s}}(\omega) \frac{ \|W\|_{\infty}}{B_T} O(TB_T) \frac{\|W\|_{\infty}}{B_T} \sup_{r,s} \hs \mathscr{P}_{r,s} \hs_1\\
&= \sum_{s=0}^{T-1} 1_{I_{T,s}}(\omega) \frac{O(1)}{B_T}. 
\end{align*}
Integrating over $\omega$ and remarking that $I_{T,s}$ is of order $B_T$, we obtain
 \begin{align*}
 \int_{0}^{2\pi} \hs \mathscr{P}_{\omega}\hs_1 d\omega = O(T).
 \end{align*}
 \end{proof}

\begin{proposition}\label{CovPerX} Assume assumptions (A1)-(A3) and (B1)-(B6) in Section 4 and 7 are satisfied, then
\begin{align*}
\mathbb{E} \Big[p_{\nu_s}^{(T)}(\tau_1,\sigma_1) \times p_{-\nu_r}^{(T)}(\tau_2,\sigma_2) \Big] = \mathbb{E} \Big[p_{\nu_s}^{(T)}(\tau_1,\sigma_1) \Big] \times \mathbb{E} \Big[ p_{-\nu_r}^{(T)}(\tau_2,\sigma_2)\Big] +  p_{r,s}^{(T)}(\tau_1,\sigma_1,\tau_2,\sigma_2) + \\
\eta(\nu_r - \nu_s) f_{\nu_s}^{XX} (\tau_1,\tau_2) f_{-\nu_s}^{XX}(\sigma_1,\sigma_2)+ \frac{1}{T} \eta(\nu_s- \nu_r) \vartheta_{1,\nu_s,\nu_r,f}(\tau_1,\tau_2) \odot \vartheta_{2,\nu_s,\nu_r,f}(\sigma_1,\sigma_2)  + \\
 \eta(\nu_s +\nu_r) f_{\nu_s}^{XX}(\tau_1,\sigma_2) f_{-\nu_s}^{XX}(\sigma_1,\tau_2) + \frac{1}{T}\eta(\nu_s+ \nu_r)  \vartheta_{3,\nu_s,\nu_r,f}(\sigma_1,\tau_2) \odot \vartheta_{4,\nu_s,\nu_r,f}(\tau_1,\sigma_2) + \\
\frac{1}{T^2}\vartheta_{1,\nu_s,\nu_r}(\tau_1,\sigma_1) \times \vartheta_{2,\nu_s,\nu_r}(\tau_2,\sigma_2) + \frac{1}{T^2}\vartheta_{3,\nu_s,\nu_r}(\tau_1,\sigma_2) \times \vartheta_{4,\nu_s,\nu_r}(\sigma_1,\tau_2) ,
\end{align*}
where  $\eta(x)$ equals one if $x\in 2\pi \mathbb{Z}$ and zero otherwise, and  $\vartheta_{i,\nu_s,\nu_r,f} \odot \vartheta_{j,\nu_s,\nu_r,f} \in Conv_C\big(L^2([0,1]^2,\mathbb{C}) \times  L^2([0,1]^2,\mathbb{C})\big)$  with a universal constant  $C$.
\end{proposition}
\begin{proof}[Proof]
To simplify notation, let 
\begin{align*}
A = \widetilde{X}_{\nu_s}^{(T)}(\tau_1) ; \qquad B = \widetilde{X}_{-\nu_s}^{(T)}(\sigma_1); \qquad  C = \widetilde{X}_{-\nu_r}^{(T)}(\tau_2); \qquad D = \widetilde{X}_{\nu_r}^{(T)}(\sigma_2).
\end{align*}
We use the formula
\begin{align*}
\mathbb{E}\big[ABCD \big] = \mathbb{E}\big[AB \big] \times \mathbb{E}\big[CD \big] + \mathbb{E}\big[AC \big] \times \mathbb{E}\big[BD \big] + \mathbb{E}\big[AD \big] \times \mathbb{E}\big[BC \big] + {\cum}\big(A,B,C,D\big).
\end{align*}
The term ${\cum}(A,B,C,D)$ will be denoted by  $p^{(T)}_{r,s}$. Applying Proposition \ref{3.1},
\begin{align*}
\mathbb{E}\big[AB\big] \times \mathbb{E}\big[CD\big] =&\mathbb{E} \Big[p_{\nu_s}^{(T)}(\tau_1,\sigma_1) \Big] \times \mathbb{E} \Big[ p_{-\nu_r}^{(T)}(\tau_2,\sigma_2)\Big]\\
\mathbb{E}\big[AC\big] \times \mathbb{E}\big[BD\big] = & \left\{ \eta(\nu_s - \nu_r)f_{\nu_s}^{XX}(\tau_1,\tau_2) + \frac{1}{T} \vartheta_{1,\nu_s,\nu_r}(\tau_1,\tau_2)\right\}  
  \times\left\{ \eta(\nu_s - \nu_r)f_{-\nu_s}^{XX}(\sigma_1,\sigma_2)  + \frac{1}{T} \vartheta_{2,\nu_s,\nu_r}(\sigma_1,\sigma_2)\right\} \\
= & \eta(\nu_s - \nu_r) f_{\nu_s}^{XX} (\tau_1,\tau_2) f_{-\nu_s}^{XX}(\sigma_1,\sigma_2)+  \frac{1}{T} \eta(\nu_s - \nu_r)\Big\{ f_{\nu_s}^{XX}(\tau_1,\tau_2)\vartheta_{2,\nu_s,\nu_r}(\sigma_1,\sigma_2) + \\ & 
 f_{-\nu_s}^{XX}(\sigma_1,\sigma_2) \vartheta_{1,\nu_s,\nu_r}(\tau_1,\tau_2) \Big\}+  \frac{1}{T^2} \vartheta_{1,\nu_s,\nu_r}(\tau_1,\tau_2) \times \vartheta_{2,\nu_s,\nu_r}(\sigma_1,\sigma_2) \\
=&  \eta(\nu_s  - \nu_r) f_{\nu_s}^{XX} (\tau_1,\tau_2) f_{-\nu_s}^{XX}(\sigma_1,\sigma_2)+  \frac{1}{T} \eta(\nu_s - \nu_r) \vartheta_{1,\nu_s,\nu_r,f}(\tau_1,\tau_2)\odot \vartheta_{2,\nu_s,\nu_r,f}(\sigma_1,\sigma_2) + \\ &   
 +  \frac{1}{T^2} \vartheta_{1,\nu_s,\nu_r}(\tau_1,\tau_2) \times \vartheta_{2,\nu_s,\nu_r}(\sigma_1,\sigma_2);\\
\mathbb{E}\big[AD\big] \times \mathbb{E} \big[BC\big] = & \left\{ \eta(\nu_s + \nu_r) f_{\nu_s}^{XX}(\tau_1,\sigma_2) + \frac{1}{T} \vartheta_{3,\nu_s,\nu_r}(\tau_1,\sigma_2)\right\} \ \times  \left\{ \eta(\nu_s + \nu_r) f_{-\nu_s}^{XX}(\sigma_1,\tau_2) +  \frac{1}{ T} \vartheta_{4,\nu_s,\nu_r}(\sigma_1,\tau_2)\right\}\\
=& \eta(\nu_s + \nu_r) f_{\nu_s}^{XX}(\tau_1,\sigma_2) f_{-\nu_s}^{XX}(\sigma_1,\tau_2) + \frac{1}{T}\eta(\nu_s + \nu_r) \Big\{  \vartheta_{4,\nu_s,\nu_r}(\sigma_1,\tau_2) f_{\nu_s}^{XX}(\tau_1,\sigma_2) + \\
& f_{-\nu_s}^{XX}(\sigma_1,\tau_2) \vartheta_{3,\nu_s,\nu_r}(\tau_1,\sigma_2)\Big\}  +  \frac{1}{T^2} \vartheta_{3,\nu_s,\nu_r}(\tau_1,\sigma_2) \times \vartheta_{4,\nu_s,\nu_r}(\sigma_1,\tau_2) \\
=& \eta(\nu_s + \nu_r) f_{\nu_s}^{XX}(\tau_1,\sigma_2) f_{-\nu_s}^{XX}(\sigma_1,\tau_2) + \frac{1}{T}\eta(\nu_s + \nu_r)  \vartheta_{3,\nu_s,\nu_r,f}(\sigma_1,\tau_2) \odot \vartheta_{4,\nu_s,\nu_r,f} (\tau_1,\sigma_2)  \\
& +  \frac{1}{T^2} \vartheta_{3,\nu_s,\nu_r}(\tau_1,\sigma_2) \times \vartheta_{4,\nu_s,\nu_r}(\sigma_1,\tau_2).
\end{align*}
Combining these results completes the proof.
\end{proof}

\begin{proposition}\label{CovDenX} Assume assumptions (A1)-(A3) and (B1)-(B6) in Section 4 and 7 are satisfied, then 
\begin{align*}
&\mathbb{E} \left[ \Big\{f^{(T)}_{\omega}(\tau_1,\sigma_1) - f_{\omega}(\tau_1,\sigma_1) \Big\} \times \Big\{ \overline{f^{(T)}_{\omega}(\tau_2,\sigma_2)} - f_{ -\omega}(\tau_2,\sigma_2) \Big\}\right] = \\
&\qquad O\big(T^{-1} B_T^{-1} \big) \times \Big\{  f_{\omega}^{XX}(\tau_1,\tau_2) f_{-\omega}^{XX}(\sigma_1,\sigma_2)  + 1_{I_T}(\omega)f_{\omega}^{XX}(\tau_1,\sigma_2) f_{-\omega}^{XX}(\tau_2,\sigma_1)  \Big\}\\
&\qquad + O\big(T^{-1} B_T \big)\Big\{\vartheta_1(\tau_1,\tau_2) \odot  \vartheta_2(\sigma_1,\sigma_2) + 1_{I_T}(\omega)\vartheta_3(\tau_1,\sigma_2) \odot \vartheta_4(\sigma_1,\tau_2)\Big\} \\
&\qquad + 1_{I_T}(\omega) \times O\big(T^{-1}  \big) \times \Big\{f_{\omega}^{XX}(\tau_1,\sigma_2)f_{-\omega}^{XX,(1)}(\tau_2,\sigma_1) + f_{-\omega}^{XX}(\tau_2,\sigma_1) f_{\omega}^{XX,(1)}(\tau_1,\sigma_2)  \Big\}  \\
&\qquad +\frac{1}{T^2}  \sum_{s,r=0}^{T-1} W^{(T)}(\omega -\nu_s) W^{(T)}(\omega -\nu_r) p_{r,s}^{(T)}(\tau_1,\sigma_1,\tau_2,\sigma_2),
\end{align*}
where $I_T$ is as in Lemma \ref{SumW} and $\vartheta_i \odot \vartheta_j \in Conv_C\Big(L^2([0,1]^2,\mathbb{C}) \times L^2([0,1]^2,\mathbb{C}) \Big)$.
\end{proposition}

\begin{proof}We use the same notation $A,B,C,D$ as in the proof of Proposition \ref{CovPerX}. By definition of $f_{\omega}^{(T)}$,
\begin{align*}
\mathbb{E} \left[ f^{(T)}_{\omega}(\tau_1,\sigma_1)  \times \overline{f^{(T)}_{\omega}(\tau_2,\sigma_2)} \right] = \Big(\frac{2\pi}{ T}\Big)^2 \sum_{s,r=0}^{T-1} W^{(T)}(\omega -\nu_s) W^{(T)}(\omega -\nu_r) \times \mathbb{E} \Big[p_{\nu_s}^{(T)}(\tau_1,\tau_2)\times p_{-\nu_r}^{(T)}(\sigma_1,\sigma_2) \Big] .
\end{align*}
We use Proposition \ref{CovPerX} to decompose
 $\mathbb{E} \Big[p_{\nu_s}^{(T)}(\tau_1,\tau_2)\times p_{-\nu_r}^{(T)}(\sigma_1,\sigma_2) \Big]$ and treat each part separately.  Consider first $\mathbb{E}[AB]\times \mathbb{E}[CD]$, given by
 \begin{align*}
& \Big(\frac{2\pi}{ T}\Big)^2 \sum_{s,r=0}^{T-1} W^{(T)}(\omega -\nu_s) W^{(T)}(\omega -\nu_r) \times \mathbb{E}\Big[p_{\nu_s}^{(T)}(\tau_1,\tau_2) \Big] \times \mathbb{E}\Big[p_{-\nu_r}^{(T)}(\sigma_1,\sigma_2) \Big]\\
&= \left\{\frac{2\pi}{T} \sum_{s=0}^{T-1} W^{(T)}(\omega -\nu_s)   \mathbb{E}\Big[p_{\nu_s}^{(T)}(\tau_1,\sigma_1)\Big]  \right\} \times \left\{\frac{2\pi}{T} \sum_{r=0}^{T-1} W^{(T)}(\omega -\nu_r)   \mathbb{E}\Big[p_{-\nu_r}^{(T)}(\tau_2,\sigma_2)\Big] \right\} \\
& = \mathbb{E} f_{\omega}^{(T)}(\tau_1,\sigma_1) \times \mathbb{E} \overline{f_{\omega}^{(T)}(\tau_2,\sigma_2)} .
\end{align*}
Note that $\eta(x) = 0$ when $x \neq 2k\pi$. Next, consider $\mathbb{E}[AC] \times \mathbb{E}[BD]$ which is
\begin{align*}
&\Big(\frac{2\pi}{ T}\Big)^2 \sum_{s,r=0}^{T-1} W^{(T)}(\omega -\nu_s) W^{(T)}(\omega -\nu_r) \times 
\Big[\eta(\nu_s - \nu_r) f_{\nu_s}^{XX} (\tau_1,\tau_2) f_{-\nu_s}^{XX}(\sigma_1,\sigma_2)+ \\
& \frac{1}{T} \eta(\nu_s - \nu_r) \vartheta_{1,\nu_s,\nu_r,f}(\tau_1,\tau_2) \odot\vartheta_{2,\nu_s,\nu_r,f}(\sigma_1,\sigma_2) + \frac{1}{T^2} \vartheta_{1,\nu_s,\nu_r}(\tau_1,\tau_2) \vartheta_{2,\nu_s,\nu_r}(\sigma_1,\sigma_2)\Big] \\
&= \Big(\frac{2\pi}{ T}\Big)^2 \sum_{s=0}^{T-1} W^{(T)}(\omega -\nu_s) W^{(T)}(\omega - \nu_s) \times f_{\nu_s}^{XX} (\tau_1,\tau_2) f_{-\nu_s}^{XX}(\sigma_1,\sigma_2)  + B_T^{-2} T^{-2} \vartheta_{1,f}(\tau_1,\tau_2) \odot \vartheta_{2,f}(\sigma_1,\sigma_2),
\end{align*}
where
\begin{align*}
 &T^{-2} B_T^{-2} \times \vartheta_{1,f}(\tau_1,\tau_2) \odot \vartheta_{2,f}(\sigma_1,\sigma_2) = \\
 & \Big(\frac{2\pi}{ T}\Big)^2 \sum_{s,r=0}^{T-1} W^{(T)}(\omega -\nu_s) W^{(T)}(\omega - \nu_s) \times  \frac{1}{T}  \vartheta_{1,\nu_s,\nu_r,f}(\tau_1,\tau_2) \odot \vartheta_{2,\nu_s,\nu_r,f}(\sigma_1,\sigma_2)\\
 &+ \Big(\frac{2\pi}{ T}\Big)^2 \sum_{s,r=0}^{T-1} W^{(T)}(\omega -\nu_s) W^{(T)}(\omega -\nu_r) \times \frac{1}{T^2}  \vartheta_{1,\nu_s,\nu_r}(\tau_1,\tau_2)  \vartheta_{2,\nu_s,\nu_r}(\sigma_1,\sigma_2).
\end{align*}
For the term containing $f_{\nu_s}^{XX} (\tau_1,\tau_2) f_{-\nu_s}^{XX}(\sigma_1,\sigma_2)$, we replace $\nu_s$ by $u_s$ and use a Taylor expansion as in Lemma \ref{SchNo12}
\begin{align*}
f_{u_s}^{XX}(\tau_1,\tau_2) =& f_{\omega}^{XX}(\tau_1,\tau_2)  + \sum_{j=1}^{p-1} \frac{(u_s -\omega)^j}{j!}  \frac{\partial^j f_{\alpha}^{XX}(\tau_1,\tau_2)}{\partial \alpha^j}\Big|_{\alpha = \omega}  + (u_s -\omega)^p g_{2,u_s,\omega}(\tau_1,\tau_2) \\
f_{-u_s}^{XX}(\sigma_1,\sigma_2)=& f_{-\omega}^{XX}(\sigma_1,\sigma_2) + \sum_{j=1}^{p-1} \frac{(\omega-u_s)^j}{j!}  \frac{\partial^j f_{-\alpha}^{XX}(\sigma_1,\sigma_2)}{\partial \alpha^j}\Big|_{\alpha = \omega}  + (\omega -u_s )^p g_{2,u_s,\omega}(\sigma_1,\sigma_2).
\end{align*}
Their product becomes
\begin{align*}
f_{u_s}^{XX}(\tau_1,\tau_2) \times f_{-u_s}^{XX}(\sigma_1,\sigma_2) =& f_{\omega}^{XX}(\tau_1,\tau_2) f_{-\omega}^{XX}(\sigma_1,\sigma_2)  + (\omega -u_s) \Big\{f_{\omega}^{XX}(\tau_1,\tau_2) f^{XX,(1)}_{-\omega}(\sigma_1,\sigma_2) - \\ 
&f_{-\omega}^{XX}(\sigma_1,\sigma_2)f^{XX,(1)}_{\omega}(\tau_1,\tau_2)\Big\}  
+ (\omega -u_s)^2 \times \vartheta_{1,u_s,g}(\tau_1,\tau_2) \odot \vartheta_{2,u_s,g}(\sigma_1,\sigma_2).
\end{align*} 
Taking the sum over $s$ and using Lemma \ref{SumW}, now gives
\begin{align*}
 &\Big(\frac{2\pi}{ T}\Big)^2 \sum_{s=0}^{T-1} W^{(T)}(\omega -u_s) W^{(T)}(\omega - u_s ) \times f_{u_s}^{XX} (\tau_1,\tau_2) f_{-u_s}^{XX}(\sigma_1,\sigma_2)  = \\
 & \Big(\frac{2\pi}{ T}\Big)^2 \sum_{s=0}^{T-1} W^{(T)}(\omega -u_s) W^{(T)}(\omega -u_s ) \times f_{\omega}^{XX}(\tau_1,\tau_2) f_{-\omega}^{XX}(\sigma_1,\sigma_2) +  \\
&\Big(\frac{2\pi}{ T}\Big)^2 \sum_{s=0}^{T-1} W^{(T)}(\omega -u_s) W^{(T)}(\omega -u_s ) \times  (\omega -u_s) \times \Big\{f_{\omega}^{XX}(\tau_1,\tau_2)f^{XX,(1)}_{-\omega}(\sigma_1,\sigma_2)  - f^{XX,(1)}_{\omega}(\tau_1,\tau_2)f_{-\omega}^{XX}(\sigma_1,\sigma_2) \Big\} + \\
&
 \Big(\frac{2\pi}{ T}\Big)^2 \sum_{s=0}^{T-1} W^{(T)}(\omega -u_s) W^{(T)}(\omega -u_s ) \times (\omega -u_s) ^2 \times  \vartheta_{1,\nu_s,g}(\tau_1,\tau_2) \odot \vartheta_{2,\nu_s,g}(\sigma_1,\sigma_2) \\
& = O\big(T^{-1} B_T^{-1} \big) \times f_{\omega}^{XX}(\tau_1,\tau_2) f_{-\omega}^{XX}(\sigma_1,\sigma_2)   + O\big(T^{-1} B_T \big) \times \vartheta_{1,g}(\tau_1,\tau_2) \odot \vartheta_{2,g}(\sigma_1,\sigma_2) .
\end{align*}
Turning to  $\mathbb{E}[AD]\times \mathbb{E}[BC]$, similar manipulations yield
\begin{align*}
&\Big(\frac{2\pi}{ T}\Big)^2 \sum_{s,r=0}^{T-1} W^{(T)}(\omega -\nu_s) W^{(T)}(\omega -\nu_r) \times
 \Big[\eta(\nu_s + \nu_r) f_{\nu_s}^{XX}(\tau_1,\sigma_2) f_{-\nu_s}^{XX}(\sigma_1,\tau_2) + \\
 & \qquad \frac{1}{T}\eta(\nu_s + \nu_r) \vartheta_{3,\nu_s,\nu_r,f}(\tau_1,\sigma_2) \odot \vartheta_{4,\nu_s,\nu_r,f}(\sigma_1,\tau_2) + \frac{1}{T^2} \vartheta_{3,\nu_s,\nu_r}(\tau_1,\sigma_2) \times \vartheta_{4,\nu_s,\nu_r}(\sigma_1,\tau_2) \Big]  \\
&=  \Big(\frac{2\pi}{ T}\Big)^2 \sum_{s=0}^{T-1} W^{(T)}(\omega -\nu_s) W^{(T)}(\omega + \nu_s) \times \left[ f_{\nu_s}^{XX}(\tau_1,\sigma_2) f_{-\nu_s}^{XX}(\sigma_1,\tau_2) \right] + B_T^{-2} T^{-2} \vartheta_{3,f}(\tau_1,\sigma_2) \odot \vartheta_{4,f}(\sigma_1,\tau_2) .
\end{align*}
Again, replacing $\nu_s$ by $u_s$, using a Taylor expansion, and employing Lemma \ref{SumW}, we have
\begin{align*}
\frac{1}{T} \sum_{s=0}^{T-1} \big| W^{(T)}(\omega -u_s) \big| \times \big| W^{(T)}(\omega + u_s)\big| \times \big|\omega - u_s\big|^j & = 1_{I_T}(\omega) O\big(B_T^{j-1} \big).
\end{align*}
Then 
\begin{align*}
 &\Big(\frac{2\pi}{ T}\Big)^2 \sum_{s=0}^{T-1} W^{(T)}(\omega -u_s) W^{(T)}(\omega  + u_s) \times f_{u_s}^{XX} (\tau_1,\sigma_2) f_{-u_s}^{XX}(\sigma_1,\tau_2)  = \\
 &\qquad \Big(\frac{2\pi}{ T}\Big)^2 \sum_{s=0}^{T-1} W^{(T)}(\omega -u_s) W^{(T)}(\omega + u_s ) \times f_{\omega}^{XX}(\tau_1,\sigma_2) f_{-\omega}^{XX}(\sigma_1,\tau_2) +  \\
&\qquad \Big(\frac{2\pi}{ T}\Big)^2 \sum_{s=0}^{T-1} W^{(T)}(\omega -u_s) W^{(T)}(\omega  + u_s ) \times  (\omega -u_s) \times \big\{f_{\omega}^{XX}(\tau_1,\sigma_2)f^{XX,(1)}_{-\omega}(\sigma_1,\tau_2)  - f^{XX,(1)}_{\omega}(\tau_1,\sigma_2)f_{-\omega}^{XX}(\sigma_1,\tau_2) \big\}  + \\
&\qquad \Big(\frac{2\pi}{ T}\Big)^2 \sum_{s=0}^{T-1} W^{(T)}(\omega -u_s) W^{(T)}(\omega + u_s ) \times (\omega -u_s) ^2 \times  \vartheta_{3,\nu_s,g}(\tau_1,\sigma_2) \odot \vartheta_{4,\nu_s,g}(\sigma_1,\tau_2) \\
&\qquad =  1_{I_T}(\omega) O\big(T^{-1} B_T^{-1} \big) \times f_{\omega}^{XX}(\tau_1,\sigma_2) f_{-\omega}^{XX}(\sigma_1,\tau_2) + 1_{I_T}(\omega)  O\big(T^{-1}  \big)   \big\{f_{\omega}^{XX}(\tau_1,\sigma_2) f_{-\omega}^{XX,(1)}(\sigma_1,\tau_2) + \\
&\qquad \qquad 1_{I_T}(\omega)  f_{-\omega }^{XX}(\sigma_1,\tau_2)f_{\omega}^{XX,(1)}(\tau_1,\sigma_2)\big\}  + 1_{I_T}(\omega)  O\big(T^{-1} B_T \big) \times \vartheta_3(\tau_1,\sigma_2) \odot \vartheta_4(\sigma_1,\tau_2).
\end{align*}
Finally, we turn to ${\cum}(A,B,C,D) $, which consists in
\begin{align*}
\frac{1}{T^2} \sum_{r,s=0}^{T-1} W^{(T)}(\omega -\nu_s) W^{(T)}(\omega - \nu_r ) p_{r,s}^{(T)}(\tau_1,\sigma_1,\tau_2,\sigma_2).
\end{align*}
For random variables $U$ and $V$ with $\mathbb{E}U = u, \mathbb{E} v = v$ and constants $a$ and $b$, it holds that
\begin{align*}
 \mathbb{E} \big[ (U -a) \times (V-b) \big]   & = \mathbb{E} [UV] - av - bu + ab = \mathbb{E} \big[ (U-u)(V-v) \big] + (a-u)(b-v) .
\end{align*}
We use this formula with $U = f_{\omega}^{(T)}(\tau_1,\sigma_1)$, $V = f_{-\omega}^{(T)}(\tau_2,\sigma_2)$, $a = f_{\omega}(\tau_1,\sigma_1)$, and $b = f_{-\omega}(\tau_2,\sigma_2)$ to obtain
\begin{align*}
&\mathbb{E}\left[ \Big\{f^{(T)}_{\omega}(\tau_1,\sigma_1) - f_{\omega}^{XX}(\tau_1,\sigma_1) \Big\} \times \Big\{f^{(T)}_{-\omega}(\tau_2,\sigma_2) - f_{-\omega}^{XX}(\tau_2,\sigma_2) \Big\}\right]  = \\
& \qquad \qquad \mathbb{E} \left[ \Big\{ f^{(T)}_{\omega}(\tau_1,\sigma_1) - \mathbb{E} f_{\omega}^{(T)}(\tau_1,\sigma_1)\Big\} \times \Big\{ f^{(T)}_{2\pi -\omega}(\tau_2,\sigma_2) - \mathbb{E} f_{2\pi -\omega}^{(T)}(\tau_2,\sigma_2)\Big\}\right] +\\
&\qquad \qquad \Big\{\mathbb{E} f_{\omega}^{(T)}(\tau_1,\sigma_1) -  f_{\omega}^{XX}(\tau_1,\sigma_1)\Big\} \times \Big\{\mathbb{E} f_{\omega}^{(T)}(\tau_2,\sigma_2) -  f_{\omega}^{XX}(\tau_2,\sigma_2)\Big\}  \\
& = O\big(T^{-1} B_T^{-1} \big) \times \Big\{  f_{\omega}^{XX}(\tau_1,\tau_2) f_{-\omega}^{XX}(\sigma_1,\sigma_2)  + f_{\omega}^{XX}(\tau_1,\sigma_2) f_{-\omega}^{XX}(\tau_2,\sigma_1)  \Big\}+ \\
&\qquad  O\big(T^{-1} B_T \big)\Big\{\vartheta_1(\tau_1,\tau_2) \odot \vartheta_2(\sigma_1,\sigma_2) + \vartheta_3(\tau_1,\sigma_2) \odot \vartheta_4(\sigma_1,\tau_2)\Big\} \\
&\qquad + 1_{I_T}(\omega) \times O\big(T^{-1}  \big) \times \Big\{f_{\omega}^{XX}(\tau_1,\sigma_2)f_{-\omega}^{XX,(1)}(\tau_2,\sigma_1) + f_{-\omega}^{XX}(\tau_2,\sigma_1) f_{\omega}^{XX,(1)}(\tau_1,\sigma_2)  \Big\}  \\
& \qquad + \frac{1}{T^2} \sum_{r,s=0}^{T-1} W^{(T)}(\omega -\nu_s) W^{(T)}(\omega - \nu_r ) p_{r,s}^{(T)}(\tau_1,\sigma_1,\tau_2,\sigma_2).
\end{align*}
\end{proof}

\begin{proposition}\label{FhatXX} Assume assumptions (A1)-(A3) and (B1)-(B6) in Section 4 and 7 are satisfied, then there exists an universal constant $C$ such that
\begin{align*}
\mathbb{E} \hs \widehat{\F}_{\omega,T}^{XX} -\F_{\omega}^{XX} \hs_2^2 \leq C\times T^{-1} B_T^{-1} .
\end{align*} 
\end{proposition}
\begin{proof}[Proof of Proposition \ref{FhatXX}]
By part C of Lemma \ref{SCoij},
\begin{align*}
&\mathbb{E} \hs \widehat{\F}_{\omega,T}^{XX} -\F_{\omega}^{XX} \hs_2^2  = \\
&\sum_{i,j}\mathbb{E}\left[ \Big\{f^{(T)}_{\omega}(\tau_1,\sigma_1) - f_{\omega}^{XX}(\tau_1,\sigma_1) \Big\} \times \Big\{f^{(T)}_{-\omega}(\tau_2,\sigma_2) - f_{-\omega}^{XX}(\tau_2,\sigma_2) \Big\}\right]   \overline{\varphi_i^{\omega}(\tau_1)} \varphi_j^{\omega}(\sigma_1) \varphi_i^{\omega}(\tau_2) \overline{\varphi_j^{\omega}(\sigma_2)} d\tau_1 d\sigma_1 d\tau_2 d\sigma_2.
\end{align*}
We first decompose the  right hand side by Proposition \ref{CovDenX}, then we apply Lemma \ref{VPI} and follow the proof of Proposition \ref{4cumS1} to obtain the upper bound $C T^{-1}B_T^{-1}$.
\end{proof}

\begin{proposition}\label{RidOp} Assume assumptions (A1)-(A3) and (B1)-(B6) in Section 4 and 7 are satisfied, 
the operator $\widehat{\F}_{\omega,T}^{XX} + \zeta_T\mathscr{I}$ is strictly positive definite on an event $G_T$ satisfying $\mathbb{P}[G_T]\stackrel{T\rightarrow\infty}{\longrightarrow}1$.
\end{proposition}
Note that this proposition establishes that even if the kernel function $W$ takes on some negative values, the ridge-estimator $\widehat{\F}_{\omega,T}^{XX} + \zeta_T\mathscr{I}$ will remain positive definite with high probability. Hence, we can find its inverse operator.
\begin{proof}[Proof of Proposition \ref{RidOp}]
 By the result in our last proposition, there exists a constant $C$ that does not depend on $\omega$ such that:
\begin{align*}
\mathbb{E} \hs \widehat{\F}_{\omega,T}^{XX} -\F_{\omega}^{XX} \hs_2^2 \leq C\times T^{-1} B_T^{-1} .
\end{align*}
Let  $\delta$ be a positive number such that $\gamma + 2 \delta < \frac{2\beta - \alpha}{\alpha + 2\beta}$. Define 
\begin{align*}
G_T = \Big\{\theta: \theta \in \Omega ;  \hs\widehat{\F}_{\omega,T}^{XX} -\F_{\omega}^{XX}\hs_2 \leq C^{1/2} T^{-1/2} B_T^{-1/2} T^{\delta}\Big\}.
\end{align*}
Then for $\delta > 0$, $\mathbb{P}(G_T) \rightarrow 1$. 
Let $\widehat{\lambda}^{\omega}_{j,T}$  denote the $j$th eigenvalue of $\widehat{\F}_{\omega,T}^{XX}$. Then,
on the even $G_T$, we have
\begin{align*}
 C^{1/2} T^{-1/2} B_T^{-1/2} T^{\delta} &\geq \hs\widehat{\F}_{\omega,T}^{XX} - \F_{\omega}^{XX} \hs_2 \geq \big|\widehat{\lambda}_{j,T}^{\omega}  - \lambda_j^{\omega} \big| \geq \lambda_j^{\omega} - \widehat{\lambda}_{j,T}^{\omega} \\
\widehat{\lambda}_{j,T}^{\omega} &\geq - C^{1/2} T^{-1/2} B_T^{-1/2} T^{\delta}.
\end{align*}
Since $B_T = T^{-\gamma}$ and 
$\alpha/(\alpha + 2\beta) < 1/2 - \gamma/2 -\delta$, it must be that $\zeta_T > C^{1/2} T^{-1/2} T^{\gamma/2} T^{\delta}  = C^{1/2} T^{-1/2} B_T^{-1/2} T^{\delta} $. It follows that 
\begin{align*}
\zeta_T + \widehat{\lambda}_{j,T}^{\omega} \geq T^{-\alpha/(\alpha + 2\beta)} - C^{1/2} B_T^{-1/2} T^{-1/2} T^{\delta} > 0,
\end{align*}
on $G_T$. 
\end{proof}

\section*{Acknowledgements}
Peter Hall was a mentor to the first author and role model for both authors, and this paper is dedicated to his memory.

\bibliographystyle{imsart-nameyear}
\bibliography{biblio1}
\end{document}